\documentclass[a4paper]{amsart}

\usepackage{tikz-cd}

\usepackage[%
backend=bibtex,%
style=alphabetic,%
natbib=false,%
isbn=false, 
doi=false, 
url=false, 
eprint=true, 
giveninits=true, 
backref=false, 
maxbibnames=99,%
]{biblatex}%

\AtEveryBibitem{%
  \ifentrytype{unpublished}{}{\clearfield{note}}%
  }%
  \renewbibmacro{in:}{} 

  \usepackage{amssymb,amsmath}%
  \usepackage{binarytree}%
  \usepackage{doi}%
  \usepackage{eucal}%
  \usepackage{graphicx}%
  \usepackage{mathscinet}%
  \usepackage{mathtools}%
  \usepackage{tikz-cd}%
  \usepackage{xparse}%

  \NewDocumentCommand{\set}{mo}{\{#1\IfValueT{#2}{\mid #2}\}}

  \newcommand{\hooklongrightarrow}{\lhook\joinrel\longrightarrow}
  
  \newcommand{\Spectra}{\operatorname{Sp}} \newcommand{\Spaces}{\mathcal{S}}

  \newcommand{\op}{{\mathrm{op}}} \newcommand{\rev}{{\mathrm{rev}}}
  \newcommand{\cop}{{\mathrm{cop}}}
  \newcommand{\noloc}{\rotatebox[origin=c]{180}{$\colon$}}
  
  \newcommand{\ZZ}{\mathbb{Z}} \newcommand{\kk}{\mathbf{k}}
  \newcommand{\EE}{\mathbb{E}} \newcommand{\A}{\mathcal{A}}
  \newcommand{\C}{\mathcal{C} }\newcommand{\B}{\mathcal{B}}
  \newcommand{\D}{\mathcal{D}} \newcommand{\F}{\mathcal{F}}
  \newcommand{\G}{\mathcal{G}} 
  \newcommand{\K}{\mathcal{K}} 
  \NewDocumentCommand{\W}{o}{\mathcal{W}\IfValueT{#1}{_{\mathrm{#1}}}}
   \newcommand{\E}{\mathcal{E}}
  \newcommand{\M}{\mathcal{M}} \newcommand{\N}{\mathcal{N}}
  
  \NewDocumentCommand{\Cofib}{}{\mathrm{Cofib}}
  \NewDocumentCommand{\Fib}{}{\mathrm{Fib}}
  \NewDocumentCommand{\Weq}{o}{\mathrm{W}\IfValueT{#1}{_{\mathrm{#1}}}}

  \NewDocumentCommand{\MLoc}{O{\Weq} m}{L_{#1}(#2)}

  \NewDocumentCommand{\unit}{o}{\mathbf{1}\IfValueT{#1}{_{#1}}}
  \NewDocumentCommand{\Lotimes}{s}{\otimes\IfBooleanT{#1}{^{\mathbb{L}}}}

  \NewDocumentCommand{\id}{o}{\mathrm{id}\IfValueT{#1}{_{#1}}}
  
  \usepackage{hyperref}%
  \hypersetup{%
    colorlinks,%
    linkcolor={red!80!black},%
    citecolor={blue!80!black},%
    urlcolor={blue!80!black}%
  }%

  \usepackage{thmtools}%
  \usepackage{cleveref}%

  \numberwithin{equation}{section}%

  \newcounter{thmintro}%

  \declaretheorem[style=theorem,sibling=equation]{corollary}%
  \declaretheorem[style=theorem,sibling=equation]{lemma}%
  \declaretheorem[style=theorem,sibling=equation]{proposition}%
  \declaretheorem[style=theorem,sibling=equation]{theorem}%
  \declaretheorem[style=theorem,sibling=thmintro,name=Theorem,
  Refname={Theorem,Theorems},numbered=no]{theorem intro}%
  \declaretheorem[style=theorem,sibling=thmintro,name=Corollary,
  Refname={Corollary,Corollaries},numbered=no]{corollary intro}%
  \declaretheorem[style=definition,sibling=equation]{definition}%
  \declaretheorem[style=definition,numbered=no,name=Definition]{definition*}%
  \declaretheorem[style=definition,sibling=equation,name={Definition-Proposition}]{defprop}%
  \declaretheorem[style=definition,sibling=equation,Refname={Notation,Notations}]{notation}%
  \declaretheorem[style=definition,sibling=equation]{variant}%
  \declaretheorem[style=definition,sibling=equation]{setting}%
  \declaretheorem[style=remark,sibling=equation]{remark}%
  \declaretheorem[style=remark,sibling=equation]{example}%

  \AddToHook{env/theorem/begin}{     \crefalias{equation}{theorem}}
  \AddToHook{env/example/begin}{     \crefalias{equation}{example}}
  \AddToHook{env/proposition/begin}{    \crefalias{equation}{proposition}}
  \AddToHook{env/lemma/begin}{     \crefalias{equation}{lemma}}
  \AddToHook{env/corollary/begin}{     \crefalias{equation}{corollary}}
  \AddToHook{env/remark/begin}{    \crefalias{equation}{remark}}
  \AddToHook{env/definition/begin}{    \crefalias{equation}{definition}}
  \AddToHook{env/defprop/begin}{    \crefalias{equation}{defprop}}
  \AddToHook{env/notation/begin}{    \crefalias{equation}{notation}}
  \AddToHook{env/variant/begin}{    \crefalias{equation}{variant}}

  \NewDocumentCommand{\Mod}{mo}{\operatorname{Mod}_{#1}\IfValueTF{#2}{(#2)}{}}
  \NewDocumentCommand{\LMod}{mD<>{}od<>}{\operatorname{LMod}_{#1}^{\mathrm{#2}}\IfValueTF{#3}{(#3)}{}\IfValueT{#4}{_{\mathrm{#4}}}}
  \NewDocumentCommand{\LModu}{mD<>{}o}{\underline{\operatorname{LMod}}_{#1}^{\mathrm{#2}}\IfValueTF{#3}{(#3)}{}}
  \NewDocumentCommand{\mmod}{mD<>{}o}{\operatorname{mod}_{#1}^{\mathrm{#2}}\IfValueTF{#3}{(#3)}{}}
  \NewDocumentCommand{\mmodu}{mD<>{}o}{\underline{\operatorname{mod}}_{#1}^{\mathrm{#2}}\IfValueTF{#3}{(#3)}{}}
  \NewDocumentCommand{\perf}{mD<>{}o}{\operatorname{perf}_{#1}^{\mathrm{#2}}\IfValueTF{#3}{(#3)}{}}
  \NewDocumentCommand{\RMod}{mD<>{}o}{\operatorname{RMod}_{#1}^{\mathrm{#2}}\IfValueTF{#3}{(#3)}{}}
  \NewDocumentCommand{\Bimod}{mmD<>{}o}{{}_{#1}\!\operatorname{Bimod}_{#2}^{\mathrm{#3}}\IfValueTF{#4}{(#4)}{}}
  \NewDocumentCommand{\StLMod}{smo}{\operatorname{StMod}_{#2}\IfValueTF{#3}{(#3)}{}}
  \NewDocumentCommand{\stLmod}{smo}{\operatorname{stmod}_{#2}\IfValueTF{#3}{(#3)}{}}

  \NewDocumentCommand{\augAlg}{om}{\operatorname{Alg}^{\mathrm{aug}}\IfValueTF{#1}{_{#1}}{}{(#2)}}
  \NewDocumentCommand{\Alg}{om}{\operatorname{Alg}\IfValueTF{#1}{_{#1}}{}{(#2)}}
  \NewDocumentCommand{\coAlg}{om}{\operatorname{coAlg}\IfValueTF{#1}{_{#1}}{}{(#2)}}
  \NewDocumentCommand{\biAlg}{om}{\operatorname{biAlg}\IfValueTF{#1}{_{#1}}{}{(#2)}}
  \NewDocumentCommand{\HopfAlg}{om}{\operatorname{Hopf}\IfValueTF{#1}{_{#1}}{}{(#2)}}
  \NewDocumentCommand{\CAlg}{om}{\operatorname{CAlg}\IfValueT{#1}{_{#1}}{(#2)}}
  
  \NewDocumentCommand{\Inj}{m}{\operatorname{LMod}_{{#1}}^{\mathrm{inj}}}
  \NewDocumentCommand{\Proj}{m}{\operatorname{LMod}_{#1}^{\mathrm{proj}}}

  \NewDocumentCommand{\Ho}{m}{\operatorname{Ho}(#1)}

  \NewDocumentCommand{\ldual}{m}{{#1}^\vee}
  \NewDocumentCommand{\rdual}{m}{{}^\vee{#1}}

  \NewDocumentCommand{\comment}{om}{\footnote{{\color{blue}\IfValueT{#1}{#1:
        }#2}}}
  
  \NewDocumentCommand{\coev}{o}{\operatorname{coev}\IfValueT{#1}{_{#1}}}
  \NewDocumentCommand{\ev}{o}{\operatorname{ev}\IfValueT{#1}{_{#1}}}
  \RenewDocumentCommand{\H}{O{\bullet}m}{\operatorname{H}^{#1}(#2)}
  
  \newcommand{\PrL}{\operatorname{Pr}^L}
  \newcommand{\PrStL}{\operatorname{Pr}_{\mathrm{St}}^L}
  \newcommand{\PrStLcg}{\operatorname{Cat}_\infty^{\mathrm{perf}}}
  \newcommand{\cats}{\operatorname{Cat}_\infty}
  
  \NewDocumentCommand{\DerCat}{od<>m}{\operatorname{D}\IfValueT{#1}{^\mathrm{#1}}\IfValueT{#2}{_{#2}}(#3)}
  \NewDocumentCommand{\Ch}{od<>md<>}{\operatorname{Ch}\IfValueT{#1}{^\mathrm{#1}}\IfValueT{#2}{_{\mathrm{#2}}}(#3)\IfValueT{#4}{_{\mathrm{#4}}}}
  \NewDocumentCommand{\KCh}{od<>m}{\operatorname{K}\IfValueT{#1}{^\mathrm{#1}}\IfValueT{#2}{_{\mathrm{#2}}}(#3)}

  \NewDocumentCommand{\smashprod}{mm}{{#1}\#{#2}}

  \NewDocumentCommand{\HH}{sO{}m}{\operatorname{HH}\IfBooleanTF{#1}{_{#2}}{^{#2}}(#3)}
  
  \NewDocumentCommand{\Ind}{o}{\operatorname{Ind}\IfValueT{#1}{(#1)}}
  \NewDocumentCommand{\Fun}{d<>omm}{\operatorname{Fun}\IfValueT{#1}{_{#1}}\IfValueT{#2}{^{\mathrm{#2}}}(#3,#4)}
  \NewDocumentCommand{\LFun}{d<>omm}{\operatorname{LFun}\IfValueT{#1}{_{#1}}\IfValueT{#2}{^{\mathrm{#2}}}(#3,#4)}
  \NewDocumentCommand{\Hom}{d<>omm}{\operatorname{Hom}\IfValueT{#1}{_{#1}}\IfValueT{#2}{^{#2}}(#3,#4)}
  \NewDocumentCommand{\Map}{d<>omm}{\operatorname{Map}\IfValueT{#1}{_{#1}}\IfValueT{#2}{^{#2}}(#3,#4)}
  \NewDocumentCommand{\RHom}{d<>omm}{\mathbb{R}\!\operatorname{Hom}\IfValueT{#1}{_{#1}}\IfValueT{#2}{^{#2}}(#3,#4)}
  \RenewDocumentCommand{\hom}{d<>omm}{\operatorname{hom}\IfValueT{#1}{_{#1}}\IfValueT{#2}{^{#2}}(#3,#4)}
  \NewDocumentCommand{\EEnd}{d<>om}{\mathbf{End}\IfValueT{#1}{_{#1}}\IfValueT{#2}{^{#2}}(#3)}
  \NewDocumentCommand{\Ext}{d<>omm}{\operatorname{Ext}\IfValueT{#1}{_{#1}}\IfValueT{#2}{^{#2}}(#3,#4)}
  
  \title{Hopfological algebra, revisited}

  \author[J.~O.~G{\'o}mez]{Juan Omar G{\'o}mez} \address[G{\'o}mez]{Fakultat
    f{\"u}r Mathematik, Universit{\"a}t Bielefeld, 33501 Bielefeld, Germany}
  \email{jgomez@math.uni-bielefeld.de}

  \author[G.~Jasso]{Gustavo Jasso}
  \address[Jasso]{Mathematisches Institut\\Universit{\"a}t zu K{\"o}ln\\ Weyertal 86-90\\
    \\50931 K{\"o}ln, Germany} \email{gjasso@math.uni-koeln.de}

  \author[M.~Nielsen]{Marius Nielsen} \address[Nielsen]{Department of
    Mathematical Sciences\\Norwegian University of Science and
    Technology\\Trondheim, Norway} \email{marius.v.b.nielsen@ntnu.no}

\begin{document}

  \maketitle

\begin{abstract}
  We propose an $\infty$-categorical approach to Khovanov--Qi's Hopfological
  algebra that, in particular, refines several foundational aspects of the
  theory by recasting the previous constructions in terms of $\infty$-categories
  of modules in monoidal $\infty$-categories. This perspective leads to a more
  general variant of Hopfological algebra that takes place over an arbitrary
  rigidly-compactly generated symmetric monoidal stable $\infty$-category, which
  we also outline in the article. In the appendix, we compare the construction
  of Hopfological derived categories to that of Holm--J{\o}rgensen's $Q$-shaped
  derived categories.
\end{abstract}

\setcounter{tocdepth}{1}
\tableofcontents

\section*{Introduction}

In this article, we propose an $\infty$-categorical approach to Hopfological
algebra that, in particular, refines several foundational aspects of the theory
by recasting the previous constructions in terms of $\infty$-categories of
modules in monoidal $\infty$-categories.

Hopfological algebra was introduced by Khovanov and Qi in~\cite{Kho16,Qi14} as a
variant of classical homological algebra that could be used in the
categorification of quantum invariants of $3$-manifolds, see~\cite{QS17} for an
early survey of applications of the theory. The rough idea is the following.
Given a finite-dimensional Hopf algebra $H$ over an arbitrary field, the stable
category $\LModu{H}$ of left $H$-modules admits a triangulated structure by a
well-known theorem of Happel~\cite{Hap88}. The stable category $\LModu{H}$ is a
compactly-generated triangulated category and, moreover, the stable category of
finite-dimensional $H$-modules $\mmodu{H}\subseteq\LModu{H}$ is precisely its
full subcategory of compact objects. So far, these considerations only make use
of the fact that $H$ is a Frobenius algebra.

The Hopf algebra structure of $H$ implies that the stable category $\LModu{H}$
is a monoidal category with respect to the tensor product over the ground field.
The tensor product is exact in each variable separately and it restricts to the
subcategory $\mmodu{H}$. The upshot is that the Grothendieck group
$K_0(\mmodu{H})$ is endowed with a ring structure, and $\mmodu{H}$ is then
regarded as a categorification of its Grothendieck ring. As explained
in~\cite[Section~4.3]{Qi14}, various interesting rings are categorified in this
way.

The next step is to categorify \emph{modules} over the Grothendieck ring
$K_0(\mmodu{H})$. The Hopfological algebra paradigm provides a rich supply of
triangulated categories equipped with a right action of $\mmodu{H}$. For
simplicity, we take as input a left $H$-module algebra $A$, that is an
associative algebra whose multiplication and unit maps are left $H$-module
homomorphisms. The key insight is that the category $\LMod{\smashprod{A}{H}}$ of
left modules over the smash product algebra $\smashprod{A}{H}$ is endowed with a
class of `Hopfological quasi-isomorphisms', which are those maps that are stable
isomorphisms of the underlying $H$-modules. The punchline is then that the
localisation of $\LMod{\smashprod{A}{H}}$ at the class of Hopfological
quasi-isomorphisms yields a compactly-generated triangulated category
$\DerCat{A,H}$, called the Hopfological derived category of $A$, equipped with a
canonical right action of $\LModu{H}$. Consequently, the Grothendieck group of
the full subcategory $\perf{A,H}\subseteq\DerCat{A,H}$ of compact objects
inherits the structure of a right module over the Grothendieck ring
$K_0(\mmodu{H})$. Classical derived categories of differential graded (=DG)
algebras are obtained by letting $H$ be the graded algebra of dual numbers in a
variable of cohomological degree $1$. In this way, the theory of Hopfological
derived categories can be regarded as an extension of the theory of classical
derived categories.

More generally, it is expected that many constructions in classical homological
algebra extend to Hopfological derived categories. Notably, in~\cite{QS22,QS23},
Qi and Sussan introduce $p$-analogues of Hochschild (co)homology for $p$-DG
algebras, where $p$ is an odd prime number, and use these to construct link
invariants that categorify the Jones polynomial. The link to Hopfological
algebra is that $p$-DG algebras are left module algebras for the Hopf algebra
$\kk[\partial]/(\partial^p)$ over a field $\kk$ of characteristic $p$. Other
Hopfological analogues of classical invariants have been also investigated
in~\cite{Far21,OT20a,Oha25a}.

Following the seminal work of Khovanov~\cite{Kho16}, the foundations of
Hopfological algebra were laid out by Qi in~\cite{Qi14} by extending Keller's
seminal work on derived categories of differential graded
algebras~\cite{Kel94}. Afterwards, alternative approaches using Quillen's theory
of model categories were pursued in~\cite{OT20a,Oha24,Oha25}.

In this article, we revisit the foundations of the subject from the perspective
of the theory of $\infty$-categories, which serves as a model for the theory of
$(\infty,1)$-categories. From a technical standpoint, we leverage the robust
theory of monoidal $\infty$-categories, and of algebras and modules therein,
developed extensively by Lurie in~\cite{Lur17}. For this, we work with the
monoidal stable $\infty$-category $\StLMod{H}$ of $H$-modules, which is the
canonical $\infty$-categorical enhancement of the triangulated $\LModu{H}$.
Given a left $H$-module algebra $A$, we propose to define the Hopfological
derived $\infty$-category of $A$ to be the stable $\infty$-category\footnote{In
  the sequel, we reserve the symbol $\DerCat{A,H}$ for the Hopfological derived
  $\infty$-category, and denote its triangulated $1$-categorical variant by
  $\Ho{\DerCat{A,H}}$ instead.}
\[
  \DerCat{A,H}\coloneqq\LMod{A}[\StLMod{H}]
\]
of left $A$-modules internal to the monoidal $\infty$-category $\StLMod{H}$;
with this definition, $\DerCat{A,H}$ is canonically a right $\StLMod{H}$-module
$\infty$-category. As we explain in \Cref{def:Hopfological_derived_category},
the homotopy category of the stable $\infty$-category $\DerCat{A,H}$ is
precisely the triangulated Hopfological derived category of $A$, further
justifying the above definition. From a conceptual point of view, this approach
allows to embed Hopfological algebra into the theory of right
$\StLMod{H}$-module $\infty$-categories in a precise sense. From a practical
point of view, the general theory of $\infty$-categories of modules has many
straightforward, but interesting, consequences in this context. For example,
Hopfological derived $\infty$-categories of algebras with trivial $H$-action
admit a particularly simple description in terms of Lurie's relative tensor
product.

\begin{theorem intro}[\Cref{thm:splitting_formula}]
  Suppose that $H$ acts trivially on $A$. Then, there is an equivalence of
  $\infty$-categories
  \[
    \DerCat{A,H}\simeq\DerCat{\LMod{A}}\otimes_{\DerCat{\Mod{\kk}}}\StLMod{H},
  \]
  where the relative tensor product is taken over the derived $\infty$-category
  of vector spaces.
\end{theorem intro}

The original construction of Hopfological derived categories does not involve
ordinary derived categories nor chain complexes of modules. The following
derived Morita invariance result, which is an immediate consequence of the
tensor-product formula above, is therefore surprising.

\begin{theorem intro}[\Cref{thm:Morita_invariance}]
  Let $A$ and $B$ be a pair of left $H$-module algebras on which $H$ acts
  trivially. Suppose that $X$ is a chain complex of ordinary $B$-$A$-bimodules
  such that the functor
  \[
    X\Lotimes*_A-\colon\DerCat{\LMod{A}}\stackrel{\sim}{\longrightarrow}\DerCat{\LMod{B}}
  \]
  is an equivalence of derived $\infty$-categories. Then, there is an induced
  equivalence of Hopfological derived $\infty$-categories
  \[
    X\Lotimes*_A\colon\DerCat{A,H}\stackrel{\sim}{\longrightarrow}\DerCat{B,H}.
  \]
\end{theorem intro}

The interplay between model categories and
$\infty$-categories, which is central in our work, also permits us to lift
Krause's recollement for the stable $\infty$-category $\StLMod{H}$~\cite{Kra05}
to the Hopfological setting.

\begin{theorem intro}[\Cref{thm:Krauses-recollement-Hopfological}]
  There is a commutative diagram of recollements of compactly-generated stable
  $\infty$-categories
  \[
    \begin{tikzcd}
      \DerCat{A,H}\ar[hookrightarrow]{r}\dar&\LMod{A}[\KCh{\Inj{H}}]\ar{d}\ar{r}\ar[shift
      right=0.5em]{l}\ar[shift
      left=0.5em]{l}&\DerCat{\LMod{\smashprod{A}{H}}}\ar[shift
      right=0.5em,hook]{l}\ar[hook',shift left=0.5em]{l}\dar\\
      \StLMod{H}\ar[hookrightarrow]{r}&\KCh{\Inj{H}}\ar{r}\ar[shift
      right=0.5em]{l}\ar[shift left=0.5em]{l}&\DerCat{\LMod{H}}\ar[shift
      right=0.5em,hook]{l}\ar[hook',shift left=0.5em]{l}
    \end{tikzcd}
  \]
  where the vertical functors are the corresponding forgetful functors and the
  bottom row is given by Krause's recollement~\cite{Kra05}.
\end{theorem intro}

The $\infty$-categorical approach provides a new perspective on the derived
Morita theory for the Hopfological derived $\infty$-category, discussed first
in~\cite{Qi14}, see~\Cref{subsec:Morita_theory} for details. From the point of
view of categorification, in lieu of the lax symmetric monoidal structure on the
algebraic $K$-theory functor~\cite{BGT14}, the right action of the Grothendieck
ring $K_0(\stLmod{H})$ on $K_0(\perf{A,H})$ extends to a right action
\[
  K(\perf{A,H})\otimes K(\stLmod{H})\longrightarrow K(\stLmod{H})
\]
of the full algebraic $K$-theory ring spectrum $K(\stLmod{H})$ on
$K(\perf{A,H})$, see~\Cref{subsec:additive_localising_inv,subsubsec:K-theory}
for details. The $\infty$-categorical formalism also leads to conceptual
definitions of Hopfological analogues of Hochschild (co)homology,
see~\Cref{subsubsec:Hochschild}. Finally, in \Cref{subsec:descent} we explain a
descent property enjoyed by the derived Hopfological $\infty$-category,
which is a special case of the robust descent property of module $\infty$-categories.

The interpretation of Hopfological derived $\infty$-categories as
$\infty$-categories of modules leads to a variant of Hopfological algebra in
which the derived $\infty$-category of the ground field is replaced by an
arbitrary rigidly-compactly generated symmetric stable $\infty$-category; for
example, one can do Hopfological algebra over the sphere spectrum. We outline
this generalisation in \Cref{sec:spectral_Hopfological_algebra}, the results of
which properly contain the results mentioned above. However, we see value in
including both, for the proof techniques used are different: In
\Cref{sec:Hopfological_dercats} we leverage the interplay between the theory of
Quillen model categories and the theory of $\infty$-categories, while in
\Cref{sec:spectral_Hopfological_algebra} we work exclusively with
$\infty$-categories.

\subsection*{Relation to other works}

As mentioned above, the model-categorical approach to Hopfological algebra was
pioneered in~\cite{Far21,OT20a,Oha24,Oha25}. In this article, we make heavy use of
model-categorical techniques. In particular, we construct the Hopfological
derived $\infty$-category of a left $H$-module algebra $A$ as the underlying
$\infty$-category of the model structure constructed by Ohara in~\cite{Oha25}.
Although, in principle, we could work within the $\infty$-categorical framework
from the outset, we have opted to use model categories as this facilitates the
comparison with the previous constructions at the triangulated level. Indeed,
this permits us to use various `rectification' results for $\infty$-categories
of modules from~\cite{Lur17} in order to perform the various necessary
comparisons in a straightforward manner. Combined with work of
Becker~\cite{Bec14}, these rectification results yield an alternative
model-categorical presentation of Hopfological derived ($\infty$-)categories
that, to our knowledge, is not present in the literature,
see~\Cref{rmk:cosing-presentation_Hopfological_derived_cat}.

At a conceptual level, this article is also related to the second-named author's
work~\cite{Jas25} and, indeed, some of our results are of similar nature. The
article~\cite{Jas25} deals not with Hopfological derived categories, but with
$Q$-shaped derived categories in the sense of Holm and J{\o}rgensen~\cite{HJ22}.
The reader that is familiar with both theories would have noticed that there are
certain formal similarities between them, which we attempt to elucidate in the
appendix.

\subsection*{Structure of the article}

In \Cref{sec:preliminaries} we recall the necessary preliminaries; these include
the relationship between (monoidal) model categories and (monoidal)
$\infty$-categories, rectification results for algebras objects and their
modules, stable $\infty$-categories and their relationship with Frobenius exact
categories, and Hovey's (monoidal) abelian model structures. In
\Cref{sec:base_category}, we discuss in some detail the base monoidal
$\infty$-category, namely the stable $\infty$-category of $H$-modules. The
material in these first two sections is well known to experts. Finally, in
\Cref{sec:Hopfological_dercats} we revisit the construction of the Hopfological
derived category from the $\infty$-categorical perspective and discuss some of
the many consequences that can be obtained by applying the general theory of
modules in monoidal $\infty$-categories, as well as our main results. In
\Cref{sec:spectral_Hopfological_algebra}, we outline an approach to Hopfological
algebra over an arbitrary rigidly-compactly generated symmetric stable
$\infty$-category. The comparison with Holm--J{\o}rgensen's $Q$-derived
categories is conducted in the appendix.

\subsection*{Conventions}

We work over an arbitrary field $\kk$ and denote the category of all
$\kk$-vector spaces by $\Mod{\kk}$. If $A$ is a $\kk$-algebra, we denote the
category of all left $A$-modules by $\LMod{A}$; similarly, we denote the
category of all right $A$-modules by $\RMod{A}$. If the algebra $A$ is
commutative, we also write
\[
  \Mod{A}\coloneqq\LMod{A}\cong\RMod{A}.
\]
In~\cite{Lur17}, the notation $\LMod{A}$ is most often used to denote the
derived $\infty$-category of the category of left $A$-modules, which we denote
by $\DerCat{\LMod{A}}$ instead. More generally, given a monoidal category $\M$,
we write $\Alg{\M}$ for the category of (unital, associative) algebra objects in
$\M$. For an algebra object $A\in\Alg{\M}$, we write $\LMod{A}[\M]$ for the
category of left $A$-module objects in $\M$ and $\RMod{A}[\M]$ for the category
of right $A$-module objects in $\M$. We also use these notations and conventions in the
$\infty$-categorical setting. For details on the theory of algebras and modules
in monoidal ($\infty$-)categories, we refer the reader
to~\cite[Chapter~7]{EGNO15} (for the $1$-categorical case) and~\cite[Chapters~3
and~4]{Lur17} or \cite[Chapter~1]{Erg22} (for the $\infty$-categorical case).
Throughout the article, we use freely the theory of $\infty$-categories
developed by Joyal~\cite{Joy,Joy02,Joy08}, Lurie~\cite{Lur09,Lur17,Lur18SAG} and
others; we refer the reader to~\cite{ACam16,Gro20,Jas26} for surveys of the
theory, and to \cite{Cis19,Lan21,Kerodon} for textbook accounts. We also assume
basic familiarity with the theory of Quillen model categories, for which our
main reference is~\cite{Hov99}. In addition, we also assume that the reader is
familiar with the theory of compactly-generated triangulated categories, for
which~\cite{Nee01} is a useful reference. Finally, our main references for the
theory of Hopf algebras and their representations are~\cite{EGNO15,Kas95,Mon93}.

\subsection*{Acknowledgements} 
JOG is supported by the Deutsche Forschungsgemeinschaft (Project-ID 491392403 – TRR 358).

\section{Preliminaries}
\label{sec:preliminaries}

\subsection{Quillen model categories and their underlying $\infty$-categories}
\label{subsec:Quillen-model-cats}

Let $\A$ be a category with all small limits and all small colimits. Recall that
a \emph{(closed) model category structure} on $\A$ is a triple
$(\Cofib,\Weq,\Fib)$ of (in general, non-full) subcategories of $\A$ called
\emph{cofibrations}, \emph{weak equivalences} and \emph{fibrations},
respectively, subject to a number of axioms, see~\cite[Definition~1.1.3]{Hov99}
for the precise definition. For later use, we recall that an object $x\in\A$ is
\emph{cofibrant} if the unique morphism $\emptyset\to x$ is a cofibration;
dually, $x\in\A$ is \emph{fibrant} if the unique morphism $x\to *$ is a
fibration. We denote by $\C$ and $\F$ the classes of cofibrant and of fibrant
objects, respectively. In this article, we are interested in the
\emph{underlying $\infty$-category}~\cite[Definition~1.3.4.1]{Lur17}
\[
  \MLoc{\A}\coloneqq\A[\Weq^{-1}]
\]
of $\A$ which, by definition, is the $\infty$-categorical localisation of $\A$
at the class of weak equivalences. In this case, we may identify
$\Ho{\MLoc{\A}}$ with the $1$-categorical localisation of $\A$ at the class of
weak equivalences~\cite[Remark~7.1.6]{Cis19}, so that we may regard $\MLoc{\A}$
as an `$\infty$-categorical enhancement' of $\Ho{\MLoc{\A}}$.

\begin{remark}
  The $\infty$-category $\MLoc{\A}$ can be equivalently obtained as the
  $\infty$-categorical localisation of any of the following subcategories of
  $\A$ at the class of weak equivalences: the full subcategories of cofibrant
  objects, of fibrant objects and of cofibrant-fibrant objects,
  see~\cite[Remark~1.3.4.16]{Lur17} and~\cite[Theorem~7.5.18]{Cis19}.
  \[
    \begin{tikzcd}
      \MLoc{\C\cap\F}\rar{\sim}\dar{\wr}&\MLoc{\C}\dar{\wr}\\
      \MLoc{\F}\rar{\sim}&\MLoc{\A}
    \end{tikzcd}
  \]
  Recall also that every weak equivalence between cofibrant-fibrant objects is a
  homotopy equivalence in the sense of~\cite[Definition~1.2.4]{Hov99}.
\end{remark}

Under suitable assumptions, the underlying $\infty$-category of a model category
has robust categorical properties.

\begin{theorem}[{\cite[Proposition~1.3.4.22]{Lur17}}]
  \label{thm:MLoc-presentable}
  Let $\A$ be a combinatorial model category; that is, the category $\A$ is
  presentable and the model structure on $\A$ is cofibrantly generated in the
  sense of \cite[Definition~2.1.17]{Hov99}. Then, the underlying
  $\infty$-category $\MLoc{\A}$ is a presentable
  $\infty$-category.\footnote{Details on the theory of presentable
    $\infty$-categories can be found in~\cite[Chapter~5]{Lur09}. In short,
    presentable $\infty$-categories are large $\infty$-categories that are
    `controlled' by small $\infty$-categories (so that certain set-theoretic
    issues can be avoided). In addition, a strong form of the Adjoint Functor
    Theorem is valid in this context, see~\cite[Corollary~5.5.2.9]{Lur09}.}
\end{theorem}

The following theorem is our main tool for constructing the monoidal
$\infty$-categories of interest in this article.

\begin{theorem}[Lurie]
  \label{thm:MLoc-presentably-monoidal}
  Let $\M=(\M,\unit,\otimes)$ be a monoidal category. Suppose that $\M$ is
  endowed with a combinatorial model structure that is monoidal in the sense
  of~\cite[Definition~4.2.6]{Hov99}. The following statements hold:
  \begin{enumerate}
  \item\cite[Example~4.1.7.6, Lemma~4.1.8.8]{Lur17} The presentable
    $\infty$-category $\MLoc{\M}$ inherits the structure of a monoidal
    $\infty$-category
    \begin{center}
      $(\MLoc{\M},\unit,\Lotimes*),$
    \end{center}
    in the sense of~\cite[Definition~4.1.1.10]{Lur17}, which is symmetric if the
    monoidal structure on $\M$ is symmetric. Moreover, the localisation functor
    \[
      \M\longrightarrow\MLoc{\M}
    \]
    is monoidal and is universal among all monoidal functors
    out of $\M$ that invert the morphisms in $\Weq$.
  \item\cite[Lemma~4.1.8.8]{Lur17} The tensor product functor
    \begin{center}
      $-\Lotimes-\colon\MLoc{\M}\times\MLoc{\M}\longrightarrow\MLoc{\M}$
    \end{center}
    preserves colimits in each variable separately.
  \item\cite[Corollary~5.5.2.9]{Lur09} The monoidal structure on $\MLoc{\M}$ is
    biclosed.
  \end{enumerate}
\end{theorem}

Below we recall several facts about monoidal categories, algebra objects and
their modules that are used later in the article. We begin by recalling a
convenient definition.

\begin{definition}
  A \emph{presentably (symmetric) monoidal $\infty$-category} is a presentable
  $\infty$-category $\M$ equipped with a (symmetric) monoidal structure whose
  tensor product functor
  \[
    -\otimes-\colon\M\times\M\longrightarrow\M
  \]
  preserves colimits in each variable separately.\footnote{This means that $\M$
    can be regarded as an algebra object in the $\infty$-category $\PrL$ of
    presentable $\infty$-categories and colimit-preserving functors, endowed
    with Lurie's tensor product~\cite[Proposition~4.8.1.15]{Lur17}.}
\end{definition}

\begin{remark}
  In the context of \Cref{thm:MLoc-presentably-monoidal}, the underlying
  $\infty$-category $\MLoc{\M}$ is a presentably monoidal $\infty$-category. For
  a suitable converse to \Cref{thm:MLoc-presentably-monoidal}, see~\cite{NS17a}.
\end{remark}

\begin{notation}[{\cite[Construction~5.2.5.18]{Lur17}}]
  \label{not:reversed_monoidal_cat}
  Let $\M$ be a monoidal $\infty$-category. We denote the monoidal
  $\infty$-category obtained from $\M$ by reversing the monoidal product (but
  not the direction of the morphisms) by $\M^\rev$. If the monoidal
  $\infty$-category $\M$ refines to a symmetric monoidal $\infty$-category, then there is an identification
  $\M\simeq\M^\rev$.
\end{notation}

\begin{remark}
  Let $\M$ be a monoidal model category. Then, $\M^\rev$ is also a monoidal
  model category (with identical model structure) and there is a canonical
  identification
  \[
    \MLoc{\M^\rev}\simeq\MLoc{\M}^\rev.
  \]
\end{remark}

Our $\infty$-categorical approach to Hopfological algebra relies heavily on the
construction of $\infty$-categories of modules internal to monoidal
$\infty$-categories. The first result that we need is the following
rectification result for associative algebra objects.

\begin{theorem}[{\cite{SS00,Lur17}}]
  \label{thm:SS-transfer-model_structure}
  Let $\M$ be a combinatorial monoidal model category such that every object is
  cofibrant. The following statements hold:
  \begin{enumerate}
  \item\cite{SS00}\cite[Proposition~4.1.8.3]{Lur17} The category $\Alg{\M}$ of
    algebra objects in $\M$ admits a (right-transferred) combinatorial model
    structure determined as follows:
    \begin{itemize}
    \item A morphism in $\Alg{\M}$ is a weak equivalence if its underlying
      morphism is a weak equivalence in $\M$.
    \item A morphism in $\Alg{\M}$ is a fibration if its underlying morphism is
      a fibration in $\M$.
    \item A morphism in $\Alg{\M}$ is a cofibration if it has the left lifting
      property with respect to the trivial fibrations.
    \end{itemize}
    Moreover, the forgetful functor $\Alg{\M}\to\M$ is a right Quillen functor.
  \item\cite[Theorem~4.1.8.4]{Lur17} There is a canonical equivalence of
    $\infty$-categories
    \begin{center}
      $\MLoc{\Alg{\M}}\stackrel{\sim}{\longrightarrow}\Alg{\MLoc{\M}}.$
    \end{center}
  \end{enumerate}
\end{theorem}

\begin{notation}[{\cite[Construction~5.2.5.18]{Lur17}}]
  Let $\M$ be a monoidal $\infty$-category. The passage $\M\mapsto\M^\rev$ is
  compatible with the formation of algebra objects: There is an induced
  equivalence of $\infty$-categories
  \[
    \rev\colon\Alg{\M}\stackrel{\sim}{\longrightarrow}\Alg{\M^\rev},\qquad
    A\longmapsto A^\rev,
  \]
  where $A^\rev=A$ as objects of $\M$. Thus, if the monoidal $\infty$-category
  $\M$ admits a symmetric refinement, we obtain an equivalence of $\infty$-categories
  \[
    \rev\colon\Alg{\M}\stackrel{\sim}{\longrightarrow}\Alg{\M},\qquad
    A\longmapsto A^\rev.
  \]
\end{notation}

We also need the following rectification result for (bi)module objects.

\begin{theorem}[Lurie]
  \label{thm:Lurie-transfer-model_structure}
  Let $\M$ be a combinatorial monoidal model category and $A\in\Alg{\M}$ an
  algebra object whose underlying object is cofibrant in $\M$. The following statements
  hold:
  \begin{enumerate}
  \item\cite[Theorem~4.3.3.15]{Lur17} The category $\LMod{A}[\M]$ of left
    $A$-module objects in $\M$ admits a (right-transferred) combinatorial model
    structure determined as follows:
    \begin{itemize}
    \item A morphism in $\LMod{A}[\M]$ is a weak equivalence if its underlying
      morphism is a weak equivalence in $\M$.
    \item A morphism in $\LMod{A}[\M]$ is a fibration if its underlying morphism
      is a fibration in $\M$.
    \item A morphism in $\LMod{A}[\M]$ is a cofibration if it has the left
      lifting property with respect to the trivial fibrations.
    \end{itemize}
    Moreover, the forgetful functor $\LMod{A}[\M]\to\M$ is both a left Quillen
    functor and a right Quillen functor (in particular, it admits a left and a
    right adjoint).
  \item\cite[Theorem~4.3.3.17]{Lur17} There is a canonical equivalence of
    $\infty$-categories
    \begin{center}
      $\MLoc{\LMod{A}[\M]}\stackrel{\sim}{\longrightarrow}\LMod{A}[\MLoc{\M}],$
    \end{center}
    where the right-hand side denotes the $\infty$-category of left $A$-module
    objects in the monoidal $\infty$-category of $\MLoc{\M}$.
  \end{enumerate}
  Entirely analogous statements hold for the category $\Bimod{B}{A}[\M]$ of
  $B$-$A$-bimodule objects in $\M$, over a pair $A,B\in\Alg{\M}$ of algebra
  objects whose underlying objects are cofibrant.
\end{theorem}

\begin{remark}[{\cite[Construction~4.6.3.1, Remark~4.6.3.2]{Lur17}}]
  Let $\M$ be a monoidal $\infty$-category and $A\in\Alg{\M}$ an algebra object.
  There is a canonical equivalence of $\infty$-categories
  \[
    \rev\colon\LMod{A}[\M]\stackrel{\sim}{\longrightarrow}\RMod{A^\rev}[\M^\rev],\qquad
    M\longmapsto M^\rev,
  \]
  where $M^\rev=M$ as objects of $\M$. More generally, for a pair of algebra
  objects $A,B\in\Alg{\M}$, there is a canonical equivalence of
  $\infty$-categories
  \[
    \rev\colon\Bimod{B}{A}[\M]\stackrel{\sim}{\longrightarrow}\Bimod{A^\rev}{B^\rev}[\M^\rev],\qquad
    M\longmapsto M^\rev.
  \]
  Thus, if the monoidal $\infty$-category $\M$ admits a symmetric refinement, we obtain
  equivalences of $\infty$-categories
  \[
    \rev\colon\LMod{A}[\M]\stackrel{\sim}{\longrightarrow}\RMod{A^\rev}[\M],\qquad
    M\longmapsto M^\rev,
  \]
  and
  \[
    \rev\colon\Bimod{B}{A}[\M]\stackrel{\sim}{\longrightarrow}\Bimod{A^\rev}{B^\rev}[\M],\qquad
    M\longmapsto M^\rev.
  \]
\end{remark}

If the base monoidal $\infty$-category has sufficient (co)limits, then we have
access to the familiar change-of-algebras adjoint triple.

\begin{proposition}[{\cite[Corollary~4.3.3.10]{Lur17}}]
  \label{prop:change_of_algebra}
  Let $\M$ be a presentably monoidal $\infty$-category. Given algebra objects
  $A,B,C\in\Alg{\M}$ and a morphism of algebra objects ${f\colon A\to B}$, there
  is an adjoint triple
  \[
    \begin{tikzcd}
      \Bimod{C}{B}[\M]\ar{r}[description]{f^*}&\Bimod{C}{A}[\M]\lar[shift
      left=0.5em]{f_*}\lar[shift right=0.5em,swap]{f_!}.
    \end{tikzcd}
  \]
  At the level of objects, the functor $f_!$ is given by the formula $f_!\colon
  X\mapsto X\otimes_B A$, where $\otimes_B$ denotes the relative tensor product
  over $B$, see~\cite[Section~4.4]{Lur17}.
\end{proposition}

We also remind the reader that lax monoidal functors preserve algebra objects
and their modules. The precise statement is given below; for further details, we 
refer the reader to {\cite[Section~2.1]{Lur17}} and {\cite[Remark~1.1.11]{Erg22}}, 
as well as the references therein.

\begin{proposition}[Lurie]
  \label{prop:monoidal_functors_preserve_algs_mods}
  Let $F\colon\M\to\N$ be a lax monoidal functor between monoidal
  $\infty$-categories. The following statements hold:
  \begin{enumerate}
  \item The functor $F$ induces a functor
    \begin{center}
      $F\colon\Alg{\M}\longrightarrow\Alg{\N},\qquad A\longmapsto FA,$
    \end{center}
    between the corresponding $\infty$-categories of algebra objects.
  \item If the monoidal structures on $\M$ and $\N$ admit symmetric refinements and the
    functor $F$ is lax symmetric monoidal, then it induces a functor
    \begin{center}
      $F\colon\CAlg{\M}\longrightarrow\CAlg{\N},\qquad A\longmapsto FA,$
    \end{center}
    between the corresponding $\infty$-categories of commutative algebra
    objects.\footnote{By a commutative algebra object, we mean an algebra object
      over the commutative $\infty$-operad~\cite[Definition~2.1.1.18]{Lur17}.}
  \item The functor $F$ induces a functor
    \begin{center}
      $\RMod{A}[\M]\longrightarrow\RMod{FA}[\N],\qquad M\longmapsto FM,$
    \end{center}
    between the corresponding $\infty$-categories of internal right modules.
    Analogous statements hold for $\infty$-categories of left modules and of
    bimodules.
  \end{enumerate}
\end{proposition}

We also need the following general facts about module $\infty$-categories.

\begin{theorem}[Lurie]
  \label{thm:props_of_Mod}
  Let $\M$ be a presentably monoidal $\infty$-category. Given an algebra object
  $A\in\Alg{\M}$, the following statements hold:
  \begin{enumerate}
  \item \cite[Section~4.3.2, Corollary 4.3.3.10]{Lur17} The $\infty$-category
    $\LMod{A}[\M]$ is presentable and it is right-tensored over $\M$ in the
    sense of~\cite[Definition~4.2.1.19]{Lur17}.
  \item\cite[Lemma~5.3.2.12(3)]{Lur17}\footnote{Although the statement of
      \cite[Lemma~5.3.2.12]{Lur17} involves uncountable (regular) cardinals, the
      proof of compact generation remains valid for the (regular) cardinal
      $\omega$.} If the $\infty$-category $\M$ is compactly generated in the
    sense of~\cite[Definition~5.5.7.1]{Lur09}, then $\LMod{A}[\M]$ is compactly
    generated by all objects of the form $A\otimes X$, where $X\in\M$ is a
    compact object.
  \item \cite[Theorem~4.8.4.6, Remark~4.8.4.7]{Lur17} Let $\N$ be a presentable
    $\infty$-category that is left-tensored over $\M$. Then, there is a
    canonical equivalence of $\infty$-categories
    \begin{center}
      $\LMod{A}[\M]\otimes_{\M}\N\simeq\LMod{A}[\N],$
  \end{center}
  where the left-hand side is defined via the relative tensor product
  \begin{center}
    $\RMod{\M}[\PrL]\times\LMod{\M}[\PrL]\longrightarrow\PrL,\qquad (\C,\D)\longmapsto\C\otimes_\M\D.$
  \end{center}
  \end{enumerate}
  Entirely analogous statements hold for $\infty$-categories of right modules in
  $\M$.
\end{theorem}

Over a symmetric monoidal $\infty$-category, one has the following Deligne-type
tensor product for module $\infty$-categories, compare
with~\cite[Section~5]{Del90}.

\begin{corollary}[{\cite{Lur17}}]
  \label{cor:Deligne-tensor-product}
  Let $\M$ be a presentably symmetric monoidal presentable $\infty$-category. Given a pair of algebra
  objects $A,B\in\Alg{\M}$, there is an equivalence of $\infty$-categories
  \[
    \LMod{A}[\M]\otimes_\M\LMod{B}[\M]\simeq\LMod{A\otimes B}[\M].
  \]
\end{corollary}
\begin{proof}
  Indeed, \Cref{thm:props_of_Mod}, yields the second equivalence in the
  following chain of equivalences of $\infty$-categories
  \begin{align*}
    \LMod{A}[\M]\otimes_\M\LMod{B}[\M]&\simeq\LMod{A}[\M]\otimes_\M\RMod{B^\rev}[\M]\\
                                      &\simeq\LMod{A}[\RMod{B^\rev}[\M]]\\
                                      &\simeq\Bimod{A}{B^\rev}[\M]\\
                                      &\simeq\LMod{A\otimes B}[\M].
  \end{align*}
  The third equivalence is discussed in~\cite[p.~738]{Lur17} and the fourth
  equivalence follows from~\cite[Proposition~4.6.3.11]{Lur17}.
\end{proof}

\subsection{Stable $\infty$-categories}

An $\infty$-category $\C$ is \emph{stable} if it admits a zero object, finite
limits, finite colimits, and the loop functor $\Omega\colon x\mapsto 0\times_x0$
is an equivalence~\cite[Proposition 1.4.2.11]{Lur17}. In this case, the homotopy
category $\Ho{\C}$ is additive and the pair $(\Ho{\C},\Sigma)$ is endowed with
the structure of a triangulated category, where the suspension functor
$\Sigma\colon x\mapsto 0\amalg_x0$ is the adjoint quasi-inverse of $\Omega$; the
triangles in $\Ho{\C}$ are induced by all the diagrams in $\C$ of the form
\[
  \begin{tikzcd}
    x\rar\dar\ar[phantom]{dr}[description]{\square}&y\dar\rar\ar[phantom]{dr}[description]{\square}&0\dar\\
    0\rar&z\rar&\Sigma x
  \end{tikzcd}
\]
in which both squares are bicartesian, see~\cite[Theorem~1.1.2.14]{Lur17} for
details.

For later use, we recall that the $\infty$-category $\PrStLcg$ of essentially
small idempotent-complete stable $\infty$-categories admits a closed symmetric
monoidal structure~\cite[Theorem~3.1]{BGT13}. Given a pair of objects
$\A,\B\in\PrStLcg$, the corresponding internal $\operatorname{Hom}$ object is
the stable $\infty$-category $\Fun[ex]{\A}{\B}$ of exact functors between them.
We remind the reader that a functor $F\colon\A\to\B$ is \emph{exact} if it
satisfies the following equivalent conditions~\cite[Proposition~1.1.4.1]{Lur17}:
\begin{itemize}
\item The functor $F$ preserves finite colimits.
\item The functor $F$ preserves finite limits.
\end{itemize}
In particular, we may consider algebra objects in $\PrStLcg$, that is monoidal
stable $\infty$-categories whose tensor product is exact in each variable
separately.

\begin{definition}
  \label{def:right-M-module_cat-small}
  Let $\M\in\Alg{\PrStLcg}$, that is $\M$ is a monoidal stable
  $\infty$-category whose tensor product is exact in each variable separately. We call
  $\RMod{\M}[\PrStLcg]$ the $\infty$-category of \emph{right $\M$-module (stable) $\infty$-categories}. If the
  monoidal structure on $\M$ is symmetric, we write
  $\Mod{\M}[\PrStLcg]=\RMod{\M}[\PrStLcg]$ and call it the $\infty$-category of
  \emph{$\M$-module (stable) $\infty$-categories} or of \emph{$\M$-linear
    (stable) $\infty$-categories}.
\end{definition}

We shall primarily be concerned with the closed symmetric monoidal
$\infty$-category $\PrStL$ of presentable stable
$\infty$-categories\footnote{The homotopy category of a presentable stable
  $\infty$-category is a well-generated triangulated category in the sense of
  Neeman~\cite{Nee01}, see ~\cite[Proposition~6.10]{Ros05}
  and~\cite[Proposition~A.3.7.6]{Lur09}.} and colimit-preserving functors
between them~\cite[Proposition~4.8.2.18]{Lur17}. Given a pair of objects
$\A,\B\in\PrStL$, the corresponding internal $\operatorname{Hom}$ object is the
presentable stable $\infty$-category $\LFun{\A}{\B}$ of colimit-preserving
functors between them; by the Adjoint Functor
Theorem~\cite[Corollary~5.5.2.9]{Lur09}, every colimit-preserving functor
$\A\to\B$ admits a right adjoint, hence the notation. The algebra objects of
$\PrStL$ play a central role in this article.

\begin{definition}
  \label{def:right-M-module_cat}
  Let $\M\in\Alg{\PrStL}$, that is $\M$ is a presentably monoidal stable
  $\infty$-category. We call
  $\RMod{\M}[\PrStL]$ the $\infty$-category \emph{right $\M$-module (stable) $\infty$-categories}. If the
  monoidal structure on $\M$ is symmetric, we write
  $\Mod{\M}[\PrStL]=\RMod{\M}[\PrStL]$ and call it the $\infty$-category of
  \emph{$\M$-module (stable) $\infty$-categories} or of \emph{$\M$-linear (stable) $\infty$-categories}.
\end{definition}

The symmetric monoidal $\infty$-categories $\PrStLcg$ and $\PrStL$ are related
by the $\operatorname{Ind}$-completion\footnote{Given an essentially small
  (stable) $\infty$-category $\C$, its $\operatorname{Ind}$-completion is the
  (stable) $\infty$-category obtained from $\C$ by freely adjoining filtered
  colimts, see~\cite[Section~5.3.5]{Lur09}
  and~\cite[Proposition~1.1.3.6]{Lur17}.} functor
\[
  \Ind\colon\PrStLcg\longrightarrow\PrStL,\qquad
  \C\longmapsto\Ind[\C],
\]
This is a symmetric monoidal functor (see~\cite[Example~5.3.6.8]{Lur09}
and~\cite[Remark~4.8.1.8]{Lur17}) that identifies $\PrStLcg$ with the non-full
(!) subcategory of $\PrStL$ whose objects are the compactly-generated stable
$\infty$-categories and with morphisms the colimit-preserving functors that
preserve compact objects. In particular, we may apply
\Cref{prop:monoidal_functors_preserve_algs_mods} in this context.

\begin{remark}
  We have introduced certain ambiguity in
  \Cref{def:right-M-module_cat,def:right-M-module_cat-small}, for the
  terminology does not specify whether we are in the `large world' of presentable
  $\infty$-categories or in the `small world.' However, this distinction is implicit in
  the size of $\M$.
\end{remark}

\subsection{Frobenius exact categories}

Let $\E$ be a \emph{Frobenius exact category}, that is a Quillen exact category
with enough projectives, enough injectives, and such that the classes of
projective objects and of injective objects in $\E$
coincide~\cite[Section~I.2]{Hap88}. The \emph{stable category} of $\E$, denoted
$\underline{\E}$, is the quotient of $\E$ by its ideal of morphisms that factor
through a projective-injective object. We denote by $\Weq[st]$ the class of
\emph{stable isomorphisms} in $\E$, that is the morphisms in $\E$ whose class in
$\underline{\E}$ is an isomorphism. Recall also that the stable category of $\E$
carries a triangulated structure~\cite[Theorem~I.2.6]{Hap88}.

The following result is well-known to experts,
see~\cite[Proposition~2.3.1]{Jas26} for a proof that relies heavily on results
of~\cite{Cis19}. The upshot is that every algebraic triangulated category in the
sense of~\cite[Section~3.6]{Kel06} admits an $\infty$-categorical enhancement.

\begin{proposition}[Cisinski]
  \label{prop:Cisinski-Frobenius-stable-cat}
  Let $\E$ be a Frobenius exact category. Then, the $\infty$-categorical
  localisation $\E[\Weq[st]^{-1}]$ is a stable $\infty$-category and, moreover,
  there is a canonical equivalence of triangulated categories
  \[
    \Ho{\E[\Weq[st]^{-1}]}\simeq\underline{\E}.
  \]
\end{proposition}

\subsection{Hovey's abelian model categories}

The model category structures that we consider in this article are, for the most
part, abelian model structures in the sense of Hovey. The definition is as
follows.

\begin{definition}[{\cite{Hov02}}]
  Let $\A$ be an abelian category with all small limits and all small colimits.
  A model structure on $\A$ is \emph{abelian} if the cofibrations are precisely
  the monomorphisms whose cokernel is cofibrant and the fibrations are precisely
  the epimorphisms whose kernel is fibrant.
\end{definition}

\begin{notation}
  Given an abelian model category $\A$, we denote by $\W$ the full subcategory
  of $\A$ spanned by the \emph{acyclic objects}, that is those objects $x\in\A$
  such that the unique morphism $0\to x$ is a weak equivalence (equivalently, by
  the 2-out-of-3 property, such that the unique morphism $x\to 0$ is a weak
  equivalence).
\end{notation}

\begin{definition}
  \label{def:determined_abelian_model_structure}
  Let $\A$ be an abelian category with all small limits and all small colimits.
  Suppose given a triple $(\C,\W,\F)$ of full subcategories of $\A$ and define
  the following classes of morphisms in $\A$:
  \begin{itemize}
  \item $\Cofib$ is the class of all monomorphisms with cokernel in $\C$.
  \item $\Fib$ is the class of all epimorphisms with kernel in $\F$.
  \item $\Weq$ is the class of morphisms of the form $w=pi$ where $i$ is a
    monomorphism with cokernel in $\W$ and $p$ is an epimorphism with kernel in
    $\W$.
  \end{itemize}
  We say that the triple $(\C,\W,\F)$ of full subcategories of $\A$
  \emph{determines an abelian model structure on $\A$} if the triple
  $(\Cofib,\Weq,\Fib)$ is an abelian model structure, in which case $\C$ is the
  class of cofibrant objects, $\F$ is the class of fibrant objects, and $\W$ is
  the class of acyclic objects.
\end{definition}

\begin{remark}
  There is a close relationship between abelian model categories and complete
  cotorsion pairs, see~\cite[Theorem~2.2]{Hov02} for the precise statement. We
  only recall the following fact. Let $\A$ be an abelian category endowed with
  an abelian model category structure determined by a triple of full
  subcategories $(\C,\W,\F)$. Then, one has the following equalities of full
  subcategories of $A$:
  \begin{align*}
    \C&=\set{X\in\A}[\forall Y\in\W\cap\F,\ \Ext<\A>[1]{X}{Y}=0],\\
    \W\cap\F&=\set{Y\in\A}[\forall X\in\C,\ \Ext<\A>[1]{X}{Y}=0],\intertext{and}
    \C\cap\W&=\set{X\in\A}[\forall Y\in\F,\ \Ext<\A>[1]{X}{Y}=0],\\
    \F&=\set{Y\in\A}[\forall X\in\C\cap\W,\ \Ext<\A>[1]{X}{Y}=0].
  \end{align*}
  For a survey of the theory of abelian (and, more generally, exact) model
  categories, we refer the reader to~\cite{Sto13}.
\end{remark}

\begin{definition}
  \label{def;hereditary_abelian_model_category_structure}
  Let $\A$ be an abelian category with all small limits and all small colimits.
  An abelian model structure on $\A$ is \emph{hereditary} if the
  class $\C$ of cofibrant objects is closed under taking kernels of epimorphisms in $\C$
  and the class $\F$ closed under taking cokernels of monomorphisms in $\F$.
\end{definition}

\begin{remark}
  \label{rmk:injective-model-structures-are-hereditary}
  Every abelian model category structure in which every object is cofibrant is
  hereditary~\cite[Corollary 1.1.12]{Bec14}. The corresponding statement for
  abelian model category structures in which every object is fibrant also holds.
  The abelian model structures that we consider in this article satisfy at least
  one of these conditions and are therefore hereditary.
\end{remark}

Hereditary abelian model categories are a source of stable $\infty$-categories,
as the following result shows.

\begin{theorem}
  \label{thm:Gillespie}
  Let $\A$ be an abelian category with all small limits and all small colimits.
  Suppose that $\A$ is endowed with a hereditary abelian model structure. The
  following statements hold:
  \begin{enumerate}
  \item\cite[Proposition~5.2]{Gil11} The full subcategory $\C\cap\F\subseteq\A$
    is a Frobenius exact category with respect to the exact structure inherited
    from $\A$. Moreover, the class of projective-injective objects in $\C\cap\F$
    is precisely $\C\cap\W\cap\F$ and a morphism in $\C\cap\F$ is a weak
    equivalence if and only if it is a stable isomorphism.
  \item The underlying $\infty$-category
    \[
      \MLoc[\Weq]{\A}\simeq(\C\cap\F)[\Weq[st]^{-1}]
    \]
    is a stable $\infty$-category
    (see~\Cref{prop:Cisinski-Frobenius-stable-cat}), which is presentable if the
    model structure on $\A$ is combinatorial (see~\Cref{thm:MLoc-presentable}).
    Moreover, there is a canonical equivalence of triangulated categories
    \[
      \Ho{\MLoc{\A}}\simeq\underline{(\C\cap\F)},
    \]
    where the right-hand side denotes the stable category of the Frobenius exact
    category $\C\cap\F$.
  \item If the model structure on $\A$ is combinatorial, then the presentable
    stable $\infty$-category $\MLoc[\Weq]{\A}$ is compactly generated if and
    only if its homotopy category $\Ho{\MLoc{\A}}$ is a compactly-generated
    triangulated category (see~\cite[Remark 1.4.4.3]{Lur17}).
  \end{enumerate}
\end{theorem}

The following criterion for detecting \emph{monoidal} abelian model categories
is useful.

\begin{theorem}[{\cite[Proposition~7.2]{Hov02}}]
  \label{thm:Hovey-abelian_monoidal_model_cat}
  Let $\A$ be an abelian category with small limits and small colimits that
  endowed is with a monoidal structure. Suppose that $\A$ is endowed with an
  abelian model structure such that the following conditions are satisfied:
  \begin{enumerate}
  \item The unit $\unit\in\A$ is cofibrant.
  \item The tensor product $x\otimes y$ of a pair of cofibrant objects
    $x,y\in\A$ is cofibrant.
  \item The full subcategory $\C\cap\W\subseteq\A$ of cofibrant acyclic objects
    is a $\C-\otimes$-ideal: for a pair of objects $x,y\in\C$, if $x\in\C\cap \W$
    or $y\in \C\cap\W$, then $x\otimes y\in\C\cap\W$.
  \item If $f\colon x\to y$ is a cofibration then, for each object $z\in\A$, the
    morphism ${z\otimes f\colon z\otimes x\to z\otimes y}$ is a monomorphism.
  \end{enumerate}
  Then, $\A$ is a monoidal model category.
\end{theorem}

We record the following immediate corollary of
\Cref{thm:Hovey-abelian_monoidal_model_cat}.

\begin{corollary}
  \label{thm:Hovey-injective_abelian_monoidal_model_cat}
  Let $\A$ be an abelian category with small limits and small colimits that is
  endowed with a monoidal structure such that the tensor product functor
  \[
    -\otimes-\colon\A\times\A\longrightarrow\A
  \]
  is exact in each variable separately. Suppose that $\A$ is endowed with an
  abelian model structure in which every object is cofibrant. If the full
  subcategory $\W\subseteq\A$ of acyclic objects is a $\otimes$-ideal, then $\A$
  is a monoidal model category.
\end{corollary}

We now combine some of the previous recollections into a single result that we
use heavily in the sequel.

\begin{theorem}
  \label{thm:MLoc-presentably-stable-monoidal}
  Let $\M=(\M,\mathbf{1},\otimes)$ be a combinatorial monoidal hereditary
  abelian model category. The following statements hold:
  \begin{enumerate}
  \item The $\infty$-category $\MLoc{\M}$ is a presentably monoidal stable
    $\infty$-category, which is symmetric if the monoidal structure on $\M$ admits a
    symmetric refinement. In particular, the monoidal structure on $\MLoc{\M}$ is
    biclosed.
  \item The presentable $\infty$-category $\MLoc{\M}$ is compactly generated if
    and only if its homotopy category $\Ho{\MLoc{\M}}$ is a compactly-generated
    triangulated category.
  \end{enumerate}
\end{theorem}
\begin{proof}
  Combine \Cref{thm:MLoc-presentably-monoidal,thm:Gillespie}.
\end{proof}

\section{The base monoidal stable $\infty$-category}
\label{sec:base_category}

We fix a finite-dimensional Hopf algebra $H$. The tensor product of vector
spaces $\otimes=\otimes_\kk$ induces a monoidal structure on the category
\[
  \LMod{H}=\LMod{H}[\Mod{\kk}]
\]
of all left $H$-modules with monoidal unit the simple $H$-module $\kk$, which is a
symmetric monoidal structure if and only if $H$ is cocommutative, see for
example~\cite[Sections~1.8]{Mon93}. In particular, the tensor product functor
\[
  -\otimes-\colon\LMod{H}\times\LMod{H}\longrightarrow\LMod{H}
\]
is exact and preserves small colimits in each variable separately, so that $\LMod{H}$ is a presentably
monoidal category.

\begin{notation}
  \label{not:internal_hom_LMod_H}
  For formal reasons, the monoidal structure on $\LMod{H}$ is biclosed. In
  view of the adjunction isomorphisms
  \[
    \Hom<H>{H\otimes Y}{Z}\cong\Hom<\kk>{Y}{Z}\quad\text{and}\quad
    \Hom<H>{X\otimes H}{Y}\cong\Hom<\kk>{X}{Y},
  \]
  the underlying vector space of both the right and the left internal
  $\operatorname{Hom}$ objects is the vector space of $\kk$-linear maps; beware,
  however, that the two internal $\operatorname{Hom}$ objects are endowed with
  different left $H$-module structures since we do not assume that the Hopf
  algebra $H$ is cocommutative. In case there is a need to distinguish between
  these, we denote them by
  \[
    \Hom<\kk>[r]{Y}{Z}\in\LMod{H}\qquad\text{and}\qquad\Hom<\kk>[l]{X}{Y}\in\LMod{H},
  \]
  respectively.
\end{notation}

We also remind the reader that $H$ is a Frobenius algebra, so that the
Grothendieck category $\LMod{H}$ is a Frobenius abelian
category~\cite[Theorem~2.1.3]{Mon93}.

\begin{remark}
  We also need to consider the monoidal category $\LMod{H}<rev>$
  (\Cref{not:reversed_monoidal_cat}), which corresponds to the monoidal
  structure on $\LMod{H}$ associated to the Hopf algebra $H^\cop$ obtained from
  $H$ by reversing its comultiplication while preserving its multiplication (and
  which is equipped with the inverse
  antipode).
\end{remark}

We also need the following standard facts.

\begin{lemma}[{\cite[Proposition~2.3.10]{Wei94}}]
  \label{lemma:LR-preserving-proj-inj}
  Let $L\colon \A\rightleftarrows\B\colon R$ be an adjoint pair between abelian
  categories. If $R$ is exact, then $L$ preserves projective objects. Dually, if
  $L$ is exact, then $R$ preserves injective objects.
\end{lemma}

\begin{corollary}
  \label{coro:hom-injectives}
  Let $V\in\LMod{H}$. The functor
  \[
    -\otimes V\colon\LMod{H}\longrightarrow\LMod{H}
  \]
  preserves projective(-injective) $H$-modules. Dually, the functor
  \[
    \Hom<\kk>{V}{-}\colon\LMod{H}\longrightarrow\LMod{H}
  \]
  preserves (projective-)injective $H$-modules.
\end{corollary}

\subsection{The stable $\infty$-category of $H$-modules}
\label{subsec:StModH}

We recall the model-categorical approach to constructing the canonical
$\infty$-categorical enhancement of the large stable category of $H$-modules.

\begin{theorem}[{\cite[Theorem~2.2.12]{Hov99} and \cite[Theorem~9.5]{Hov02}}]
  \label{thm:Hovey-H-Mod}
  The triple
  \[
    (\LMod{H},\Inj{H},\LMod{H}),
  \]
  where $\Inj{H}\subseteq\LMod{H}$ is the full subcategory spanned by the
  injective $H$-modules, determines a combinatorial monoidal hereditary abelian
  model category structure on $\LMod{H}$ whose class $\Weq[st]$ of weak
  equivalences is that of stable isomorphisms. In particular, every $H$-module
  is cofibrant and fibrant in this model structure.
\end{theorem}
\begin{proof}
  The existence of the claimed combinatorial abelian model category structure on
  $\LMod{H}$ is a special case of~\cite[Theorem~2.2.12]{Hov99}. That this model
  structure is monoidal is shown in~\cite[Proposition~9.5]{Hov02} in the special
  case of group algebras of finite groups over a principal ideal domain, but the
  proof applies in this setting and is, in fact, simpler since we work over a
  field. Indeed, according to
  \Cref{thm:Hovey-injective_abelian_monoidal_model_cat}, it is enough to show
  that the subcategory $\Inj{H}\subseteq\LMod{H}$ is a $\otimes$-ideal, which
  follows from \Cref{coro:hom-injectives}.
\end{proof}

\Cref{thm:Hovey-H-Mod,thm:MLoc-presentably-stable-monoidal} permit us to make
the following definition.

\begin{definition}
  \label{def:StModH}
  The \emph{stable $\infty$-category of $H$-modules} is the presentably monoidal
  stable $\infty$-category
  \[
    \StLMod{H}\coloneqq\MLoc[\Weq[st]]{\LMod{H}}=(\LMod{H})[\Weq[st]^{-1}].
  \]
  There is a canonical equivalence of triangulated categories
  \[
    \Ho{\StLMod{H}}\simeq\LModu{H},
  \]
  where the right-hand side denotes the large stable category of $H$-modules. In
  particular, the $\infty$-category $\StLMod{H}$ is compactly generated by
  \cite[Corollary~5.4]{Kra05} and the tensor product restricts to the full
  subcategory
  \[
    \stLmod{H}\coloneqq\StLMod{H}^\omega\subseteq\StLMod{H}
  \]
  spanned by the compact objects, which we may identify with the
  finite-dimensional $H$-modules. Notice also that $\stLmod{H}$ is the thick
  subcategory of $\StLMod{H}$ generated by the simple $H$-modules.
\end{definition}

We also need the following alternative construction of the stable
$\infty$-category $\StLMod{H}$. As usual, we denote the category of cochain
complexes of $H$-modules by $\Ch{\LMod{H}}$. This is a monoidal category with
the Day convolution tensor product~\cite{Day70}
\[
  \textstyle(X\otimes Y)^k\coloneqq\bigoplus_{i+j=k}X^i\otimes Y^j,\qquad
  X,Y\in\Ch{\LMod{H}},
\]
which is symmetric if and only if the Hopf algebra $H$ is cocommutative. This
monoidal structure is biclosed: The right internal $\operatorname{Hom}$ object
is given by the graded left $H$-module (see~\Cref{not:internal_hom_LMod_H})
\[
  \textstyle\hom<\kk>[r]{Y}{Z}^i\coloneqq\prod_{j\in\ZZ}\Hom<\kk>[r]{Y^{j}}{Z^{i+j}},\qquad
  Y,Z\in\Ch{\LMod{H}},
\]
that is endowed with the differential
\[
  \partial(f)\coloneqq d_Y\circ f-(-1)^if\circ d_X,\qquad f\in\hom<k>[r]{Y}{Z}^i.
\]
The definition of the left internal $\operatorname{Hom}$ object is entirely
analogous. The biclosed monoidal structure descends to the homotopy category
of cochain complexes of $H$-modules since it is given by $\kk$-linear functors.

We also let
$\Ch{\Inj{H}}$ be the category of cochain complexes of injective $H$-modules and
$\Ch<ac>{\Inj{H}}$ its full subcategory of acyclic complexes. In the proof
below, we use the following observation.

\begin{lemma}
  \label{lemma:internal_hom_preserves_acyclics_inj}
  If $Y\in\Ch{\LMod{H}}$ and $Z\in\Ch<ac>{\LMod{H}}$, then
  $\hom<\kk>[r]{Y}{Z}\in\Ch<ac>{\LMod{H}}$. Similarly, if
  $Z\in\Ch<ac>{\Inj{H}}$, then $\hom<\kk>[r]{Y}{Z}\in\Ch<ac>{\Inj{H}}$. The
  analogous statements for left internal $\operatorname{Hom}$ objects also hold.
\end{lemma}
\begin{proof}
  If $Z$ is acyclic, then it is contractible as a complex of vector spaces.
  Hence $\hom<\kk>[r]{Y}{Z}$ is also acyclic. If $Z$ is a
  complex of injective $H$-modules, \Cref{coro:hom-injectives} implies that
  $\hom<\kk>[r]{Y}{Z}\in\Ch{\Inj{H}}$.
\end{proof}

\begin{theorem}[Becker]
  \label{thm:Becker-icosing}
  The following statements hold:
  \begin{enumerate}
  \item \cite[Proposition~2.2.1]{Bec14} The category $\Ch{\LMod{H}}$ is endowed
    with a combinatorial monoidal hereditary abelian model category structure
    determined by the triple
    \[
      (\Ch{\Mod{H}},\W[co,sing],\Ch<ac>{\Inj{H}}),
    \]
    where
    \begin{center}
      $\W[co,sing]\coloneqq\set{X\in\Ch{\LMod{H}}}[\forall
      Y\in\Ch<ac>{\Inj{H}},\ \Ext<H>[1]{X}{Y}=0]$
    \end{center}
    is the full subcategory of \emph{singular coacyclic} complexes. In
    particular, every chain complex of $H$-modules is cofibrant in this model
    structure. We denote the class of weak equivalences of this model structure
    by $\Weq[co,sing]$.
  \item\cite[Proposition~3.1.5]{Bec14} The functor of $0$-cocycles
    \[
      Z^0\colon\Ch{\LMod{H}}\longrightarrow\LMod{H}
    \]
    is a right Quillen equivalence, where the target is endowed with the model
    structure from \Cref{thm:Hovey-H-Mod}. Finally, the induced equivalence of
    underlying $\infty$-categories
    \[
      \MLoc[\Weq[co,sing]]{\Ch{\LMod{H}}}\stackrel{\simeq}{\longrightarrow}\StLMod{H}
    \]
    is monoidal.
  \end{enumerate}
\end{theorem}
\begin{proof}
  The only statements that are not contained in~\cite{Bec14} are those
  concerning the monoidal structure. In view
  of~\Cref{thm:Hovey-injective_abelian_monoidal_model_cat}, in order to show
  that the model structure is monoidal, it is enough to show that the class
  $\W[co,sing]\subseteq\Ch{\LMod{H}}$ is a $\otimes$-ideal. Indeed, let
  $X\in\W[co,sing]$, $Y\in\Ch{\Mod{R}}$, and $Z\in\Ch<ac>{\Inj{H}}$. Then,
  \begin{align*}
    \Ext<H>[1]{X\otimes Y}{Z}&\cong[X\otimes
                               Y,Z[1]]\\&\cong[X,\hom<\kk>[r]{Y}{Z}[1]]\\&\cong\Ext<H>[1]{X}{\hom<\kk>[r]{Y}{Z}}\stackrel{!}{=}0,
  \end{align*}
  where $[X,Y]$ denotes the vector space of homotopy classes of morphisms of
  complexes of $H$-modules. The required vanishing condition then follows from
  the assumption $X\in\W[co,sing]$ together with the fact that
  $\hom<\kk>[r]{Y}{Z}$ is an acyclic complex of injective $H$-modules
  (\Cref{lemma:internal_hom_preserves_acyclics_inj}). The case in which the
  roles of $X$ and $Y$ are exchanged is dealt with analogously, using the left
  internal $\operatorname{Hom}$ object. That the induced equivalence
  \[
    \MLoc[\Weq[co,sing]]{\Ch{\LMod{H}}}\stackrel{\simeq}{\longrightarrow}\StLMod{H}
  \]
  is monoidal follows immediately from the fact that the left adjoint to the
  functor $X\mapsto Z^0(X)$, namely the inclusion
  $\LMod{H}\hookrightarrow\Ch{\LMod{H}}$ into the complexes concentrated in
  degree $0$, is clearly monoidal, and hence it is monoidal left Quillen
  equivalence (recall that all objects in $\Ch{\LMod{H}}$ and $\LMod{H}$ are
  cofibrant with respect to the model structures being considered).
\end{proof}

\begin{notation}
  \label{not:cosing}
  We write $\Ch{\LMod{H}}<co,sing>$ to indicate that we consider the category
  $\Ch{\LMod{H}}$ as a model category equipped with the model structure
  from~\Cref{thm:Becker-icosing}.
\end{notation}

\begin{variant}
  \label{variant:graded}
  Following the discussion in~\cite[Remark~1.3.13]{Bec14}, in
  \Cref{thm:Becker-icosing}, we may replace the category of vector spaces by the
  category of $\ZZ$-graded vector spaces and consider a finite-dimensional Hopf
  algebra therein (compare also with~\cite[Remark~3.7]{Qi14}).
\end{variant}

\subsection{The derived $\infty$-category of $H$}

We now recall the model-categorical approach to constructing the derived
$\infty$-category of $H$.

\begin{proposition}[{\cite[Theorem~2.3.13]{Hov99}, \cite[Example~3.2]{Hov02}}]
  \label{prop:injective-model-str-Ch}
  The category $\Ch{\LMod{H}}$ is endowed with a combinatorial monoidal
  hereditary abelian model structure determined by the triple
  \[
    (\Ch{\LMod{H}},\Ch<ac>{\LMod{H}},\Ch<fib>{\LMod{H}}).
  \]
  Here, $\Ch<ac>{\LMod{H}}\subseteq\Ch{\LMod{H}}$ is the full subcategory
  spanned by the acyclic complexes and
  \[
    \Ch<fib>{\LMod{H}}\coloneqq\set{Y\in\Ch{\LMod{H}}}[\forall
    X\in\Ch<ac>{\LMod{H}},\ \Ext<H>[1]{X}{Y}=0]
  \]
  is the full subcategory of \emph{DG injective complexes}. In particular, every
  chain complex of $H$-modules is cofibrant in this model structure. The class
  of weak equivalences of this model structure is the class $\Weq[qis]$ of
  quasi-isomorphisms.
\end{proposition}
\begin{proof}
  That the model structure on $\Ch{\LMod{H}}$ is monoidal follows immediately
  from \Cref{thm:Hovey-injective_abelian_monoidal_model_cat} and the fact that
  the tensor product of acyclic complexes of $\kk$-modules is again acyclic.
\end{proof}

\begin{notation}
  \label{not:inj}
  We write $\Ch{\LMod{H}}<inj>$ to indicate that we consider the category
  $\Ch{\LMod{H}}$ as a model category equipped with the model structure
  from~\Cref{prop:injective-model-str-Ch}.
\end{notation}

\Cref{prop:injective-model-str-Ch} and
\Cref{thm:MLoc-presentably-stable-monoidal} permit us to make the following
definition.

\begin{definition}
  The \emph{derived $\infty$-category} of $H$ is the presentably monoidal stable
  $\infty$-category
  \[
    \DerCat{\LMod{H}}\coloneqq\MLoc[\Weq[qis]]{\Ch{\LMod{H}}}=\Ch{\LMod{H}}[\Weq[qis]^{-1}].
  \]
  The homotopy category $\Ho{\DerCat{\LMod{H}}}$ is the ordinary derived
  category of $H$-modules. In particular, the $\infty$-category
  $\DerCat{\LMod{H}}$ is compactly generated and the tensor product restricts to
  the subcategory
  \[
    \perf{H}\coloneqq\DerCat{\LMod{H}}^\omega\subseteq\DerCat{\LMod{H}}
  \]
  spanned by the compact objects, that we may identify with bounded complexes of
  finite-dimensional projective $H$-modules. Beware, however, that the unit
  object ${\kk\in\DerCat{\LMod{H}}}$ is not compact unless $H$ is semisimple.
  Hence, in general, $\perf{H}$ is only a non-unital (!) monoidal
  $\infty$-category.
\end{definition}

\begin{remark}
  \label{rmk:forgetful_DH}
  The (monoidal) forgetful functor
  \[
    U\colon\Ch{\LMod{H}}<inj>\longrightarrow\Ch{\Mod{\kk}}<inj>
  \]
  is a monoidal left Quillen functor. Indeed, since the functor $U$ is exact and
  every object of the domain and of the codomain is cofibrant, it suffices to
  observe that $U$ sends acyclic complexes of $H$-modules to acyclic complexes
  of vector spaces. Therefore, the induced functor
  between derived $\infty$-categories
  \[
    \DerCat{\LMod{H}}\longrightarrow\DerCat{\Mod{\kk}}
  \]
  is also monoidal.
\end{remark}

\subsection{The coderived $\infty$-category of $H$}

We now recall the model-categorical approach to constructing the
$\infty$-categorical enhancement of the homotopy category of cochain complexes
of injective $H$-modules.

\begin{proposition}[{\cite[Propositions~1.3.6 and 1.3.8]{Bec14}}]
  \label{prop:Becker-ico}
  The category $\Ch{\LMod{H}}$ admits a combinatorial monoidal hereditary
  abelian model structure determined by the triple
  $(\Ch{\LMod{H}},\W[co],\Ch{\Inj{H}})$, where
  \begin{align*}
    \W[co]&\coloneqq\set{X\in\Ch{\LMod{H}}}[\forall Y\in\Ch{\Inj{H}},\ \Ext<H>[1]{X}{Y}=0]
  \end{align*}
  is the class of \emph{coacyclic complexes}. In particular, every complex of
  $H$-modules is cofibrant in this model structure. We denote the class of weak
  equivalences of this model structure by $\Weq[co]$.
\end{proposition}
\begin{proof}
  The proof that this model structure is monoidal is entirely analogous to that
  of \Cref{thm:Becker-icosing}.
\end{proof}

\begin{notation}
  \label{not:co}
  We write $\Ch{\LMod{H}}<co>$ to indicate that we consider the category
  $\Ch{\LMod{H}}$ as a model category equipped with the model structure
  from~\Cref{prop:Becker-ico}.
\end{notation}

\Cref{prop:Becker-ico} and \Cref{thm:MLoc-presentably-stable-monoidal} permit us
to make the following definition, where we borrow the terminology used by
Positselski in~\cite{Pos11} to avoid the awkward terminology `homotopy
$\infty$-category.'

\begin{definition}
  The \emph{coderived $\infty$-category} of $H$ is the presentably monoidal
  stable $\infty$-category
  \[
    \KCh{\Inj{H}}\coloneqq\MLoc[\Weq[co]]{\Ch{\LMod{H}}}=\Ch{\LMod{H}}[\Weq[co]^{-1}].
  \]
  The homotopy category $\Ho{\KCh{\Inj{H}}}$ is equivalent to the ordinary
  homotopy category of cochain complexes of injective $H$-modules. In
  particular, the $\infty$-category $\KCh{\Inj{H}}$ is compactly generated and
  the homotopy category $\Ho{\KCh{\Inj{H}}^\omega}$ of the full subcategory
  \[
    \DerCat[b]{\mmod{H}}\simeq\KCh{\Inj{H}}^\omega\subseteq\KCh{\Inj{H}}
  \]
  spanned by the compact objects is equivalent to the bounded derived category
  of finite-dimensional $H$-modules~\cite[Proposition~2.3]{Kra05}. Notice that
  $\DerCat[b]{\mmod{H}}$ is generated as an idempotent-complete stable
  $\infty$-category by the simple $H$-modules. The tensor product on
  $\KCh{\Inj{H}}$ restricts to the full subcategory of compact objects, which
  contains the monoidal unit. Therefore $\KCh{\Inj{H}}^\omega$ is a monoidal
  $\infty$-category.
\end{definition}

\begin{remark}
  \label{rmk:identity_KInj}
  The (monoidal) forgetful functor
  \[
    U\colon\Ch{\LMod{H}}<co>\longrightarrow\Ch{\Mod{\kk}}<inj>
  \]
  is a monoidal left Quillen functor. Indeed, since the functor $U$ is exact and
  every object of the domain and of the codomain is cofibrant, it suffices to
  observe that $U$ sends the subcategory $\W[co]\subseteq\Ch{\LMod{H}}$ into the
  subcategory of acyclic complexes
  of vector spaces~\cite[Corollary~1.3.7]{Bec14}. Therefore, the induced functor
  between $\infty$-categories
  \[
    \KCh{\LMod{H}}\longrightarrow\DerCat{\Mod{\kk}}
  \]
  is also monoidal.
\end{remark}

\subsection{Interlude: $\kk$-linear structures}
\label{subsec:linear_structures}

The counit map $H\to\kk$ induces an exact monoidal functor
\[
  \Ch{\Mod{\kk}}\longrightarrow\Ch{\LMod{H}}
\]
that preserves contractible complexes. This functor is a monoidal left Quillen
functor with respect to the injective model structure on $\Ch{\Mod{\kk}}$
(\Cref{prop:injective-model-str-Ch}) and any of the model structures on
$\Ch{\LMod{H}}$ described in \Cref{thm:Becker-icosing} and
\Cref{prop:injective-model-str-Ch,prop:Becker-ico}. Indeed, it suffices to
observe the following facts:
\begin{itemize}
\item Every complex of vector spaces cofibrant.
\item Acyclic complexes of vector spaces are contractible.
\item The (trivial) cofibrations in $\Ch{\Mod{\kk}}$ are the monomorphisms (with
  contractible cokernel).
\item The acyclic objects (in the sense of model categories) in any of the three
  model structures on $\Ch{\Mod{H}}$ that we consider clearly contain the
  contractible complexes.
\end{itemize}
Thus, passing to underlying $\infty$-categories we obtain colimit-preserving
monoidal functors
\begin{align*}
  \DerCat{\Mod{\kk}}&\longrightarrow\StLMod{H},\\
  \DerCat{\Mod{\kk}}&\longrightarrow\KCh{\Inj{H}},\\
  \DerCat{\Mod{\kk}}&\longrightarrow\DerCat{\LMod{H}}
\end{align*}
Moreover, the first two functors restrict to exact monoidal functors
\[
  \perf{\kk}\longrightarrow\stLmod{H}\qquad\text{and}\qquad\perf{\kk}\longrightarrow\DerCat[b]{\mmod{H}}.
\]
All of these functors are symmetric monoidal when the Hopf algebra $H$ is
cocommutative. It follows from \Cref{prop:change_of_algebra} that, by
restriction along the corresponding monoidal functor, we may regard
$\StLMod{H}$, $\KCh{\Inj{H}}$ and $\DerCat{\LMod{H}}$ as
$\DerCat{\Mod{\kk}}$-linear categories. We regard $\stLmod{H}$, and
$\DerCat[b]{\mmod{H}}$ as $\perf{\kk}$-linear categories in the same way.

\begin{remark}
  \label{rmk:section_identity_DH}
  According to~\cite[Corollary~3.4.1.7]{Lur17}, there is an equivalence of
  $\infty$-categories
  \[
    \Alg{\Mod{\DerCat{\Mod{\kk}}}[\PrStL]}\stackrel{\sim}{\longrightarrow}\Alg{\PrStL}_{\DerCat{\Mod{\kk}}/},
  \]
  where $\Alg{\PrStL}_{\DerCat{\Mod{\kk}}/}$ denotes the $\infty$-category of
  algebras under $\DerCat{\Mod{\kk}}$. Thus, via the canonical
  colimit-preserving monoidal functor
  \[
    \DerCat{\Mod{\kk}}\longrightarrow\DerCat{\LMod{H}},
  \]
  we may regard
  \[
    \DerCat{\LMod{H}}\in\Alg{\Mod{\DerCat{\Mod{\kk}}}[\PrStL]}
  \]
  as an algebra object of the $\infty$-category
  $\Mod{\DerCat{\Mod{\kk}}}[\PrStL]$ of $\kk$-linear presentable stable
  $\infty$-categories. Combined with \Cref{rmk:forgetful_DH}, we conclude that the forgetful functor
  $U\colon\DerCat{\LMod{H}}\to\DerCat{\kk}$ fits into a commutative diagram of
  colimit-preserving monoidal functors
  \[
    \begin{tikzcd}[column sep=tiny]
      &\DerCat{\Mod{\kk}}\drar{\id}\dlar\\
      \DerCat{\LMod{H}}\ar{rr}[swap]{U}&&\DerCat{\Mod{\kk}}
    \end{tikzcd}
  \]
  If we identify $\DerCat{\Mod{\kk}}$ with the initial object of the
  $\infty$-category $\Alg{\PrStL}_{\DerCat{\Mod{\kk}}/}$ determined by the
  identity functor $\id\colon\DerCat{\Mod{\kk}}\to\DerCat{\Mod{\kk}}$, the
  previous commutative diagram permits us to regard $\DerCat{\LMod{H}}$ as an
  \emph{augmented} algebra object of $\Mod{\DerCat{\Mod{\kk}}}[\PrStL]$.
  Similarly, we may regard $\KCh{\Inj{H}}$ as an augmented algebra object of
  $\Mod{\DerCat{\Mod{\kk}}}[\PrStL]$, see \Cref{rmk:identity_KInj}.
\end{remark}

\subsection{Krause's recollement}

The following result is a special case of the main theorem in~\cite{Kra05}.

\begin{theorem}[Krause]
  \label{thm:Krauses-recollement}
  The following statements hold:
  \begin{enumerate}
  \item\cite[Theorem~4.2]{Kra05} There is a recollement of compactly-generated
    stable $\infty$-categories
    \begin{center}
      \begin{tikzcd}
        \StLMod{H}\ar[hookrightarrow]{r}[description]{i}&\KCh{\Inj{H}}\ar{r}[description]{p}\ar[shift
        right=0.5em]{l}[swap]{i_L}\ar[shift
        left=0.5em]{l}{i_R}&\DerCat{\LMod{H}}.\ar[shift
        right=0.5em,hook]{l}[swap]{p_L}\ar[hook',shift left=0.5em]{l}{p_R}
      \end{tikzcd}
    \end{center}
    Explicitly, this means the following:
    \begin{itemize}
    \item The are adjoint triples $i_L\dashv i\dashv i_R$ and $p_L\dashv p\dashv
      p_R$.
    \item The functors $i$, $p_L$ and $p_R$ are fully faithful.
    \item The essential image of $i$ is precisely the kernel of $p$.
    \end{itemize}
    Moreover, in the above recollement,
    \begin{itemize}
    \item the essential image of $i$ is the two-sided $\otimes$-ideal of
      $\KCh{\Inj{H}}$ spanned by the acyclic complexes of injective $H$-modules,
    \item the essential image of $p_R$ is the full subcategory of
      $\KCh{\Inj{H}}$ spanned by the $K$-injective complexes of $H$-modules in
      the sense of~\cite{Spa88}, and
    \item the essential image of $p_L$ is a two-sided $\otimes$-ideal in
      $\KCh{\Inj{H}}$.
    \end{itemize}
  \item\cite[Corollary~5.4]{Kra05} The above recollement restricts to a
    localisation sequence
    \begin{center}
      \begin{tikzcd}
        \stLmod{H}&\DerCat[b]{\mmod{H}}\ar{l}[swap]{i_L}&\perf{H}\ar[hook']{l}[swap]{p_L}
      \end{tikzcd}
    \end{center}
    between the corresponding full subcategories of compact objects. That is,
    the canonical exact functor
    \begin{center}
      $(\DerCat[b]{\mmod{H}}/\perf{H})^\flat\stackrel{\sim}{\longrightarrow}\stLmod{H}$
    \end{center}
    is an equivalence, where $\C\mapsto\C^\flat$ denotes the passage to the
    idempotent-completion.\footnote{See~\cite[Section~5.1.4]{Lur09}
      and~\cite[Corollary~1.1.3.7]{Lur17} for information on idempotent
      completions of (stable) $\infty$-categories.}
  \end{enumerate}
\end{theorem}
\begin{proof}
  In much broader generality, Krause constructs the required recollement at the
  level of the corresponding triangulated homotopy
  categories~\cite[Theorem~4.2]{Kra05}. In~\cite[Corollary~2.2.2]{Bec14}, Becker
  provides an alternative construction of Krause's recollement by leveraging
  several Quillen adjunctions between model categories, including those model
  categories that we use to construct the three stable $\infty$-categories
  involved. From this it follows that Krause's recollement lifts to the
  $\infty$-categorical level. The existence of the restricted
  localisation sequence can be verified at the level of triangulated homotopy
  categories~\cite[Proposition~5.15]{BGT13}, and hence its existence follows
  from~\cite[Corollary~5.4]{Kra05}.
  
  That the essential image of $i$ is a two-sided $\otimes$-ideal in
  $\KCh{\Inj{H}}$ follows from \Cref{coro:hom-injectives} and that the
  contractible complexes of vector spaces form a two-sided $\otimes$-ideal of
  $\Ch{\Mod{\kk}}$. Finally, the claim essential image of $p_L$ is a two-sided
  $\otimes$-ideal in $\KCh{\Inj{H}}$ follows from \Cref{coro:hom-injectives}
  and the characterisation of the essential image as the full subcategory
  spanned by all complexes $X\in\KCh{\Inj{H}}$ such that $[X,Y]=0$ for each
  acyclic complex of injective $H$-modules (in fact, we have shown this in the
  proof of~\Cref{thm:Becker-icosing}).
\end{proof}

\begin{remark}
  \label{rmk:Beckers_proof}
  The proof of~\cite[Corollary~2.2.2]{Bec14} relies, in part, on the fact that
  the identity functor of $\Ch{\LMod{H}}$ determines Quillen adjunctions as
  indicated below~\cite[Corollary~1.3.7, pp.~216--217]{Bec14}:
  \[
    \begin{tikzcd}[column sep=large]
      \Ch{\LMod{H}}<co,sing>\ar{r}[description]{\mathbb{R}\id}&\Ch{\LMod{H}}<co>\ar{r}[description]{\mathbb{L}\id}\ar[shift
      right=0.5em]{l}[swap]{\mathbb{L}\id}&\Ch{\LMod{H}}<inj>\ar[shift
      left=0.5em]{l}{\mathbb{R}\id}.
    \end{tikzcd}
  \]
  After passing to underlying $\infty$-categories, these induce the adjoint pairs
  \[
    \begin{tikzcd}
        \StLMod{H}\ar[hookrightarrow]{r}[description]{i}&\KCh{\Inj{H}}\ar{r}[description]{p}\ar[shift
        right=0.5em]{l}[swap]{i_L}&\DerCat{\LMod{H}},\ar[hook',shift left=0.5em]{l}{p_R}
      \end{tikzcd}
    \]
    which fully determine the recollement as soon as one shows that the
    essential image if $i$ is precisely the kernel of $p$.
\end{remark}

\section{Hopfological derived $\infty$-categories}
\label{sec:Hopfological_dercats}

As in \Cref{subsec:StModH}, we fix a finite-dimensional Hopf algebra $H$. We
also fix an algebra object $A\in\Alg{\LMod{H}}$, that is an associative and
unital $\kk$-algebra $A$ such that the multiplication and unit morphisms
\[
  A\otimes A\longrightarrow A\qquad\text{and}\qquad \kk\longrightarrow A
\]
are left $H$-module homomorphisms. The category of $\LMod{A}[\LMod{H}]$ of left
$A$-modules internal to the category of left $H$-modules is isomorphic to the
Grothendieck category $\Mod{\smashprod{A}{H}}$ of left modules over the
\emph{smash product algebra} $\smashprod{A}{H}$~\cite[Definition~4.1.3]{Mon93},
see for example~\cite[Exercise~7.8.32]{EGNO15} (the precise definition of the
smash product algebra does not play an explicit role in the sequel). In
particular, there is the forgetful functor
\[
  \LMod{\smashprod{A}{H}}\cong\LMod{A}[\LMod{H}]\longrightarrow\LMod{H}
\]
which admits an exact left adjoint
\[
  A\otimes-\colon\LMod{H}\longrightarrow\LMod{\smashprod{A}{H}}
\]
and exact right adjoint
\[
  \Hom<\kk>{A}{-}\colon\LMod{H}\longrightarrow\LMod{\smashprod{A}{H}}.
\]

\begin{definition}
  The \emph{derived $\infty$-category} of $\smashprod{A}{H}$ is the underlying
  $\infty$-category
  \[
    \DerCat{\LMod{\smashprod{A}{H}}}\coloneqq\MLoc[\Weq[qis]]{\Ch{\LMod{\smashprod{A}{H}}}},
  \]
  of the (injective) model category structure on $\Ch{\LMod{\smashprod{A}{H}}}$
  described in~\cite[Theorems~2.3.13]{Hov99}, whose weak equivalences
  are the quasi-isomorphisms. The homotopy category
  $\Ho{\DerCat{\LMod{\smashprod{A}{H}}}}$ is the ordinary derived category of
  $\smashprod{A}{H}$, and therefore the presentable $\infty$-category
  $\DerCat{\LMod{\smashprod{A}{H}}}$ is compactly generated. The full
  subcategory
  \[
    \perf{\smashprod{A}{H}}\coloneqq\DerCat{\LMod{\smashprod{A}{H}}}^\omega\subseteq\DerCat{\LMod{\smashprod{A}{H}}}
  \]
  of compact objects is spanned by the bounded complexes of finitely-generated
  projective ($\smashprod{A}{H}$)-modules.
\end{definition}

Recall, that if \(A\in \Alg{\LMod{H}}\) is an algebra object in \(\LMod{H}\), then it canonically defines an algebra structure on \(A[0]\in  \Ch{\LMod{H}}\). In particular, we maybe consider \(A\) as an algebra object of \(\DerCat{\LMod{H}}\) by \cref{thm:MLoc-presentably-monoidal}. We record the following result for later use.

\begin{proposition}
  \label{prop:der-cat-smash}
  There is an equivalence of $\infty$-categories
  \[
    \DerCat{\LMod{\smashprod{A}{H}}}\simeq\LMod{A}[\DerCat{\LMod{H}}]
  \]
  between the derived $\infty$-category of the smash product algebra
  $\smashprod{A}{H}$ and the $\infty$-category of left $A$-modules internal to
  the derived $\infty$-category of $H$. Consequently, the presentable
  $\infty$-category $\DerCat{\LMod{\smashprod{A}{H}}}$ is right-tensored over
  $\DerCat{\LMod{H}}$.
\end{proposition}
\begin{proof}
  Recall that
  \[
    \DerCat{\LMod{H}}=\MLoc[\Weq[qis]]{\Ch{\LMod{H}}},
  \]
  where the category $\Ch{\LMod{H}}$ is equipped with the model structure
  described in \Cref{prop:injective-model-str-Ch}. The required equivalence of
  $\infty$-categories is then a direct consequence of
  \Cref{thm:Lurie-transfer-model_structure}:
  \begin{align*}
    \DerCat{\LMod{\smashprod{A}{H}}}&=\MLoc[\Weq[qis]]{\Ch{\LMod{\smashprod{A}{H}}}}\\&\cong\MLoc[\Weq[qis]]{\LMod{A}[\Ch{\LMod{H}}]}\\&\simeq\LMod{A}[\MLoc[\Weq[qis]]{\Ch{\LMod{H}}}]\\&=\LMod{A}[\DerCat{\LMod{H}}].
  \end{align*}
  The fact that the presentable $\infty$-category
  $\DerCat{\LMod{\smashprod{A}{H}}}$ is right-tensored over $\DerCat{\LMod{H}}$
  follows from \Cref{thm:props_of_Mod}.
\end{proof}

\subsection{The Hopfological derived $\infty$-category of $A$}

In this section we revisit Ohara's model-categorical approach to constructing
Hopfological derived categories from an $\infty$-categorical perspective.

\begin{notation}
  \label{not:W-C-Hopfologica-der-cat}
  We let $\W\subseteq\LMod{\smashprod{A}{H}}$ be the full subcategory of objects
  whose underlying $H$-module is (projective-)injective and define
  \[
    \C\coloneqq\set{X\in\LMod{\smashprod{A}{H}}}[\forall Y\in\W,\
    \Ext<\smashprod{A}{H}>[1]{X}{Y}=0].
  \]
\end{notation}

\begin{remark}
  \label{rmk:W}
  The subcategory $\W$ is closed under direct summands and it satisfies the
  2-out-of-3 property: given a short exact sequence
  \[
    0\longrightarrow X\longrightarrow Y\longrightarrow Z\longrightarrow0
  \]
  in $\LMod{\smashprod{A}{H}}$, if two of the modules $X$, $Y$ and $Z$ lie in
  $\W$, so does the third.
\end{remark}

\begin{proposition}
  \label{prop:AH-CW-cotorsion-pair}
  There is an equality
  \[
    \W=\set{Y\in\LMod{\smashprod{A}{H}}}[\forall S\in\LMod{H}\text{: simple},\
    \Ext<\smashprod{A}{H}>[1]{A\otimes S}{Y}=0]
  \]
  of subcategories of $\LMod{\smashprod{A}{H}}$. Consequently, $(\C,\W)$ forms a
  complete cotorsion pair in $\LMod{\smashprod{A}{H}}$ in the sense
  of~\cite{Sal79} and \cite[Definition~2.3]{Hov02}.
\end{proposition}
\begin{proof}
  Before beginning the proof, we remind the reader of the elementary equality
  \[
    \Inj{H}=\set{Y\in\LMod{H}}[\forall S\in\LMod{H}\text{: simple},\
    \Ext<H>[1]{S}{Y}=0],
  \]
  which can be verified as follows: By~\cite[Baer's Criterion~2.3.1]{Wei94}, an
  $H$-module $Y$ is injective if and only if $\Ext<H>[1]{X}{Y}=0$ for every
  quotient $X$ of $H$. In particular, $\Ext<H>[1]{S}{Y}=0$ for each simple
  $H$-module $S$ and, since every quotient of $H$ admits a composition series,
  the desired equality follows.

  The rest of the proof is entirely analogous to that of~\cite[Lemma~4.3 and
  Theorem~4.4]{HJ22}, but it is much easier since we work over a field.
  By \Cref{lemma:LR-preserving-proj-inj}, the exact
  functor
  \[
    A\otimes-\colon\LMod{H}\longrightarrow\LMod{\smashprod{A}{H}}
  \]
  send projective $H$-modules to projective ($\smashprod{A}{H}$)-modules, and
  hence it preserves projective resolutions. Let $V\in\LMod{H}$ and
  $Y\in\LMod{\smashprod{A}{H}}$. Choose a projective resolution $P_\bullet$ of
  $V$ so that $A\otimes P_\bullet$ is a projective resolution of $A\otimes V$.
  Then, for $i\in\ZZ$,
  \begin{align*}
    \Ext<\smashprod{A}{H}>[i]{A\otimes V}{Y}&\cong\H[i]{\Hom<\smashprod{A}{H}>{A\otimes P_\bullet}{Y}}\\&\cong\H[i]{\Hom<H>{P_\bullet}{Y}}\cong\Ext<H>[i]{V}{Y}.
  \end{align*}
  The required equality
  \[
    \W=\set{Y\in\LMod{\smashprod{A}{H}}}[\forall S\in\LMod{H} \text{: simple},\
    \Ext<\smashprod{A}{H}>[1]{A\otimes S}{Y}=0]
  \]
  is now clear. Finally, that $(\C,\W)$ forms a complete cotorsion pair in
  $\LMod{\smashprod{A}{H}}$ now follows from~\cite[Corollary~1.2.2]{Bec14}, since
  it is the cotorsion pair cogenerated by the (finite) set of objects
  $\set{A\otimes S}[S\in\LMod{H}\text{: simple}]$.
\end{proof}

The existence of the model structure described in the following theorem was
established by Ohara in~\cite{Oha24} by direct methods.

\begin{theorem}
  \label{thm:Ohara}
  The following statements hold:
  \begin{enumerate}
  \item\cite[Theorem~1.1]{Oha24} The Grothendieck category
    \[
      \LMod{\smashprod{A}{H}}\cong\LMod{A}[\LMod{H}]
    \]
    admits a combinatorial model structure determined as follows:
    \begin{itemize}
    \item A morphism in $\LMod{\smashprod{A}{H}}$ is a weak equivalence if its
      underlying morphism of $H$-modules is a stable isomorphism.
    \item A morphism in $\LMod{\smashprod{A}{H}}$ is a fibration if its
      underlying morphism of $H$-modules is an epimorphism.
    \item A morphism in $\LMod{\smashprod{A}{H}}$ is a cofibration if it has the
      left lifting property with respect to the trivial fibrations.
    \end{itemize}
    Moreover, the forgetful functor $\LMod{\smashprod{A}{H}}\to\LMod{H}$ is both
    a left Quillen functor and a right Quillen functor.
  \item The previous model structure is, in fact, the hereditary abelian model
    structure determined by the triple $(\C,\W,\LMod{\smashprod{A}{H}})$.
  \end{enumerate}
  In particular, in this model structure, every object is fibrant.
\end{theorem}
\begin{proof}
  The existence of the claimed combinatorial model structure and the Quillen
  properties of the forgetful functor are immediate consequences of
  \Cref{thm:Hovey-H-Mod,thm:Lurie-transfer-model_structure} (notice that every
  $H$-module is cofibrant). Since the forgetful functor
  $\LMod{\smashprod{A}{H}}\to\LMod{H}$ detects epimorphisms, the fibrations of
  this model structure are the epimorphisms, and the trivial fibrations are the
  epimorphisms with kernel in $\W$. Consequently, the class of  cofibrant
  objects of this model structure is precisely the class $\C$
  (see~\Cref{not:W-C-Hopfologica-der-cat}); notice also that every object of
  $\LMod{\smashprod{A}{H}}$ is fibrant.

  We now wish to show that $(\C,\W,\LMod{\smashprod{A}{H}})$ determines a
  (hereditary) abelian model structure on $\LMod{\smashprod{A}{H}}$. In view of
  \Cref{rmk:W} and \Cref{prop:AH-CW-cotorsion-pair}, we may deduce the statement
  from \cite[Corollary~1.1.9]{Bec14} as soon as we show that $\C\cap\W$ is
  precisely the class $\Proj{\smashprod{A}{H}}$ of projective
  ($\smashprod{A}{H}$)-modules. It is clear that
  $\Proj{\smashprod{A}{H}}\subseteq\C$. To show that
  $\Proj{\smashprod{A}{H}}\subseteq\W$ it suffices to observe that the class
  $\W$ is closed under small coproducts, direct summands and, by By
  \Cref{lemma:LR-preserving-proj-inj}, contains the regular representation of
  $\smashprod{A}{H}$ (whose underling $H$-module is $A\otimes H$). To show the
  converse inclusion $\C\cap\W\subseteq\Proj{\smashprod{A}{H}}$, given
  $X\in\C\cap\W$, choose an exact sequence of ($\smashprod{A}{H}$)-modules
  \[
    0\longrightarrow Y\longrightarrow P\longrightarrow X\longrightarrow 0
  \]
  with $P\in\Proj{\smashprod{A}{H}}$. Since $X\in\C\cap\W$ is an acyclic
  cofibrant object of the model structure constructed using
  \Cref{thm:Lurie-transfer-model_structure}, its identity morphism lifts along
  the epimorphism/fibration $P\twoheadrightarrow X$; hence, $X$ is a direct
  summand of $P$ and is therefore projective.

  In order to show that these two model structures coincide, it is enough to
  show that they have the same class of cofibrations and the same class of
  fibrant objects (which is clear, since all objects are fibrant in both model
  structures), see~\cite[Proposition~E.1.10]{Joy}. Therefore, it suffices to
  show that both model structures have the same class of trivial fibrations, as
  cofibrations are determined by the left-lifting property with respect to
  these. But it is clear that both model structures have the same class of
  trivial fibrations, namely the epimorphisms whose kernel lies in $\W$. This
  finishes the proof.
\end{proof}

\begin{remark}
  \label{rmk:Qis_qiso}
  The weak equivalences of the model structure described in \Cref{thm:Ohara} are
  precisely the `quasi-isomorphisms of ($\smashprod{A}{H}$)-modules' in the
  sense of~\cite[Definition~4.1]{Qi14}, see also~\cite[Remark~2.7]{Oha24}.
\end{remark}

\Cref{thm:Ohara} and \Cref{thm:Gillespie} permit us to make the following
definition. Recall also from \Cref{thm:SS-transfer-model_structure} that we may
view $A\in\Alg{\LMod{H}}$ as an algebra object of $\StLMod{H}$.

\begin{definition}
  \label{def:Hopfological_derived_category}
  The \emph{Hopfological derived $\infty$-category of $A$} is the presentable
  stable $\infty$-category
  \[
    \DerCat{A,H}\coloneqq\MLoc[\Weq]{\LMod{\smashprod{A}{H}}}\simeq\LMod{A}[\StLMod{H}],
  \]
  where the right-most equivalence is obtained
  from~\Cref{thm:Lurie-transfer-model_structure}. Its homotopy category
  $\Ho{\DerCat{A,H}}$ is the Hopfological derived category defined by Qi
  in~\cite[Section~4]{Qi14}, see~\Cref{rmk:Qis_qiso}. In particular, the
  $\infty$-category $\DerCat{A,H}$ is compactly generated
  by~\cite[Proposition~7.6]{Qi14} and is right-tensored over $\StLMod{H}$
  (\Cref{thm:props_of_Mod}). The full subcategory
  \[
    \perf{A,H}\coloneqq\DerCat{A,H}^\omega\subseteq\DerCat{A,H}
  \]
  of compact objects is the thick subcategory generated by the set
  \[
    \set{A\otimes S\in\DerCat{A,H}}[S\in\LMod{H}\text{: simple}],
  \]
  see~\cite[Corollary~7.15]{Qi14} or, alternatively, apply
  \Cref{thm:props_of_Mod}.
\end{definition}

\begin{remark}
  \label{rmk:cosing-presentation_Hopfological_derived_cat}
  The Hopfological derived $\infty$-category of $A$ admits the following
  alternative description. Equip the category $\Ch{\LMod{H}}$ with the model
  structure described in \Cref{thm:Becker-icosing}, so that there is an
  equivalence of presentable monoidal stable $\infty$-categories
  \[
    \StLMod{H}\stackrel{\sim}{\longrightarrow}\MLoc[\Weq[co,sing]]{\Ch{\LMod{H}}}.
  \]
  \Cref{thm:Lurie-transfer-model_structure} applied to this model category
  structure yields a (right-transferred) model category structure on the
  category
  \[
    \LMod{A}[\Ch{\LMod{H}}]\cong\Ch{\LMod{\smashprod{A}{H}}}
  \]
  whose underlying $\infty$-category is
  \begin{align*}
    \MLoc[\Weq[co,sing]]{\LMod{A}[\Ch{\LMod{H}}]}&\simeq\LMod{A}[\MLoc[\Weq[co,sing]]{\Ch{\LMod{H}}}]\\
                                                 &\simeq\LMod{A}[\StLMod{H}]\\
                                                 &\simeq\DerCat{A,H}.
  \end{align*}
  In particular, the Hopfological derived $\infty$-category of $A$ can be
  presented as an $\infty$-categorical localisation of the category of cochain
  complexes of $(\smashprod{A}{H})$-modules.
\end{remark}

\begin{remark}
  \label{rmk:trivial_example}
  In the special case where $A=\kk$ is the monoidal unit of $\LMod{H}$, there
  are isomorphisms of categories
  \[
    \LMod{\smashprod{\kk}{H}}\cong\LMod{\kk}[\LMod{H}]\cong\LMod{H}.
  \]
  From this, it readily follows that there is an equivalence of
  $\infty$-categories
  \[
    \DerCat{\kk,H}\simeq\StLMod{H}.
  \]
\end{remark}

When the Hopf algebra $H$ acts trivially on $A$, the Hopfological derived
$\infty$-category has a particularly simple description.

\begin{theorem}
  \label{thm:splitting_formula}
  Suppose that $H$ acts trivially on $A\in\Alg{\LMod{H}}$, so that $A$ is simply
  the datum of a $\kk$-algebra. Then, there is an equivalence of
  $\infty$-categories
  \[
    \DerCat{A,H}\simeq\DerCat{\LMod{A}}\otimes_{\DerCat{\Mod{\kk}}}\StLMod{H}
  \]
  right-tensored over $\StLMod{H}$, where the relative tensor product is taken
  over the derived $\infty$-category of the trivial Hopf algebra $\kk$.
\end{theorem}
\begin{proof}
  We begin the proof with some observations. We can interpret the assumption
  that $H$ acts trivially on $A$ as saying that $A$ lies in the image of the
  functor
  \[
    \Alg{\DerCat{\Mod{k}}}\longrightarrow\Alg{\StLMod{H}}
  \]
  induced by the canonical colimit-preserving monoidal functor
  $\DerCat{\Mod{\kk}}\to\StLMod{H}$ (\Cref{subsec:linear_structures}). Hence,
  \Cref{thm:props_of_Mod}, yields an equivalence of $\infty$-categories
  \[
    \DerCat{A,H}\simeq\LMod{A}[\StLMod{H}]\simeq\LMod{A}[\DerCat{\Mod{\kk}}]\otimes_{\DerCat{\Mod{\kk}}}\StLMod{H}.
  \]
  On the other hand, \Cref{prop:der-cat-smash} for the trivial Hopf algebra
  $\kk$ yields an equivalence of $\infty$-categories
  \[
    \DerCat{\LMod{A}}\simeq\LMod{A}[\DerCat{\Mod{\kk}}]
  \] 
  Combining these two equivalences we obtain the required equivalence of $\infty$-categories.
\end{proof}

Recall from \Cref{subsec:linear_structures} that there is an exact monoidal functor
\[
  \kk\otimes-\colon\perf{\kk}\longrightarrow\stLmod{H}.
\]
By restriction, we obtain a functor
\[
  \RMod{\stLmod{H}}[\PrStLcg]\longrightarrow\Mod{\perf{\kk}}[\PrStLcg].
\]
So that we may view $\stLmod{H}$ as a $\perf{\kk}$-linear $\infty$-category.
\begin{corollary}
  \label{coro:splitting_formula_compacts}
  Suppose that the Hopf algebra $H$ acts trivially on $A$. Then, there is an
  equivalence of $\kk$-linear $\infty$-categories
  \[
    \perf{A,H}\simeq\perf{A}\otimes_{\perf{\kk}}\stLmod{H}.
  \]
\end{corollary}
\begin{proof}
  This is an immediate consequence of \Cref{thm:splitting_formula} and the
  well-known relationship between the relative tensor product over
  $\DerCat{\Mod{\kk}}$ and over $\perf{\kk}$, compare
  with~\cite[Section~3.1]{BGT13}. Namely,
  \[
    \perf{A,H}=(\DerCat{A,H})^\omega\simeq(\DerCat{\LMod{A}}\otimes_{\DerCat{\Mod{\kk}}}\StLMod{H})^\omega=\perf{A}\otimes_{\perf{\kk}}\stLmod{H}.
  \]
\end{proof}

\Cref{thm:splitting_formula} has the following interesting consequence.

\begin{theorem}
  \label{thm:Morita_invariance}
  Let $A$ and $B$ be a pair of left $H$-module algebras on which $H$ acts
  trivially. Suppose that $X$ is a cochain complex of ordinary $B$-$A$-bimodules
  such that the $\DerCat{\Mod{\kk}}$-linear functor
  \[
    X\Lotimes*_A-\colon\DerCat{\LMod{A}}\stackrel{\sim}{\longrightarrow}\DerCat{\LMod{B}}
  \]
  is an equivalence of derived $\infty$-categories. Then, there is an induced
  equivalence of Hopfological derived $\infty$-categories
  \[
    \DerCat{A,H}\stackrel{\sim}{\longrightarrow}\DerCat{B,H}
  \]
  compatible with the right action of $\StLMod{H}$.
\end{theorem}
\begin{proof}
  The required equivalence of $\infty$-categories is an immediate consequence of
  \Cref{thm:splitting_formula} and the fact that the relative tensor product is
  functorial on $\DerCat{\Mod{\kk}}$-linear colimit-preserving functors.
\end{proof}

\subsection{A Hopfological analogue of Krause's recollement}

We now explain how to lift Krause's recollement to the Hopfological setting.

\begin{theorem}
  \label{thm:Krauses-recollement-Hopfological}
  There is a commutative diagram of recollements of compactly-generated stable
  $\infty$-categories
  \[
    \begin{tikzcd}
      \DerCat{A,H}\ar[hookrightarrow]{r}[description]{j}\dar&\LMod{A}[\KCh{\Inj{H}}]\dar\ar{r}[description]{q}\ar[shift
      right=0.5em]{l}[swap]{j_L}\ar[shift
      left=0.5em]{l}{j_R}&\DerCat{\LMod{\smashprod{A}{H}}}\ar[shift
      right=0.5em,hook]{l}[swap]{q_L}\ar[hook',shift left=0.5em]{l}{q_R}\dar\\
      \StLMod{H}\ar[hookrightarrow]{r}[description]{i}&\KCh{\Inj{H}}\ar{r}[description]{p}\ar[shift
      right=0.5em]{l}[swap]{i_L}\ar[shift
      left=0.5em]{l}{i_R}&\DerCat{\LMod{H}}\ar[shift
      right=0.5em,hook]{l}[swap]{p_L}\ar[hook',shift left=0.5em]{l}{p_R}
    \end{tikzcd}
  \]
  where the vertical functors are the corresponding forgetful functors and the
  bottom row is given by Krause's recollement (\Cref{thm:Krauses-recollement}).
\end{theorem}
\begin{proof}
  Firstly, it follows from \Cref{rmk:Beckers_proof} that there is commutative
  diagram of Quillen functors (as indicated with the decorations) 
  \begin{center}\vspace{-1em}
  \resizebox{\textwidth}{!}{\begin{tikzcd}[ampersand replacement=\&]
      \LMod{A}[\Ch{\LMod{H}}]<co,sing>\ar{d}\ar{r}[description]{\mathbb{R}\id}\&\LMod{A}[\Ch{\LMod{H}}]<co>\ar{d}\ar{r}[description]{\mathbb{L}\id}\ar[shift
      right=0.5em]{l}[swap]{\mathbb{L}\id}\&\LMod{A}[\Ch{\LMod{H}}]<inj>\ar[shift
      left=0.5em]{l}{\mathbb{R}\id}\ar{d}\\
      \Ch{\LMod{H}}<co,sing>\ar{r}[description]{\mathbb{R}\id}\&\Ch{\LMod{H}}<co>\ar{r}[description]{\mathbb{L}\id}\ar[shift
      right=0.5em]{l}[swap]{\mathbb{L}\id}\&\Ch{\LMod{H}}<inj>\ar[shift
      left=0.5em]{l}{\mathbb{R}\id}
    \end{tikzcd}}
  \end{center}
  Here, the top row is given by the category $\LMod{A}[\Ch{\LMod{H}}]$ equipped
  with the right-transferred model structures from the bottom row, for which the
  forgetful functors are both right Quillen functors and left Quillen functors.
  (\Cref{thm:Lurie-transfer-model_structure}). The claim that the functors in
  the top row are Quillen functors readily
  follows. Passing to underlying $\infty$-categories in the above diagram, we
  obtain a commutative diagram 
    \[
    \begin{tikzcd}
      \DerCat{A,H}\ar{r}[description]{j}\ar{d}[description]{U}&\LMod{A}[\KCh{\Inj{H}}]\ar{d}[description]{U}\ar{r}{q}\ar[shift
      right=0.5em]{l}[swap]{j_L}&\DerCat{\LMod{\smashprod{A}{H}}}\ar[hook',shift left=0.5em]{l}{q_R}\ar{d}[description]{U}\\
      \StLMod{H}\ar[hookrightarrow]{r}[description]{i}&\KCh{\Inj{H}}\ar{r}[description]{p}\ar[shift
      right=0.5em]{l}[swap]{i_L}&\DerCat{\LMod{H}}\ar[hook',shift
      left=0.5em]{l}{p_R}
    \end{tikzcd}
  \]
  Recall that the downwards-pointing forgetful functors are conservative. To
  prove that the top row of the diagram determines a recollement, it suffices to
  show the following:
  \begin{itemize}
  \item The functors $j$ and $q_R$ are fully faithful. Since the case of $q_R$
    can be shown analogously, we only prove that $j$ is fully faithful. For
    this, it suffices to prove that, for each object $X\in\DerCat{A,X}$, the counit map $\varepsilon_X\colon j_L
    j(X)\to X$ of the adjunction $j_L\dashv j$ is invertible. For this, it suffices
    to observe that the commutativity of the diagram implies that
    $U(\varepsilon_X)$ is equivalent to the counit map $i_Li(U(X))
    \to U(X)$ of the adjunction $i_L\dashv i$, which is invertible for $i$ is fully faithful. Since the functor
    $U$ is conservative, the claim follows.
  \item The essential image if $j$ is precisely the kernel of $q$. To prove
    this, we observe first that
    \[
      U\circ q\circ j\simeq p\circ i\circ U\simeq 0.
    \]
    Since the functor $U$ is conservative, we deduce that $q\circ i\simeq 0$.
    Suppose now that $X\in\LMod{A}{\KCh{\Inj{H}}}$ lies in the kernel of $q$. We
    need to prove that the unit map $\eta_X\colon X\to jj_L(X)$ is invertible. It
    suffices to show that $U(\eta_X)$ is invertible (again, since $U$ is
    conservative). The commutativity of the diagram implies that $U(\eta_X)$ is
    equivalent to the unit map $\alpha_{U(X)}\colon U(X)\to ii_L(U(X))$; to prove that the latter
    map is invertible, it is enough to notice that $p(U(X))\simeq
    U(q(X))\simeq0$ by assumption. Since the kernel of $p$ is precisely the
    image of $i$, it follows that $\alpha_{U(X)}$ is invertible, as required.
  \end{itemize}
  This finishes the proof of the theorem.
\end{proof}

\subsection{Hopfological Morita theory}
\label{subsec:Morita_theory}

In this section we revisit the Hopfological analogue of (derived) Morita theory
discussed in~\cite[Section~8]{Qi14}. Notice that, in contrast to
\emph{op.~cit.}, we do not assume the Hopf algebra $H$ to be commutative nor
cocommutative.

Throughout this subsection, we fix a pair of algebra objects
$A,B\in\Alg{\LMod{H}}$. Recall that a $B$-$A$-bimodule object of $\LMod{H}$ is a
left $H$-module $M$ equipped with an associative and unital action map
\[
  B\otimes X\otimes A\longrightarrow X,
\]
given by a morphism of left $H$-modules. It is straightforward to verify that
the category $\Bimod{B}{A}[\LMod{H}]$ is an abelian category with small limits
and small colimits, which are created by the forgetful functor to $\LMod{H}$.
Moreover, the forgetful functor to $\LMod{H}$ is part of an adjunction
\[
  \begin{tikzcd}
    B\otimes-\otimes A\colon\LMod{H}\rar[shift
    left]&\Bimod{B}{A}[\LMod{H}]\noloc U\lar[shift left]
  \end{tikzcd}
\]
where both functors are exact. From this, it easily follows that
$\Bimod{B}{A}[\LMod{H}]$ has enough projectives. Indeed, the left adjoint
$X\mapsto B\otimes P\otimes A$ preserves projective objects (because it has an
exact right adjoint, see~\Cref{lemma:LR-preserving-proj-inj}). Hence, given $X\in\Bimod{B}{A}[\LMod{H}]$, we can choose
an epimorphism $p\colon P\twoheadrightarrow UX$ in $\LMod{H}$ with $P$ a projective $H$-module
to obtain, by adjunction, a morphism $\overline{p}\colon B\otimes P\otimes A\to
X$. Finally, since the forgetful functor detects epimorphisms, the commutative
 diagram
\[
  \begin{tikzcd}
    P\rar{\eta_P}\ar[two heads]{dr}[swap]{p}&U(B\otimes P\otimes A)\dar{U(\overline{p})}\\
    &X
  \end{tikzcd}
\]
shows that $\overline{p}$ is an epimorphism; here, $\eta_P\colon P\to U(B\otimes
P\otimes A)$ denotes the unit of the adjunction. It now also follows that
$\Bimod{B}{A}[\LMod{H}]$ is a Grothendieck category.

We begin our discussion by formulating the following
rectification result for Hopfological bimodules.

\begin{proposition}
  \label{prop:transfer-bimodules}
  The following
  statements hold:
  \begin{enumerate}
  \item The category $\Bimod{B}{A}[\LMod{H}]$ of
    $B$-$A$-bimodule objects in $\LMod{H}$ admits a (right-transferred) combinatorial model
    structure determined as follows:
    \begin{itemize}
    \item A morphism in $\Bimod{B}{A}[\LMod{H}]$ is a weak equivalence if its underlying
      morphism is a stable isomorphism in $\LMod{H}$.
    \item A morphism in $\Bimod{B}{A}[\LMod{H}]$ is a fibration if it is an epimorphism.
    \item A morphism in $\Bimod{B}{A}[\LMod{H}]$ is a cofibration if it has the left
      lifting property with respect to the trivial fibrations.
    \end{itemize}
    Moreover, the forgetful functor $\Bimod{B}{A}[\LMod{H}]\to\LMod{H}$ is both a left Quillen
    functor and a right Quillen functor (in particular, it admits a left and a
    right adjoint).
  \item There is a canonical equivalence of
    $\infty$-categories
    \begin{center}
      $\MLoc[\Weq[st]]{\Bimod{B}{A}[\LMod{H}]}\stackrel{\sim}{\longrightarrow}\Bimod{B}{A}[\StLMod{H}].$
    \end{center}
  \item In fact, the above model category structure is the hereditary abelian
    model category structure determined by the triple
    $(\C,\W,\Bimod{B}{A}[\LMod{H}])$, where $\W\subseteq\Bimod{B}{A}[\LMod{H}]$
    is the full subcategory of objects whose underlying $H$-module is
    (projective-)injective and
    \begin{center}
      $\C\coloneqq\set{X\in\Bimod{B}{A}[\LMod{H}]}[\forall Y\in\W,\ \Ext<B\text{-}A>[1]{X}{Y}=0].$
    \end{center}
  \end{enumerate}
\end{proposition}
\begin{proof}
  The first two statement statements follow from
  \Cref{thm:Lurie-transfer-model_structure} applied to the model structure
  described in~\Cref{thm:Hovey-H-Mod}. The third statement can be proven using a
  similar argument to the one used in the proof of \Cref{thm:Ohara}, noticing
  that \Cref{prop:AH-CW-cotorsion-pair} also admits an analogue in this context.
  Namely,
  \[
    \W=\set{Y\in\Bimod{B}{A}[\LMod{H}]}[\forall S\in\LMod{H}\text{: simple},\
    \Ext<B\text{-}A>[1]{B\otimes
      S\otimes A}{Y}=0].\qedhere
  \]
\end{proof}

\begin{remark}
  There is a variant of \Cref{prop:transfer-bimodules} in which
  \Cref{thm:Lurie-transfer-model_structure} is applied to the model structure on
  $\Ch{\LMod{H}}$ described in \Cref{thm:Becker-icosing}, compare with
  \Cref{rmk:cosing-presentation_Hopfological_derived_cat}.
\end{remark}

\Cref{prop:transfer-bimodules} and \Cref{thm:Gillespie} permits us to make the
following definition.

\begin{definition}
  We call the presentable stable $\infty$-category
  \[
    \Bimod{B}{A}[\StLMod{H}]\simeq\MLoc[\Weq[st]]{\Bimod{B}{A}[\LMod{H}]}
  \]
  the \emph{Hopfological derived $\infty$-category of $B$-$A$-bimodules}. This a
  compactly-generated stable $\infty$-category, with
  \[
    \set{B\otimes S\otimes A\in\Bimod{B}{A}[\StLMod{H}]}[S\in\LMod{H}\text{: simple}]
  \]
  as a set of compact generators, see~\Cref{thm:props_of_Mod}.
\end{definition}

In what follows, we specialise part of the discussion in~\cite[p.~738]{Lur17} to
our setting. According to \cite[Theorem~4.8.4.1]{Lur17}, there is an equivalence
of $\infty$-categories
\begin{align*}
  \Bimod{B}{A}[\StLMod{H}]&\simeq\LFun<\StLMod{H}>{\LMod{A}[\StLMod{H}]}{\LMod{B}[\StLMod{H}]}\\
  &\simeq\LFun<\StLMod{H}>{\DerCat{A,H}}{\DerCat{B,H}},
\end{align*}
where, given presentable stable $\infty$-categories $\C$ and $\D$ that are right-tensored over
$\StLMod{H}$, we let $\LFun<\StLMod{H}>{\C}{\D}$ be $\infty$-category of
colimit-preserving $\StLMod{H}$-linear functors $\C\to\D$. Given a
$B$-$A$-bimodule $X\in\Bimod{B}{A}[\StLMod{H}]$, we denote its image under this
equivalence by
\[
  X\Lotimes*_A-\colon\DerCat{A,H}\longrightarrow\DerCat{B,H}.
\]
By the Adjoint Functor Theorem~\cite[Corollary~5.5.2.9]{Lur17}, this
colimit-preserving functor admits a right adjoint that we denote
\[
  \RHom<B>{M}{-}\colon\DerCat{B,H}\longrightarrow\DerCat{A,H}.
\]
Consequently, the Hopfological derived $\infty$-categories $\DerCat{B,H}$ and
$\DerCat{A,H}$ are equivalent as presentable stable $\infty$-categories
right-tensored over $\StLMod{H}$ if and only if there exists a $B$-$A$-bimodule
$X\in\Bimod{B}{A}[\StLMod{H}]$ such that the functor $X\Lotimes*_A-$ is an
equivalence of $\infty$-categories.

\begin{remark}
  Suppose that the Hopf algebra $H$ is cocommutative, so that the monoidal
  {$\infty$-}categories $\LMod{H}$ and $\StLMod{H}$ are symmetric. In this case, there are
  isomorphisms of categories
  \[
    \Bimod{B}{A}[\LMod{H}]\cong\LMod{B\otimes
      A^\rev}[\LMod{H}]\cong\LMod{\smashprod{(B\otimes A^\rev)}{H}}[\Mod{\kk}]
  \]
  that are compatible with the forgetful functor to $\LMod{H}$. It follows that,
  in the above discussion,
  we can replace the $\infty$-category $\Bimod{B}{A}[\StLMod{H}]$ by the
  (equivalent) Hopfological derived $\infty$-category
  \[
    \DerCat{B\otimes A^\rev,H}\simeq\MLoc[\Weq[st]]{\LMod{\smashprod{(B\otimes A^\rev)}{H}}[\Mod{\kk}]}.
  \]
  This corresponds to the setup considered in~\cite[Section~8]{Qi14}.
\end{remark}

\subsection{Hopfological analogues of classical invariants}

\subsubsection{Additive and localising invariants}
\label{subsec:additive_localising_inv}

Recall that $\PrStLcg$ denotes the $\infty$-category of essentially small
idempotent-complete stable $\infty$-categories and exact functors between them,
endowed with the tensor product described~\cite[Section~3.1]{BGT13}. Since the tensor product on
$\StLMod{H}$ restricts to the full subcategory $\stLmod{H}\subseteq\StLMod{H}$
of compact objects, we may regard $\stLmod{H}$ as an algebra object
of $\PrStLcg$. Moreover, the induced right action of $\stLmod{H}$ on
$\DerCat{A,H}$ preserves the full subcategory
$\perf{A,H}\subseteq\DerCat{A,H}$ of compact objects and, therefore,
\[
  \perf{A,H}\in\RMod{\stLmod{H}}[\PrStLcg].
\]
The upshot is that, according to
\Cref{prop:monoidal_functors_preserve_algs_mods}, every lax monoidal
functor
\[
  E\colon\PrStLcg\longrightarrow\M
\]
to some monoidal $\infty$-category $\M$ induces a functor
\[
  E\colon\RMod{\stLmod{H}}[\PrStLcg]\longrightarrow\RMod{E(\stLmod{H})}[\M],\qquad
  \C\longmapsto E(\C).
\]
In particular, the object $E(\perf{A,H})\in\M$ inherits a right action
of the algebra object $E(\stLmod{H})\in\Alg{\M}$.

\begin{remark}
  If the Hopf algebra $H$ is cocommutative, then $\stLmod{H}$ is a symmetric
  monoidal $\infty$-category. In this case, it is also natural to ask for the
  functor $E$ to be lax symmetric monoidal, so that $E(\stLmod{H})$
  is a commutative algebra object of $\M$.
\end{remark}

If the target monoidal $\infty$-category $\M$ is additive,\footnote{The
  definition of additive category extends verbatim to the $\infty$-categorical
  setting, see for example~\cite[Definition~C.1.5.1]{Lur18SAG}.} one may ask $E$
to be an \emph{additive invariant}~\cite[Definition~6.1]{BGT13}, which is to say
that $E$ sends reflexive localisation sequences
\[
  \begin{tikzcd}
    \A\rar[hookrightarrow,shift left]{i}&\B\rar[shift left]{p}\lar[shift
    left]&\C\lar[hook',shift left]
  \end{tikzcd}
\]
to direct-sum decompositions $E(\B)\simeq E(\A)\oplus E(\B)$ in $\M$; here, by
definition,
\begin{itemize}
\item the functor $i$ is fully faithful and embeds $\A$ into $\B$ as the kernel
  of $p$,
\item the functors $p$ and $i$ admits right adjoints, and
\item the right adjoint to $p$ is fully faithful.
\end{itemize}
Similarly, if the target monoidal $\infty$-category $\M$ is stable, then one may ask that $E$
is a \emph{localising invariant}~\cite[Definition~8.1]{BGT13}, which is to say
that $E$ sends localisation sequences
\[
  \begin{tikzcd}
    \A\rar[hookrightarrow]{i}&\B\rar{p}&\C
  \end{tikzcd}
\]
to bicartesian squares
\[
  \begin{tikzcd}
    E(\A)\rar\dar\ar[phantom]{dr}[description]{\square}&E(\B)\dar\\
    0\rar&E(\C)
  \end{tikzcd}
\]
in $\M$, called fibre-cofibre sequences; here, by definition, the functor $i$ is
fully faithful, the composition $p\circ i$ vanishes, and the canonical functor
$(\B/\A)^\flat\stackrel{\sim}{\to}\C$ is an equivalence. Notice that every
localising invariant is additive, but not conversely. If $\M$ is a presentably
symmetric monoidal $\infty$-category, one typically also asks that $E$ preserves
filtered colimits (both for additive and for localising invariants); invariants
with this additional property are sometimes called \emph{finitary}.

\begin{remark}
  One may also consider additive and localising invariants that are not lax
  monoidal. In this case, there is no natural action of $E(\stLmod{H})$ on
  $E(\perf{A,H})$.
\end{remark}

Suppose now that the Hopf algebra $H$ acts trivially on $A$. Recall that there
is a monoidal functor $\perf{\kk}\to\stLmod{H}$
(\Cref{subsec:linear_structures}) that, by
\Cref{prop:monoidal_functors_preserve_algs_mods}, induces a forgetful functor
\[
  \RMod{\stLmod{H}}[\PrStLcg]\longrightarrow\Mod{\perf{\kk}}[\PrStLcg].
\]
Suppose given a lax symmetric monoidal functor
\[
  E\colon\Mod{\perf{\kk}}[\PrStLcg]\longrightarrow\M,
\]
to a presentably symmetric monoidal stable $\infty$-category. The functor $E$
refines to a lax symmetric monoidal functor
\[
  E\colon\Mod{\perf{\kk}}[\PrStLcg]\longrightarrow\Mod{E(\perf{\kk})}[\M],
\]
where the right-hand side is equipped with the relative tensor product over the
commutative algebra object $E(\perf{\kk})\in\CAlg{\M}$. In this case,
\Cref{thm:splitting_formula} implies that there exists a canonical K{\"u}nneth-type
comparison map
\[
  E(\perf{A})\otimes_{E(\perf{\kk})}E(\stLmod{H})\longrightarrow E(\perf{A,H}).
\]
It is an interesting (and often difficult) question to determine when this map
is invertible. For example, this is the case for Hochschild
homology~\cite[Example~8.9]{CT12} and topological Hochschild
homology~\cite[Theorem~14.1]{BM24} (see also~\cite[p.~4564]{AMN18}). If we work
over a perfect field of positive characteristic and the algebra $A$ is
finite-dimensional and of finite global dimension, then this is also the case for
periodic topological Hochschild homology~\cite{BM24,AMN18} (see in particular
the paragraph immediately after~\cite[Theorem~1.3]{AMN18}). In general, for
trivial reasons, the above comparison map is invertible if $\perf{A}$ admits a full exceptional sequence, and $E$ is an
additive invariant. Indeed, in this case $E(\perf{A})\simeq
E(\perf{\kk})^{\oplus n}$, where $n$ is the size of the full exceptional
sequence, see for example~\cite[Section~2.4.2]{Tab15}.

\subsubsection{Algebraic $K$-theory of Hopfological derived $\infty$-categories}
\label{subsubsec:K-theory}

The discussion in \Cref{subsec:additive_localising_inv} applies, in particular,
to (connective) algebraic $K$-theory
\[
  K\colon\PrStLcg\longrightarrow\Spectra,\qquad \C\longmapsto K(\C),
\]
which is a finitary additive invariant~\cite[Proposition~7.10]{BGT13}, and to
non-connective algebraic $K$-theory
\[
  \mathbf{K}\colon\PrStLcg\longrightarrow\Spectra\qquad\C\longmapsto\mathbf{K}(\C),
\]
which is a finitary localising invariant~\cite[Theorem~9.8]{BGT13}; both of
these are lax symmetric monoidal functors with target the presentable stable
$\infty$-category $\Spectra$ of spectra~\cite[Theorem~1.13]{BGT14}. Hence,
keeping in mind~\Cref{rmk:trivial_example}, we have
\begin{align*}
  K(A,H)\coloneqq K(\perf{A,H})&\in\RMod{K(\kk,H)}(\Spectra),&K(\kk,H)&\coloneqq K(\stLmod{H}),\intertext{and}\mathbf{K}(A,H)\coloneqq\mathbf{K}(\perf{A,H})&\in\RMod{\mathbf{K}(\kk,H)}(\Spectra),&\mathbf{K}(\kk,H)&\coloneqq\mathbf{K}(\stLmod{H}).
\end{align*}
The functor
\[
  \textstyle\pi_*\colon\Spectra\longrightarrow\prod_{n\in\ZZ}\Mod{\ZZ},\qquad
  X\longmapsto\textstyle\bigoplus_{n\in\ZZ}\pi_n(X),
\]
that associates to a spectrum the the direct sum of its stable homotopy groups
is lax (symmetric) monoidal~\cite[Section 7.1.1]{Lur17}. Consequently, the graded abelian group
$K_*(A,H)=\pi_*(K(A,H))$ is a graded right module over the graded ring
$K_*(\kk,H)=\pi_*(\kk,H)$. In particular, the Grothendieck group $K_0(A,H)$ is a
right module over the Grothendieck ring $K_0(\kk,H)$, compare
with~\cite[Remark~7.17]{Qi14}.

\begin{remark} If the Hopf algebra $H$ is cocommutative, then $K(\kk,H)$ is an
  $\EE_\infty$-ring spectrum. In this case, the graded ring $K_*(\kk,H)$ is
  graded-commutative and the Grothendieck ring $K_0(\kk,H)$ is commutative.
\end{remark}

\subsubsection{Hopfological Hochschild (co)homology}
\label{subsubsec:Hochschild}

In this subsection we propose Hopfological analogues of Hochschild (co)homology.
Below, we specify to our setting the discussions in~\cite[Section~2]{BR23} (for
the definition Hochschild cohomology) and in~\cite[Section~4.5]{HSS17} (for the
definition of Hochschild homology). In the forthcoming discussion, we need to
assume that the Hopf algebra $H$ is cocommutative, so that $\StLMod{H}$ and
\[
  \Mod{\StLMod{H}}[\PrStL]=\LMod{\StLMod{H}}[\PrStL]
\]
are symmetric monoidal $\infty$-categories.

We begin with the definition of Hochschild cohomology. Let $A\in\Alg{\LMod{H}}$
be an algebra object. By \Cref{thm:props_of_Mod}, we may regard the
Hopfological derived $\infty$-category
\[
  \DerCat{A,H}\simeq\LMod{A}[\StLMod{H}]
\]
as a right $\StLMod{H}$-module. Consider the presentable stable
$\infty$-category
\[
  \Bimod{A}{A}[\StLMod{H}]\simeq\LFun<\StLMod{H}>{\DerCat{A,H}}{\DerCat{A,H}},
\]
We may then consider the diagonal bimodule
\[
  A\in\Bimod{A}{A}[\StLMod{H}],
\]
which is the $A$-bimodule corresponding to the identity functor of
$\DerCat{A,H}$. As explained in~\cite[Section~2.1]{BR23}, way may then also consider
the endomorphism algebra object
\[
  \HH[\bullet]{A,H}\coloneqq\EEnd<A\text{-}A>{A}\in\StLMod{H},
\]
which we refer to as the \emph{Hopfological Hochschild cohomology of $A$}.

\begin{remark}
  According to~\cite[Proposition~2.1.5]{BR23}, the Hopfological Hochschild
  cohomology object $\HH[\bullet]{A,H}\in\StLMod{H}$ carries a canonical algebra
  structure over the $\EE_2$-operad that, by Dunn's Additivity
  Theorem~\cite[Theorem~5.1.2.2]{Lur17}, corresponds to a pair of compatible
  associative algebra structures on $\HH[\bullet]{A,H}$; in this case, these
  algebra structures are induced by the composition of morphisms in $\Bimod{A}{A}[\StLMod{H}]$ and
  by its natural monoidal structure, given by the relative tensor product over $A$.
  In other words, Deligne's Conjecture is
  valid in this setting.
\end{remark}

We now turn our attention to Hochschild homology. The Hopfological derived
$\infty$-category
\[
  \DerCat{A,H}\simeq\LMod{A}[\StLMod{H}]
\]
is a dualisable object of the symmetric monoidal $\infty$-category
$\Mod{\StLMod{H}}[\PrStL]$, see~\cite[Remark~4.8.5.17 and~4.8.5.18]{Lur17}.
This means that there exist coevaluation and evaluation functors
\[
  \coev\colon\StLMod{H}\longrightarrow\DerCat{A,H}\otimes_{\StLMod{H}}\DerCat{A,H}^\vee
\]
and
\[
  \ev\colon\DerCat{A,H}^\vee\otimes_{\StLMod{H}}\DerCat{A,H}\longrightarrow\StLMod{\kk},
\]
where
\begin{align*}
  \DerCat{A,H}^\vee&\coloneqq\LFun<\StLMod{H}>{\DerCat{A,H}}{\StLMod{H}}\\
                   &\simeq\LFun<\StLMod{H}>{\LMod{A}[\StLMod{H}]}{\LMod{\kk}[\StLMod{H}]}\\
                   &\simeq\Bimod{\kk}{A}[\StLMod{H}]\\
                   &\simeq\RMod{A}[\StLMod{H}];
\end{align*}
moreover, these evaluation and coevaluation functors must satisfy suitable
variants of the triangle identities~\cite[Definition~4.6.1.7]{Lur17}.
Following~\cite[Section~4.5]{HSS17}, we define the \emph{Hopfological Hochschild
  homology} of $A$ as the image of the monoidal unit $\kk\in\StLMod{H}$ under
the composite
\[
  \begin{tikzcd}
    \kk\dar[mapsto]&\StLMod{H}\rar{\coev}\dar[dotted]&\DerCat{A,H}\otimes_{\StLMod{H}}\DerCat{A,H}^\vee\dar{\wr}\\
    \HH*[\bullet]{A,H}&\StLMod{H}&\DerCat{A,H}^\vee\otimes_{\StLMod{H}}\DerCat{A,H}\lar{\ev}
  \end{tikzcd}
\]
where the existence of the vertical equivalence on the right uses the assumption
that the monoidal structure on $\Mod{\StLMod{H}}[\PrStL]$ is symmetric.

\begin{remark}
  Similar to the classical case, there is a non-commutative calculus involving
  an action of the Hopfological Hochschild cohomology of $A$ on its Hopfological
  Hochschild homology, see~\cite[Section~2.3]{BR23} for details.
\end{remark}

\begin{remark}
  In~\cite{QS22,QS23}, Qi and Sussan introduce Hopfological analogues of
  Hochschild cohomology for the Hopf algebra $H=\kk[\partial]/(\partial^p)$ over
  a field of characteristic $p>0$. Their construction involves explicit
  (co)chain complexes and it would be interesting to compare their
  constructions to those in this section.
\end{remark}

\begin{remark}
  Hopfological Hochschild (co)homology agrees with classical Hochschild
  (co)homology when $H$ is the graded algebra of dual numbers in variable of
  cohomological degree $1$ and, consequently, $A$ is a differential graded algebra.
\end{remark}

\subsection{Descent for the Hopfological derived $\infty$-category}
\label{subsec:descent}

In this section, we assume that our Hopf algebra $H$ is cocommutative, so that
the monoidal structure on $\StLMod{H}$ is symmetric. In this subsection we translate some descent results for the Hopfological derived $\infty$-category from those for $\StLMod{H}$. We first need to fix some terminology.

The following is a particular case of a result from~\cite{Ram23}.

\begin{lemma}[{\cite[Lemma 3.23]{Ram23}}]\label{lemma-ramzi}
    Let $\C_\bullet\colon I^\lhd\to \mathrm{CAlg}(\mathrm{Pr^L_\mathrm{st}})$ be a limit diagram, and let $\C_\infty$ denote the value at the cone point $\infty$ in $I^\lhd$. For an algebra $A\in \Alg{\C_\infty}$, the canonical map
    \[
    \LMod{A}[\C_\infty] \to \lim_{i\in I} \LMod{A_i}[\C_i]
    \]
    is an equivalence of stable $\infty$-categories, where $A_i$ denotes the image of the projection functor $\Alg{\C_\infty}\to \Alg{\C_i}$.
\end{lemma}

As a consequence of our definition of the Hopfological derived $\infty$-category, we obtain the following result.

\begin{proposition}\label{prop-descent-stmod}
    Let $A\in \Alg{\LMod{H}}$. Assume that there is an equivalence of symmetric monoidal stable $\infty$-categories
    \[
    \StLMod{H} \stackrel{\simeq}{\longrightarrow} \lim_{i\in I} \StLMod{H_i} 
    \]
   for some collection of cocommutative Hopf algebras $H_i$. Then there is an equivalence of stable $\infty$-categories
    \[
    \DerCat{A,H} \stackrel{\simeq}{\longrightarrow} \lim_{i\in I} \DerCat{A_i,H_i},
    \]
    where $A_i$ is the image of $A$ under the projection functor $\StLMod{H} \to \StLMod{H_i}$.
\end{proposition}

\begin{proof}
Recall from \Cref{thm:SS-transfer-model_structure} that we can consider $A$ as
an algebra object in $\StLMod{H}$. Now, the conclusion is a direct consequence
of \Cref{lemma-ramzi}, since $\DerCat{A,H}$ is defined as
$\LMod{A}[\StLMod{H}]$, see \Cref{def:Hopfological_derived_category}.
\end{proof}

In order to show the relevance of the previous result, let us provide some examples where it is particularly useful.

\begin{example}
  Let $G$ be a finite group, and let $\kk$ be a field of positive characteristic
  $p$ dividing the order of the group $G$. We let $\kk G$ denote the group
  algebra with the usual Hopf algebra structure. Recall that, in this case, $\kk
  G$ is cocommutative. Let $\F$ denote a family\footnote{That is, a collection of subgroups closed under conjugation and passage to subgroups.} of subgroups of $G$ containing the elementary abelian $p$-subgroups.

  In this case, it was shown by Mathew in \cite[Corollary 9.16]{Mat16} that the restriction to subgroups induces an equivalence of symmetric monoidal stable $\infty$-categories
  \[
   \StLMod{\kk G} \stackrel{\simeq}{\longrightarrow} \lim_{G/E\in \mathcal{O}_\F(G)^{\mathrm{op}}} \StLMod{\kk E} 
  \]
 where $\mathcal{O}_\F(G)$ denotes the orbit category of $G$ of homogeneous $G$-spaces with isotropy in $\F$. Let $A\in \Alg{\LMod{\kk G}}$. As a direct consequence of \Cref{prop-descent-stmod}, we obtain an equivalence of stable $\infty$-categories
  \[
   \DerCat{A,\kk G} \stackrel{\simeq}{\longrightarrow} \lim_{G/E\in \mathcal{O}_\F(G)^{\mathrm{op}}} \DerCat{A,\kk E} 
  \]
  where $A$ has the structure of a $\kk E$-algebra by restricting the action via the inclusion $E\subseteq G$.
\end{example}

\begin{remark}
One might wonder if there are similar descent results for more general cocommutative Hopf algebras. We refer the interested reader to \cite[Section 4.3]{Mat16} for a discussion of the topic.
\end{remark}

\section{Hopfological algebra in rigid monoidal stable $\infty$-categories}
\label{sec:spectral_Hopfological_algebra}

In this section we outline a generalisation of Hopfological algebra in which we
replace the derived $\infty$-category of the ground field by an arbitrary
rigidly-compactly generated stable $\infty$-category.

\subsection{rigidly-compactly generated monoidal stable $\infty$-categories}

We begin by recalling a standard definition.

\begin{definition}[{\cite[Definition~4.6.1.1]{Lur17}}]
  \label{def:duality_datum}
  Let $\M$ be a monoidal $\infty$-category. A \emph{duality datum} is
  a tuple $(X,\ldual{X},\ev,\coev)$ where $X,\ldual{X}\in\M$ is a pair of objects and
  \[
    \ev\colon \ldual{X}\otimes
    X\longrightarrow\unit\qquad\text{and}\qquad\coev\colon\unit\longrightarrow
    X\otimes \ldual{X}
  \]
  are morphisms in $\M$ that define a duality datum in the ordinary monoidal
  category $\Ho{\M}$. Explicitly, the following diagrams are required to commute
  in the homotopy category $\Ho{\M}$:
  \[
    \begin{tikzcd}[column sep=tiny]
      &X\otimes \ldual{X}\otimes X\drar{\id\otimes\ev}\\
      X\ar{rr}[swap]{\id}\urar{\coev\otimes\id}&&X
    \end{tikzcd}\qquad
    \begin{tikzcd}[column sep=tiny]
      &\ldual{X}\otimes X\otimes \ldual{X}\drar{\id\otimes\coev}\\
      \ldual{X}\ar{rr}[swap]{\id}\urar{\ev\otimes\id}&&\ldual{X}
    \end{tikzcd}
  \]
  In this case, we say that $X$ is \emph{left dualisable} and $\ldual{X}$ is
  \emph{right dualisable}. An object $X\in\M$ is \emph{dualisable} if it is left and
  right dualisable. Finally, the monoidal $\infty$-category
  $\M$ is \emph{rigid} if every object is left and right dualisable.
\end{definition}

\begin{remark}
  \label{rmk:rigidity_can_be_checked_on_HoM}
  A monoidal $\infty$-category $\M$ is rigid if and only if its monoidal
  homotopy category $\Ho{\M}$ is rigid in the usual sense, compare
  with~\cite[Definitions~2.10.11]{EGNO15}.
\end{remark}

\begin{notation}
  In the context of \Cref{def:duality_datum}, the object $\ldual{X}\in\M$ is
  uniquely determined up to isomorphism in $\Ho{\M}$, see for
  example~\cite[Proposition~2.10.5]{EGNO15}. If we set $Y\coloneqq\ldual{X}$, we
  also write $\rdual{Y}\coloneqq X$. In this case, we say that $\ldual{X}$ is a
  \emph{left dual} of $X$, and that $\rdual{Y}$ is a \emph{right dual} of $Y$.
  Moreover, $\rdual{(\ldual{X})}\simeq X$ and $\ldual{(\rdual{Y})}\simeq Y$.
\end{notation}

\begin{remark}
  \label{rmk:dualisables_monoidal_subcat}
  Let $\M$ be monoidal $\infty$-category and $X,Y\in\M$ a pair of dualisable
  objects. It is easy to show that the $\ldual{Y}\otimes\ldual{X}$ is a left
  dual of $X\otimes Y$ and, similarly, $\rdual{Y}\otimes\rdual{X}$ is a right
  dual, see for example~\cite[Exercise~2.10.7]{EGNO15}. Consequently, since the
  unit object $\unit\in\M$ is easily seen to be dualisable, the dualisable
  objects of $\M$ form a monoidal subcategory.
\end{remark}

\begin{remark}
  \label{rmk:rigid_internal_homs}
  Let $\M$ be a monoidal $\infty$-category and $X\in\M$ a dualisable object. As
  an immediate consequence of the triangle identities, there are adjunctions
  \[
    (-\otimes X)\dashv(-\otimes\ldual{X})\qquad\text{and}\qquad(X\otimes-)\dashv(\rdual{X}\otimes-),
  \]
  compare with~\cite[Proposition~2.10.8]{EGNO15}.
  Consequently, if $\M$ is rigid, then it is biclosed with right internal $\operatorname{Hom}$
  functor
  \[
   \hom<\M>[r]{X}{Y}\coloneqq Y\otimes\ldual{X},\qquad X,Y\in\M,
  \]
  and left internal $\operatorname{Hom}$ functor
  \[
    \hom<\M>[l]{X}{Y}\coloneqq\rdual{X}\otimes Y,\qquad X,Y\in\M.
  \]
  In  particular, for a fixed object $X\in\M$, the functors
  and
  \[
    \hom<\M>[r]{X}{-}=(-\otimes\ldual{X})\qquad\text{and}\qquad\hom<\M>[l]{X}{-}=(\rdual{X}\otimes-)
  \]
  preserve colimits (for they have further right adjoints).
\end{remark}

The following definition is standard in tensor-triangular geometry, see for
example~\cite[Section~2.1]{Ste18}.

\begin{definition}
  A presentably symmetric monoidal stable $\infty$-category $\K$ is
  \emph{rigidly-compactly generated} if it is compactly generated as an
  $\infty$-category\footnote{Equivalently, $\Ho{\K}$ is a compactly-generated
    triangulated category.} and the dualisable objects of $\K$ are precisely its
  compact objects. In this case, the full subcategory $\K^\omega\subseteq\K$ of
  compact objects is a rigid monoidal stable subcategory of $\K$ and, in particular,
  contains the unit object (\Cref{rmk:dualisables_monoidal_subcat}).
\end{definition}

\begin{example}
  The following are examples of rigidly-compactly generated symmetric monoidal
  stable $\infty$-categories:
  \begin{itemize}
  \item The derived $\infty$-category of (differential graded) modules over a
    commutative (differential graded) ring.
  \item The stable $\infty$-category $\Spectra$ of spectra. More generally, the
    stable $\infty$-category of module spectra over a commutative ring spectrum.
  \item The stable $\infty$-category of modules over a finite-dimensional
    cocommutative Hopf algebra.
  \end{itemize}
\end{example}

Our setup also permits us to consider graded variants of the theory (compare
with~\Cref{variant:graded}).

\begin{proposition}
  \label{prop:grading}
  Let $G$ be an abelian group, considered as a monoidal category with only
  identity morphisms. Let $\K$ be a rigidly-compactly generated symmetric
  monoidal stable $\infty$-category. Then, the $\infty$-category
  \[
    \textstyle\Fun{G}{\K}\simeq\prod_{g\in G}\K
  \]
  of functors $G\to\K$ is a rigidly-compactly generated symmetric monoidal
  stable $\infty$-category. Here, $\Fun{G}{\K}$ is endowed with the Day
  convolution tensor product described in~\cite[Section~2.4]{Lur15}.
\end{proposition}
\begin{proof}
  The statement is well-known to experts; we provide a proof for the sake of
  completeness. We begin with the following observations:
  \begin{itemize}
  \item The $\infty$-category $\Fun{G}{\K}$ is stable since limits and colimits
    in functor $\infty$-categories are computed pointwise,
    see~\cite[Proposition~1.1.3.1]{Lur17}.
  \item That the $\infty$-category $\Fun{G}{\K}$ is compactly generated follows
    from the equivalence of $\infty$-categories (see for
    example~\cite[p.~281]{Jas24a})
    \[
      \Fun{G}{\K}\simeq\Fun{G}{\Spaces}\otimes_{\Spaces}\K.
    \]
    Indeed, $\Fun{G}{\Spaces}$ is compactly
    generated~\cite[Proposition~5.3.5.12]{Lur09} the tensor product of
    compactly-generated $\infty$-categories is compactly generated
    (see~\cite[Example~5.3.6.8]{Lur09} and~\cite[Remark~4.8.1.8]{Lur17}). Here,
    $\Spaces$ denotes the $\infty$-category of spaces, which is the monoidal
    unit of the $\infty$-category $\PrL$.
  \end{itemize}

  The Day convolution tensor product on $\Fun{G}{\K}$ is described in
  \cite[Section~2.4]{Lur15} in the special case where $\K=\Spectra$ is the
  stable $\infty$-category of spectra. More
  precisely,~\cite[Corollary~2.3.9]{Lur15} shows that there is a symmetric
  monoidal functor
  \[
    \cats\longrightarrow\PrStL,\qquad \C\longmapsto \Fun{\C^\op}{\Spectra},
  \]
  where $\cats$ denotes the $\infty$-category of small $\infty$-categories
  endowed with the cartesian monoidal structure. Thus, if we regard $G$ as a
  monoidal category with only identity morphisms, then $\Fun{G}{\Spectra}$
  admits a symmetric monoidal structure whose tensor product preserves colimits
  in each variable separately. We also need the explicit formula for the Day
  convolution tensor product~\cite[Remark~2.3.10]{Lur15}:
  \[
    \textstyle(X^\bullet\otimes Y^\bullet)^g=\coprod_{h+k=g}X^h\otimes Y^k,\qquad X^\bullet,Y^\bullet\in\Fun{G}{\Spectra}.
  \]
  The symmetric monoidal structure on $\Fun{G}{\K}$ is then obtained via the
  equivalences of $\infty$-categories
  \begin{align*}
    \Fun{G}{\K}&\simeq\Fun{G}{\Spaces}\otimes_{\Spaces}\K\\
               &\simeq\Fun{G}{\Spaces}\otimes_{\Spaces}(\Spectra\otimes_{\Spectra}\K)\\
               &\simeq\Fun{G}{\Spectra}\otimes_{\Spectra}\K,
  \end{align*}
  where we use that $\Spectra$ is the unit of the tensor product in $\PrStL$.

  It remains to show that $\Fun{G}{\K}$ is rigidly-compactly generated. Firslty, recall that the
  $\infty$-category $\Fun{G}{\Spaces}$ admits the representable presheaves
  $G(-,g)$, $g\in G$, as a canonical set of compact generators. Secondly,
  we note that
  \[
    G(h,g)=\begin{cases}*&h=g,\\\emptyset&h\neq g.\end{cases}
  \]
  For an object $X\in\K$ and an element $g\in G$, set
  $X(g)\coloneqq G(-,g)\otimes X$. Under the equivalence
  \[
    \textstyle\Fun{G}{\Spaces}\otimes_{\Spaces}\K\simeq\Fun{G}{\K}\simeq\prod_{g\in\G}\K,
  \]
  the object $X(g)$ corresponds to the tuple with value $X$ in degree $g$ and
  that vanishes in every other degree. The set
  \[
    \set{X(g)\in\Fun{G}{\K}}[g\in G,\ X\in\K^\omega]
  \]
  forms a set of compact generators of $\Fun{G}{\K}$, compare
  with~\cite[Remark~2.4.1]{Lur15}. In particular, the unit object
  $\unit[\K](0)\in\Fun{G}{\K}$ is compact. From this, it follows that every
  dualisable object of $\Fun{G}{\K}$ is compact: For an object
  $X^\bullet\in\Fun{G}{\K}$, there are equivalences
  \[
    \Map{X^\bullet}{-}\simeq\Map{\unit[\K](0)}{\hom{X^\bullet}{-}}
  \]
  of functors $\Fun{G}{\K}\to\Spaces$; hence, the claim follows from the
  fact that the functor
  \[
    \hom{X^\bullet}{-}\colon\Fun{G}{\K}\longrightarrow\Fun{G}{\K}
  \]
  preserves all colimits (\Cref{rmk:rigid_internal_homs}). We now show,
  conversely, that every compact object of $\Fun{G}{\K}$ is dualisable. Let
  $X^\bullet\in\Fun{G}{\K}$ be a compact object. The above description of a set
  of compact generators for $\Fun{G}{\K}$ implies that
  \[
    \textstyle X^\bullet\simeq \bigoplus_{i=1}^nX_i(g_i)
  \]
  for some $n\geq1$, compact objects $X_1,\dots,X_n\in\K^\omega$ and elements
  ${g_1,\dots,g_n\in G}$. Therefore,
  it suffices to show that $X_i(g_i)$ is dualisable for each $1\leq i\leq n$.
  Considering the tensor products
  \[
    \unit[\K](-g_i)\otimes X_i(g_i)\simeq X_i(0),\qquad 1\leq i\leq n,
  \]
  and noticing that $\unit[\K](g)$ is $\otimes$-invertible for each $g\in G$, we
  are reduced to show that $X(0)\in\Fun{G}{\K}$ is dualisable for each compact
  object $X\in\K^\omega$. This follows from the fact that the objects of
  $\Fun{G}{\K}$ that are concentrated in degree $0\in G$ form a monoidal
  subcategory and the assumption that every compact object of $\K$ is
  dualisable. This finishes the proof.
\end{proof}

\begin{example}
  Let $\kk$ be a commutative ring spectrum (for example, the Eilenberg--Mac Lane
  ring spectrum of an ordinary commutative ring). By \Cref{prop:grading}, the
  $\infty$-category
  \[
    \textstyle\Fun{\ZZ}{\LMod{\kk}[\Spectra]}\simeq\prod_{n\in\ZZ}\LMod{\kk}[\Spectra]
  \]
  of $\ZZ$-graded $\kk$-module spectra is a rigidly-compactly generated stable
  $\infty$-category.
\end{example}

\subsection{The base monoidal $\infty$-category}

\begin{setting}
  \label{setting:spectral_Hopfological_algebra}
  Throughout this section, we fix a rigidly-compactly generated symmetric
  monoidal stable $\infty$-category $\K=(\K,\otimes,\unit)$, which plays the
  role of the symmetric monoidal stable $\infty$-category $\DerCat{\Mod{\kk}}$ in
  what follows.
\end{setting}

Let $\M$ be a symmetric monoidal $\infty$-category that admits geometric
realisations of simplicial objects, and suppose that these are preserved by the
tensor product in each variable separately. Then, the $\infty$-category
$\Alg{\M}$ of algebra objects of $\M$ admits a monoidal structure such
that the forgetful functor ${\Alg{\M}\to\M}$ is
monoidal~\cite[Example~3.2.4.4]{Lur17}. This permits us to make the following
definition, where we write
\[
  \coAlg{\M}\coloneqq\Alg{\M^\op}^\op
\]
for the
  $\infty$-category of \emph{coalgebra objects} in $\M$.

\begin{definition}
  The $\infty$-category of \emph{bialgebra objects} in $\K$ is
  \[
    \biAlg{\K}\coloneqq\coAlg{\Alg{\K}}.
  \]
\end{definition}

\begin{remark}
  There is an equivalence of $\infty$-categories~\cite[Corollary~A.0.17]{Erg22}
  \[
    \biAlg{\K}\simeq\Alg{\coAlg{\K}}.
  \]
\end{remark}

\begin{remark}
  \label{rmk:auxiliary-bialgebras}
  A bialgebra object $B\in\biAlg{\K}$ is equipped, in particular, with a
  commutative diagram of algebra morphisms
  \[
    \begin{tikzcd}[column sep=small]
      &\unit\\
      B\urar{\varepsilon}&&\unit\ar{ll}{\eta}\ular[swap]{\id}
    \end{tikzcd}
  \]
  Using \Cref{prop:change_of_algebra} and the canonical identification
  $\LMod{\unit}[\K]\simeq\K$. see~\cite[Proposition~3.4.2.1]{Lur17}, we obtain a
  commutative diagram of colimit-preserving functors
  \[
    \begin{tikzcd}[column sep=small]
      &\K\drar{\id}\dlar[swap]{\varepsilon^*}\\
      \LMod{B}[\K]\ar{rr}[swap]{\eta^*}&&\K
    \end{tikzcd}
  \]
\end{remark}

\begin{proposition}[{\cite[Proposition~3.16 and~Corollary~3.19]{Bea23}
    and~\cite[Section~3.2]{Lei22}}]
  \label{prop:LModBialg}
  Let $B\in\biAlg{\K}$ be a bialgebra object. The following statements hold:
  \begin{enumerate}
  \item The $\infty$-category $\LMod{B}[\K]$ admits a monoidal structure.
  \item The canonical colimit-preserving functors
    \[
      \K\longrightarrow\LMod{B}[\K]\qquad\text{and}\qquad\LMod{B}[\K]\longrightarrow\K
    \]
    described in \Cref{rmk:auxiliary-bialgebras} are monoidal with respect to
    the above monoidal structure on $\LMod{B}[\K]$.
  \end{enumerate}
\end{proposition}

\begin{remark}
  Let $B\in\biAlg{\K}$ be a bialgebra object. It follows from
  \Cref{prop:LModBialg} that we may regard the monoidal $\infty$-category
  $\LMod{B}$ as an augmented algebra object of the $\infty$-category
  $\Mod{\K}[\PrStL]$ of $\K$-linear presentable stable $\infty$-categories
  (compare with~\Cref{rmk:section_identity_DH}).
\end{remark}

The following definition is motivated by the well-known characterisations of
finite-dimensional Hopf algebras via Tannaka--Krein Duality, see for
example~\cite[Theorems~5.2.3 and~5.3.12]{EGNO15}.

\begin{definition}
  \label{def:Hopf_algebra_object}
  Let $H\in\biAlg{\K}$ be a bialgebra object whose underlying object of $\K$ is
  compact (equivalently, dualisable). We say that $H$ is a \emph{Hopf algebra
    object} if the monoidal subcategory
  $\LMod{H}[\K^\omega]\subseteq\LMod{H}[\K]$ is rigid. Here, we identify
  $\LMod{H}[\K^\omega]$ with the full subcategory of $\LMod{H}[\K]$ spanned by
  the $H$-modules whose underlying object is compact in $\K$.
\end{definition}

\begin{remark}
  For any bialgebra \(H\), the functor
  \[
    \Map{H}{-} \colon \Alg{\K} \to \Spaces
  \]
  admits a canonical lift
  \[
    \Map{H}{-} \colon \Alg{\K} \to \Alg{\Spaces},
  \]
  where \(\Alg{\Spaces}\) denotes the \(\infty\)-category of associative monoids
  in the \(\infty\)-category of spaces. In \cite[Definition
  3.9.7]{lurie2017elliptic}, the author defines a bialgebra\footnote{Note that $H$ is assumed to be commutative and cocommutative in~\cite{lurie2017elliptic}.} to be a Hopf algebra
  if this functor factors through \emph{grouplike} associative monoids in
  \(\Spaces\) or, equivalently, if \(\pi_{0}\Map{H}{-}\colon \Alg{\K} \to \operatorname{Mon}\) factors the \(1\)-category of groups. 
It is expected that this definition agrees with the one given above; however, we leave the verification of this equivalence to the interested reader.
\end{remark}

\begin{proposition}
  \label{prop:perf_in_pvd-tensor_ideal}
  Let $H$ be a Hopf algebra object of $\K$. Then, $\LMod{H}[\K]^\omega$ is a
  two-sided $\otimes$-ideal in $\LMod{H}[\K^\omega]$.
\end{proposition}
\begin{proof}
  We prove first that every compact object of $\LMod{H}[\K]$ lies in
  $\LMod{H}[\K^\omega]$. For this, it suffices to observe that the forgetful
  functor $\LMod{H}[\K]\to\K$ preserves compact objects, for its right adjoint
  \[
    \hom<\K>{H}{-}\simeq(\ldual{H}\otimes-)\colon\K\longrightarrow\LMod{H}[\K]
  \]
  preserves filtered colimts (in fact, it preserves all colimits); here, we use
  the implicit assumption that the underlying object of $H$ is compact in $\K$,
  hence dualisable.

  To show that $\LMod{H}[\K]^\omega$ is a two-sided $\otimes$-ideal in
  $\LMod{H}[\K^\omega]$, we observe that, by \Cref{thm:props_of_Mod}, it is
  enough to prove the following statement:
  \begin{itemize}
  \item For each pair of compact objects $Y'\in\K^\omega$ and for each pair
    of objects $X,Y\in\LMod{H}[\K^\omega]$, we have
    \[
      X\otimes(H\otimes Y')\otimes Y\in \LMod{H}[\K]^\omega.
    \]
  \end{itemize}
  Since
  \[
    X\otimes(H\otimes Y')\otimes Y\simeq X\otimes H\otimes
    (Y'\otimes Y),
  \]
  we are reduced to prove that $X\otimes H$ and $H\otimes Y$ are compact in
  $\LMod{H}[\K]$. Keeping in mind \Cref{rmk:rigid_internal_homs}, this follows
  from the fact that $H\in\LMod{H}[\K]$ is compact, for there are equivalences
  \begin{align*}
    \Map<H>{X\otimes
    H}{-}&\simeq\Map<H>{H}{\hom<H>[l]{X}{-}}\intertext{and}\Map<H>{H\otimes
           Y}{-}&\simeq\Map<H>{H}{\hom<H>[r]{Y}{-}}
  \end{align*}
  of functors $\LMod{H}[\K]\to\Spaces$.
\end{proof}

\Cref{prop:perf_in_pvd-tensor_ideal} permits us to make the following
definition.

\begin{defprop}
  Let $H$ be a Hopf algebra object of $\K$. The \emph{stable $\infty$-category
    of $H$-module objects in $\K$} is the rigidly-compactly generated monoidal
  stable $\infty$-category
  \[
    \StLMod{H}[\K]\coloneqq\Ind[\LMod{H}[\K^\omega]/\LMod{H}[\K]^\omega].
  \]
\end{defprop}
\begin{proof}
  In view of \Cref{prop:perf_in_pvd-tensor_ideal}, the Verdier quotient
  \[
    \LMod{H}[\K^\omega]/\LMod{H}[\K]^\omega
  \]
  is well defined an inherits a monoidal structure from $\LMod{H}[\K^\omega]$
  such that the canonical localisation functor
  \[
    \LMod{H}[\K^\omega]\longrightarrow\LMod{H}[\K^\omega]/\LMod{H}[\K]^\omega
  \]
  is monoidal~\cite[Proposition~2.2.1.9]{Lur17}. Since monoidal functors clearly preserve dualisable
  objects and the localisation functor is dense, we conclude that
  $\LMod{H}[\K^\omega]/\LMod{H}[\K]^\omega$ is a rigid monoidal
  $\infty$-category. Finally, since $\StLMod{H}[\K]$ is idempotent complete, its
  collection of dualisable objects is closed under retracts. This shows
  that $\StLMod{H}$ is indeed a rigidly-compactly generated monoidal stable
  $\infty$-category.
\end{proof}

By construction, a version of Krause's recollement
(\Cref{thm:Krauses-recollement}) is also available in this context.

\begin{proposition}
  \label{prop:Krauses-recollement-general}
  There is a recollement of compactly-generated
  stable $\infty$-categories
  \begin{center}
    \begin{tikzcd}
      \StLMod{H}[\K]\ar[hookrightarrow]{r}[description]{i}&\Ind[\LMod{H}[\K^\omega]]\ar{r}[description]{p}\ar[shift
      right=0.5em]{l}[swap]{i_L}\ar[shift
      left=0.5em]{l}{i_R}&\LMod{H}[\K].\ar[shift
      right=0.5em,hook]{l}[swap]{p_L}\ar[hook',shift left=0.5em]{l}{p_R}
    \end{tikzcd}
  \end{center}
  Moreover, in the above recollement,
  \begin{itemize}
  \item the functor
    \[
      p_L\colon\LMod{H}[\K]\hooklongrightarrow\Ind[\LMod{H}[\K^\omega]]
    \]
    is the
    inclusion of a two-sided $\otimes$-ideal,
  \item the functor
    \[
      i\colon\StLMod{H}[\K]\hooklongrightarrow\Ind[\LMod{H}[\K^\omega]]
    \]
    is the inclusion of a two-sided $\otimes$-ideal.
  \end{itemize}
\end{proposition}
\begin{proof}
  The required recollement is obtained by applying the
  (monoidal) $\Ind$-completion functor to the defining localisation sequence
  \[
    \begin{tikzcd}
      \LMod{H}[\K^\omega]/\LMod{H}[\K]^\omega&\LMod{H}[\K^\omega]\lar[two
      heads]&\LMod{H}[\K]^\omega\lar[hook'],
    \end{tikzcd}
  \]
  see for example~\cite[Theorem 5.6.1]{Kra10}. From this, the claim about the
  functor $p_L$ follows immediately.
   Now, we verify that the essential image of $i$ is a two-sided ideal. Since
   the functor $i$ preserves colimits and the tensor product on $\StLMod{\K}$
   preserves colimits in each variable separately, it is enough to show that for each pair
   of objects
   $X,Y\in\LMod{H}[K^\omega]$ and each object $Z\in\StLMod{H}[\K]$, the object
   \[
     X\otimes i(Z)\otimes Y
   \]
   lies in essential image of $i$.
   The existence of the recollement implies that this is the case if and
   only if
   \[
     \Hom{p_L(W)}{X\otimes i(Z)\otimes Y}=0,\qquad W\in\LMod{H}[\K].
   \]
   Given that $H$ is a Hopf algebra object of $\K$, both $X$ and $Y$ are dualisable. It
   follows that
   \[
     \Hom{p_L(W)}{X\otimes i(Z)\otimes Y}\simeq \Hom{X^\vee\otimes p_L(W)\otimes {}^\vee
       Y}{i(Z)}\stackrel{!}{=}0,
   \]
   where the latter $\operatorname{Hom}$ vanishes since the essential image of
   $p_L$ is a two-sided $\otimes$-ideal. In fact, a similar reasoning shows that
   the essential image of $i$ is closed under the left and right internal
   $\operatorname{Hom}$ functors.
\end{proof}

\subsection{Hopfological derived $\infty$-categories}

\begin{definition}
  \label{def:Hopfological_der-cat-K}
  Let $H$ be a Hopf algebra object of $\K$ and $A\in\Alg{\StLMod{H}[\K]}$ an
  algebra object $\StLMod{H}[\K]$. The \emph{Hopfological derived
    $\infty$-category} of $A$ is the compactly-generated stable
  $\infty$-category
  \[
    \DerCat{A,H}_\K\coloneqq\LMod{A}[\StLMod{H}[\K]].
  \]
\end{definition}

In order to illustrate \Cref{def:Hopfological_der-cat-K}, we record an analogue of
\Cref{thm:splitting_formula} in this setting.

\begin{theorem}
  \label{thm:splitting_formula-K}
  Suppose that $H$ acts trivially on $A\in\Alg{\LMod{H}[\K]}$, so that $A$ is simply
  the datum of an algebra object of $\K$. Then, there is an equivalence of
  $\infty$-categories
  \[
    \DerCat{A,H}_\K\simeq\LMod{A}[\K]\otimes_{\K}\StLMod{H}[\K]
  \]
  right-tensored over $\StLMod{H}[\K]$.
\end{theorem}
\begin{proof}
  The assumption
  that $H$ acts trivially on $A$ means that $A$ lies in the image of the
  functor
  \[
    \Alg{\K}\longrightarrow\Alg{\StLMod{H}[\K]}
  \]
  induced by the canonical colimit-preserving monoidal functor
  $\K\to\StLMod{H}[\K]$ from \Cref{prop:LModBialg}. Hence,
  \Cref{thm:props_of_Mod} yields the claimed equivalence of $\infty$-categories
  \[
    \DerCat{A,H}_\K=\LMod{A}[\StLMod{H}[\K]]\simeq\LMod{A}[\K]\otimes_{\K}\StLMod{H}[\K].\qedhere
  \]
\end{proof}

\subsection{Hopfological analogue of Krause's recollement}

We now explain, in the generality of \Cref{def:Hopfological_der-cat-K}, how to lift Krause's recollement
(\Cref{prop:Krauses-recollement-general}) to Hopfological derived
$\infty$-categories. We begin with some recollections. As explained
in~\cite[Remark 1.1.11]{Erg22}, given a lax monoidal functor $f\colon\C\to \D$
between presentably monoidal $\infty$-categories, and an algebra object $A$ of
$\C$, there is a commutative diagram
\[
\begin{tikzcd}
      \LMod{A}[\C]\ar{r}{\bar{f}}\ar{d}{U}&\LMod{f(A)}[\D]\ar{d}{U}\\
      \C\ar{r}{f} & \D 
    \end{tikzcd}
\]
in which the vertical functors are the corresponding forgetful functors. In
other words, the extension-of-scalars functor $\bar{f}$ is compatible with the
forgetful functor.  Moreover, if $f$ is strongly monoidal and commutes with
geometric realisations of simplicial objects, then $\bar{f}$ is also compatible
with the corresponding free-module functors. That is, there is a commutative diagram
\[
\begin{tikzcd}
      \LMod{A}[\C]\ar{r}{\bar{f}}&\LMod{f(A)}[\D]\\
      \C\ar{r}{f} \ar{u}{F_A} & \D \ar{u}{F_{f(A)}}
    \end{tikzcd}
\]

\begin{remark}
\label{rem-adjointability-modules}
 Let $\K$ be a rigidly-compactly generated symmetric monoidal stable
 $\infty$-category, let $H$ be a Hopf algebra object of $\K$, and let $A\in
 \Alg{\Ind[\LMod{H}[\K^\omega]]}$. Using the notation of
 \Cref{prop:Krauses-recollement-general}, set $A_l\coloneqq i_L(A)$ and
 $A_r\coloneqq p(A)$. Then $A_l$ is an algebra object of $\StLMod{H}[\K]$ (since
 $i_L$ is monoidal), and $A_r$ is an algebra object of $\LMod{H}[\K]$ (since $p$
 is monoidal). There is a diagram
\[
    \begin{tikzcd}
       \DerCat{A_l,H}_\K \ar[dotted]{r}[description]{j}\ar{d}[description]{U}& \LMod{A}[\Ind[\LMod{H}[\K^\omega]]] \ar{d}[description]{U} \ar{r}[description]{q}\ar[shift
      right=0.5em]{l}[swap]{j_L}\ar[dotted,shift
      left=0.5em]{l}{j_R}&\LMod{A_r}[\LMod{H}[\K]]\ar[shift
      right=0.5em,dotted]{l}[swap]{q_L}\ar[dotted,shift left=0.5em]{l}{q_R}\ar{d}[description]{U}\\
      \StLMod{H}[\K]\ar[hookrightarrow]{r}[description]{i}\ar[shift left=0.5em]{u}{F_{A_l}} &\Ind[\LMod{H}[\K^\omega]]\ar{r}[description]{p}\ar[shift
      right=0.5em]{l}[swap]{i_L}\ar[shift
      left=0.5em]{l}{i_R}\ar[shift left=0.5em]{u}{F_{A}}&\LMod{H}[\K]\ar[shift
      right=0.5em,hook]{l}[swap]{p_L}\ar[hook',shift
      left=0.5em]{l}{p_R}\ar[shift left=0.5em]{u}{F_{A_r}}
    \end{tikzcd}
  \]
  The existence of the dotted functors follows from the Adjoint Functor Theorem
  \cite[Corollary 5.5.2.9]{Lur17}, together with the following observations.
  Since the forgetful functors are conservative and preserve small limits and
  small colimits, we deduce:
  \begin{itemize}
  \item The functor $j_L$ admits a right adjoint $j$, because $i$ does.
  \item The functor $j$ admits a right adjoint, because $i$ does.
  \item The functor $q$ admits both a left and a right adjoint, because $p$ does.
  \end{itemize}
  Moreover, by the recollection above, the following identities hold:
  \begin{align*}
    U\circ j_L&\simeq i_L\circ U&U\circ q&\simeq p\circ U\\
    j_L\circ F_A&\simeq F_{A_l}\circ i_L&q\circ F_A&\simeq F_{A_r}\circ p\\
    i\circ U&\simeq U\circ j&p_R\circ U&\simeq U\circ q_R.
  \end{align*}
  The last identities are obtained from the previous ones by passing to right adjoints.  
\end{remark}

\begin{theorem}
  \label{thm:Krauses-recollement-Hopfological-general}
  With the notation established in \Cref{rem-adjointability-modules},
 there is a commutative diagram of recollements of compactly-generated stable $\infty$-categories
  \[
    \begin{tikzcd}
       \DerCat{B,H}_\K \ar[hookrightarrow]{r}[description]{j}\dar& \LMod{A}[\Ind[\LMod{H}[\K^\omega]]] \dar\ar{r}[description]{q}\ar[shift
      right=0.5em]{l}[swap]{j_L}\ar[shift
      left=0.5em]{l}{j_R}&\LMod{A_r}[\LMod{H}[\K]]\ar[shift
      right=0.5em,hook]{l}[swap]{q_L}\ar[hook',shift left=0.5em]{l}{q_R}\dar\\
      \StLMod{H}[\K]\ar[hookrightarrow]{r}[description]{i}&\Ind[\LMod{H}[\K^\omega]]\ar{r}[description]{p}\ar[shift
      right=0.5em]{l}[swap]{i_L}\ar[shift
      left=0.5em]{l}{i_R}&\LMod{H}[\K]\ar[shift
      right=0.5em,hook]{l}[swap]{p_L}\ar[hook',shift
      left=0.5em]{l}{p_R}
    \end{tikzcd}
  \]
  where the vertical functors are the corresponding forgetful functors.
\end{theorem}
\begin{proof}
The proof of \Cref{thm:Krauses-recollement-Hopfological} carries over verbatim to this setting; we briefly indicate the main steps.

To show that the top row forms a recollement, it suffices to verify the following: 

  \begin{itemize}
  \item The functors $j$ and $q_R$ are fully faithful. Since the case of $q_R$
    can be shown analogously, we only prove that $j$ is fully faithful. For
    this, it suffices to prove that, for each object $X\in\DerCat{B,H}_\K$, the counit map $\varepsilon_X\colon j_L
    j(X)\to X$ of the adjunction $j_L\dashv j$ is invertible. Applying the identities from \Cref{rem-adjointability-modules}, we see that $U(\varepsilon_X)$ identifies with the counit map $i_L i(U(X))\to U(X)$ of the adjunction $i_L\dashv i$, which is an equivalence since $i$ is fully faithful. As $U$ is conservative, the claim follows.
  \item The essential image if $j$ is precisely the kernel of $q$. To prove
    this, we observe first that
    \[
      U\circ q\circ j\simeq p\circ i\circ U\simeq 0.
    \]
    Since the functor $U$ is conservative, we deduce that $q\circ i\simeq 0$.
    Suppose now that $X\in\LMod{A}[\Ind[\LMod{H}[\K^\omega]]]$ lies in the kernel of $q$. We
    need to prove that the unit map $\eta_X\colon X\to jj_L(X)$ is invertible. It
    suffices to show that $U(\eta_X)$ is invertible (again, since $U$ is
    conservative). The commutativity of the diagram implies that $U(\eta_X)$ is
    equivalent to the unit map $\alpha_{U(X)}\colon U(X)\to ii_L(U(X))$; to prove that the latter
    map is invertible, it is enough to notice that $p(U(X))\simeq
    U(q(X))\simeq0$ by assumption. Since the kernel of $p$ is precisely the
    image of $i$, it follows that $\alpha_{U(X)}$ is invertible, as required.
  \end{itemize}
This completes the proof.
\end{proof}

\appendix

\section{Hopfological vs.~$Q$-shaped derived categories}

Recently, Holm and J{\o}rgensen introduced certain analogues of classical
derived categories called $Q$-shaped derived categories~\cite{HJ22}; a useful
survey on this emerging research stream can be found in~\cite{HJ24a}. In this
appendix, we attempt to elucidate the formal similarities in the constructions of
Hopfological derived categories and of $Q$-shaped derived categories,
see~\Cref{rmk:H-vs-Q}.

Recall that, given a category $\A$, a monad on $\A$ is a monoid object $T$ in the
category of endofunctors $\A\to\A$, regarded as a monoidal category under
functor composition. A $T$-algebra object in $\A$ is an object $X\in\A$ equipped
with an action morphism $TX\to X$ satisfying suitable axioms.

\begin{theorem}[Lurie]
  \label{thm:Lurie-transfer-model_structure-monad}
  Let $\A$ be a combinatorial model category and $T\colon\A\to\A$ a
  monad on $\A$ whose underlying functor is a left Quillen functor. The
  following statements hold:
  \begin{enumerate}
  \item\label{it:Lurie-transfer-model_structure-monad}\cite[Remark~4.3.3.16]{Lur17} The category $\Alg[T]{\A}$ of $T$-algebra
    objects in $\A$ admits a (right-transferred) combinatorial model structure determined as
    follows:
    \begin{itemize}
    \item A morphism in $\Alg[T]{\A}$ is a weak equivalence if its underlying
      morphism is a weak equivalence in $\A$.
    \item A morphism in $\Alg[T]{\A}$ is a fibration if its underlying morphism
      is a fibration in $\A$.
    \item A morphism in $\Alg[T]{\A}$ is a cofibration if it has the left
      lifting property with respect to the trivial fibrations.
    \end{itemize}
    Moreover, the forgetful functor $\Alg[T]{\A}\to\A$ is both a left Quillen
    functor and a right Quillen functor.
  \item\label{it:Lurie-transfer-model_structure-monad-equivalence}\cite[Theorem~4.3.3.17]{Lur17} There is a canonical equivalence of
    $\infty$-categories
    \[
      \MLoc{\Alg[T]{\A}}\stackrel{\sim}{\longrightarrow}\Alg[\mathbb{L}T]{\MLoc{\A}},
    \]
    where the right-hand side denotes the $\infty$-category of $\mathbb{L}T$-algebra
    objects in the monoidal $\infty$-category of $\MLoc{\A}$.
  \end{enumerate}
\end{theorem}
\begin{proof}
  \eqref{it:Lurie-transfer-model_structure-monad} In
  ~\cite[Proposition~4.3.3.15]{Lur17}, Lurie deals with the special case of a
  monad on a monoidal model category that is induced by tensor product with a
  pair of algebra objects whose underlying object is cofibrant, so that the
  resulting category of algebras is the corresponding category of bimodules
  (compare with~\cite[Example~4.7.3.9]{Lur17}). The proof carries over to this
  more general setting with essentially no modifications. Here, we only make a
  couple of observations for the convenience of the reader.
  \begin{itemize}
  \item By assumption, the monad $T\colon\A\to\A$ admits a right adjoint, say
    $S$. Consequently, $S$ admits a canonical comonad structure and the
    forgetful functor $\Alg[T]{\A}\to\A$ is part of an adjoint triple
    \[
      \begin{tikzcd}
        \Alg[T]{\A}\rar&\A;\lar[shift left=0.5em]\lar[shift right=0.5em]
      \end{tikzcd}
    \]
    moreover, the induced adjoint monad-comonad pair on $\A$ induced by this
    adjoint triple is precisely the adjoint pair formed by $T$
    and $S$, see for example~\cite[Theorem~3.7.7]{BW05}.
  \item Since $\A$ is presentable and the monad $T$ preserves small colimits,
    the category $\Alg[T]{\A}$ is presentable~\cite[Theorem~2.78]{AR94}.
  \end{itemize}
  With this in mind, the argument given in~\cite[Proposition~4.3.3.15]{Lur17}
  shows that $\Alg[T]{\A}$ admits the claimed combinatorial model structure.
  Since the forgetful functor $\Alg[T]{\A}\to\A$ is a right Quillen functor, it
  only remains to show that it is also a left Quillen functor or, equivalently,
  that its right adjoint is a right Quillen functor. Unravelling the definitions,
  the latter statement is equivalent to the assertion that the underlying
  functor of the right-adjoint comonad $S$ is a right Quillen functor, which is
  true by assumption.

  \eqref{it:Lurie-transfer-model_structure-monad-equivalence} The statement
  in~\cite[Theorem~4.3.3.17]{Lur17} concerns the setting
  of~\cite[Proposition~4.3.3.15]{Lur17}, but the proof given therein applies in
  this more general setting. The main point is that the assumptions of the
  Barr--Beck--Lurie Monadicity Theorem~\cite[Theorem~4.7.3.5]{Lur17} are
  satisfied:
  \begin{itemize}
  \item The $\infty$-categories $\MLoc{\Alg[T]{\A}}$ and $\MLoc{\A}$ are presentable
    (\Cref{thm:MLoc-presentable}) and hence admit geometric realisations of
    simplicial objects (in fact, they admit all small colimits).
  \item The functor $U\colon\MLoc{\Alg[T]{\A}}\to\MLoc{\A}$ induced by the forgetful
    functor $\Alg[T]{\A}\to\A$ is conservative. This follows immediately from
    the fact that the weak equivalences in $\Alg[T]{\A}$ are created by the
    forgetful functor.
  \item The functor $U\colon\MLoc{\Alg[T]{\A}}\to\MLoc{\A}$ admits a left
    adjoint. This follows from the fact
    that the forgetful functor $\Alg[T]{\A}\to\A$ is a right Quillen functor.
  \item The functor $U\colon\MLoc{\Alg[T]{\A}}\to\MLoc{\A}$ admits a
    right adjoint, hence it preserves geometric realisations of simplicial
    objects (in fact, it preserves all colimits). This follows from the fact
    that the forgetful functor $\Alg[T]{\A}\to\A$ is a left Quillen functor.
  \end{itemize}
  Therefore, \cite[Theorem~4.7.3.5]{Lur17} implies that there exists a canonical equivalence of $\infty$-categories
  \[
    \begin{tikzcd}
      \MLoc{\Alg[T]{\A}}\ar{rr}{\sim}\drar&&\Alg[\mathbb{L}T]{\MLoc{\A}}\dlar\\
      &\MLoc{\A}
    \end{tikzcd}
  \]
  that is compatible with the corresponding forgetful functors.
\end{proof}

\begin{remark}
  \label{rmk:monad-algebra}
  \Cref{thm:Lurie-transfer-model_structure-monad} is a generalisation of
  \Cref{thm:Lurie-transfer-model_structure}. Indeed, the latter theorem is
  recovered from \Cref{thm:Lurie-transfer-model_structure-monad} by considering
  the apparent monad $T=A\otimes-\colon\M\to\M$, where $A\in\Alg{\M}$ is the given
  algebra whose underlying object is cofibrant.
\end{remark}

In what follows we adopt the assumptions in~\cite[Setup~2.9]{HJ24}. In
particular, we fix a commutative hereditary noetherian ring $\kk$, for example
$\kk=\ZZ$ or a field.

\begin{theorem}[{\cite[Theorem~6.1(a)]{HJ22}}]
  \label{thm:HJ}
  Let $Q$ be a small $\kk$-category satisfying the
  assumptions in~\cite[Setup~2.9]{HJ24}. The Grothendieck category $\LMod{Q}$ of
  $\kk$-linear functors $Q\to\Mod{\kk}$, called $Q$-modules, admits a hereditary
  abelian model structure determined by the triple of subcategories
  $(\C,\W,\LMod{Q})$, where $\W$ is the class of $Q$-modules of finite
  projective/injective dimension (see~\cite[Definition~4.1]{HJ22}) and
  \[
    \C\coloneqq\set{X\in\LMod{Q}}[\forall Y\in\W,\ \Ext<Q>[1]{X}{Y}=0].
  \]
  This is called the \emph{projective model structure} on $\LMod{Q}$.
\end{theorem}

\begin{remark}
  We do not discuss the injective model structure
  described in~\cite[Theorem~6.1(b)]{HJ22} in this appendix.
\end{remark}

\Cref{thm:HJ} permits us to make the following definition
(see~\Cref{subsec:Quillen-model-cats}, \Cref{thm:Gillespie} and
\Cref{thm:Lurie-transfer-model_structure-monad}). The reader should compare this
with the definition of the stable $\infty$-category of modules over a Hopf
algebra (\Cref{def:StModH}).

\begin{definition}
  Let $Q$ be a small $\kk$-category satisfying the assumptions
  in~\cite[Setup~2.9]{HJ24}. The \emph{$Q$-shaped derived $\infty$-category of
    $\kk$} is the underlying $\infty$-category
  \[
    \DerCat<Q>{\kk}\coloneq\MLoc{\LMod{Q}}
  \]
  of the model structure described in \Cref{thm:HJ}. This is a presentable
  stable $\infty$-category whose homotopy category is the $Q$-shaped derived
  category of $\kk$ defined in~\cite{HJ22}. Moreover, $\DerCat<Q>{\kk}$ is
  compactly generated by~\cite[Theorem~D]{HJ24}
  and~\cite[Remark~1.4.4.3]{Lur17}.
\end{definition}

\begin{proposition}
  \label{prop:HJ-Lurie}
  Let $\kk$ be a commutative hereditary noetherian ring and $Q$ a small
  $\kk$-category satisfying the assumptions in~\cite[Setup~2.9]{HJ24}. Let $A$
  be $\kk$-algebra whose underlying $\kk$-module is flat.\footnote{Compare with
    the assumptions in~\cite[Theorem~1.1]{Jas25}.} Then, the Grothendieck
  category $\LMod{A\otimes Q}$ of $\kk$-linear functors ${A\otimes Q\to\Mod{\kk}}$ admits a
  combinatorial model structure, which is obtained by applying
  \Cref{thm:Lurie-transfer-model_structure-monad} to the monad $T=U(A\otimes-)$
  on $\LMod{Q}$, where the functor
  \[
    U\colon\LMod{Q\otimes A}\longrightarrow\LMod{Q}
  \]
  is given by restriction along the canonical $\kk$-linear functor
  $Q\cong Q\otimes\kk\to Q\otimes A$.
  Moreover, this model structure agrees with the hereditary
  abelian model structure constructed in~\cite[Theorem~6.1(a)]{HJ22}.
\end{proposition}
\begin{proof}
  The structure map $\kk\to A$ induces a functor $Q\cong\kk\otimes Q\to A\otimes
  Q$. This functor gives rise to an adjoint triple
  \begin{equation*}
    \label{eq:change-of-rings_adjunction}
    \begin{tikzcd}[column sep=large]
      \LMod{Q}\rar[shift left=0.5em]{A\otimes -}\rar[shift
      right=0.5em,swap]{\Hom<\kk>{A}{-}}&\Mod{A\otimes Q},\ar{l}[description]{U}
    \end{tikzcd}
  \end{equation*}
  where we think of the functor $U\colon\Mod{A\otimes Q}\to\LMod{Q}$ as a
  forgetful functor, compare with~\cite[Corollary 3.5]{HJ22}. As with any
  adjunction, the resulting endofunctor $T\coloneqq U(A\otimes
  -)\colon\LMod{Q}\to\LMod{Q}$ is endowed with the structure of a monad. Moreover,
  the forgetful functor $U$ trivially satisfies the assumptions of the
  Barr--Beck Monadicity Theorem (even in the stronger form stated
  in~\cite[Proposition~3.7.6]{BW05}), and hence we may and we will identify
  $\Mod{A\otimes Q}$ with the category of $T$-algebra objects in $\LMod{Q}$. We
  need to show that the underlying functor of $T$ is a left Quillen functor so
  that we can apply \Cref{thm:Lurie-transfer-model_structure-monad}.

  Since the functors $U$ and $A\otimes-$ are exact (recall that $A$ is
  assumed to be $\kk$-flat), it es enough to show that the full subcategories
  $\C$ and $\W$ of $\LMod{Q}$ are closed under the action of $T$. For this, we
  observe first that the exact functors $U$ and $A\otimes-$ preserve projective
  objects, for they admit exact right adjoints. It readily follows that
  $T=U(A\otimes-)$ preserves objects of finite projective dimension, which is to
  say that $\W$ is closed under the action of $T$. In order to show that the
  subcategory $\C$ is closed under the action of $T$, we observe that, for
  $C\in\C$ and $W\in\W$, there are isomorphisms of graded
  $\kk$-modules~\cite[Lemma~4.3]{HJ22}
  \begin{align*}
    \Ext<Q>[\bullet]{U(A\otimes C)}{W}&\cong\Ext<A\otimes Q>[\bullet]{A\otimes C}{\Hom<\kk>{A}{W}}\\
                                      &\cong\Ext<Q>[\bullet]{C}{U(\Hom<\kk>{A}{W})},
  \end{align*}
  The required vanishing $\Ext<Q>[1]{U(A\otimes C)}{W}=0$ then follows from the
  fact that the exact functor $\Hom<\kk>{A}{-}$ preserves injective objects, and
  then also objects of finite injective dimension.

  We conclude by showing that the model structure on $\LMod{Q\otimes A}$
  obtained from \Cref{thm:Lurie-transfer-model_structure-monad} agrees with the
  abelian model structure constructed in~\cite[Theorem~6.1(a)]{HJ22},
  in which every object is fibrant. According to~\cite[Proposition~E.1.10]{Joy},
  it is enough to show that both model structures have the same trivial
  fibrations and the same class of fibrant objects. To prove this, observe that
  the exact forgetful functor $U\colon\Mod{A\otimes Q}\to\LMod{Q}$ detects
  epimorphisms. From this, it follows that the fibrations of the model structure
  are the epimorphisms, so that every object is fibrant, and that the trivial
  fibrations are the epimorphisms whose kernel has finite projective/injective
  dimension as a $Q$-module, which are precisely the trivial fibrations of the
  abelian model structure constructed in~\cite[Theorem~6.1(a)]{HJ22}.
  This finishes the proof.
\end{proof}

\Cref{thm:Gillespie} and \Cref{prop:HJ-Lurie} permit us to make the following
definition (compare with~\cite[Remark~3.23]{Jas25}).

\begin{definition}
  \label{def:Q-shaped-A}
  Let $Q$ be a small $\kk$-category satisfying the assumptions
  in~\cite[Setup~2.9]{HJ24} and $A$ a $\kk$-algebra whose underlying
  $\kk$-module is flat. The \emph{$Q$-shaped derived $\infty$-category of
    $A$} is the underlying $\infty$-category
  \begin{align*}
    \DerCat<Q>{A}&\coloneqq\MLoc{\LMod{A\otimes Q}}\\
                 &=\MLoc{\Alg[U(A\otimes-)]{\LMod{A\otimes Q}}}\\
                 &\simeq\Alg[U(A\otimes-)]{\DerCat<Q>{\kk}}.
  \end{align*}
  of the model structure described in \Cref{prop:HJ-Lurie}. This is a
  presentable stable $\infty$-category whose homotopy category is the $Q$-shaped
  derived category of $A$ defined in~\cite{HJ22}. Moreover, $\DerCat<Q>{A}$ is
  compactly generated by~\cite[Theorem~D]{HJ24}. Finally, there is an equivalence of
  $\DerCat{\Mod{\kk}}$-linear $\infty$-categories~\cite[Remark~3.23]{Jas25}
  \[
    \DerCat<Q>{A}\simeq\DerCat{\LMod{A}}\otimes_{\DerCat{\Mod{\kk}}}\DerCat<Q>{\kk},
  \]
  which should be compared with the equivalence in
  \Cref{thm:splitting_formula}.
\end{definition}

\begin{remark}
  \label{rmk:H-vs-Q}
  \Cref{def:Hopfological_derived_category,def:Q-shaped-A} show that both
  $Q$-shaped and Hopfological derived $\infty$-categories can be realised as
  $\infty$-categories of algebras over a colimit-preserving monad acting on a
  compactly generated stable $\infty$-category, see~\Cref{rmk:monad-algebra}. In
  the case of Hopfological derived $\infty$-categories, the monads in question
  are induced by algebra objects of the monoidal category $\LMod{H}$ while, in
  contrast, $\LMod{Q}$ in general is not endowed with a preferred monoidal
  structure, hence working with monads is essential in this comparison. In both
  cases, the corresponding derived $\infty$-categories are defined as the
  underlying $\infty$-categories of model category structures on a suitable
  category of algebras over a monad, obtained by right-transfer from a
  hereditary abelian model structure on an appropriate base category.
\end{remark}

\printbibliography

@Article{Lur15,
  author = {Lurie, Jacob},
  title  = {Rotation invariance in algebraic $K$-theory},
  year   = {2015},
}

@Article{Jas24a,
  author   = {Jasso, Gustavo},
  journal  = {C. R. Math. Acad. Sci. Paris},
  title    = {Derived equivalences of upper-triangular ring spectra via lax limits},
  year     = {2024},
  issn     = {1631-073X,1778-3569},
  pages    = {279--285},
  volume   = {362},
  doi      = {10.5802/crmath.559},
  fjournal = {Comptes Rendus Math\'{e}matique. Acad\'{e}mie des Sciences. Paris},
  mrclass  = {18G80},
  mrnumber = {4745338},
  url      = {https://doi.org/10.5802/crmath.559},
}

@Article{Ros05,
  author     = {Rosick\'y, Ji\v r\'\i},
  journal    = {Theory Appl. Categ.},
  title      = {Generalized {B}rown representability in homotopy categories},
  year       = {2005},
  issn       = {1201-561X},
  pages      = {no. 19, 451--479},
  volume     = {14},
  fjournal   = {Theory and Applications of Categories},
  mrclass    = {18G55 (55P99 55U35)},
  mrnumber   = {2211427},
  mrreviewer = {Peter Bubenik},
}

@Article{NS17a,
  author     = {Nikolaus, Thomas and Sagave, Steffen},
  journal    = {Algebr. Geom. Topol.},
  title      = {Presentably symmetric monoidal {$\infty$}-categories are represented by symmetric monoidal model categories},
  year       = {2017},
  issn       = {1472-2747,1472-2739},
  number     = {5},
  pages      = {3189--3212},
  volume     = {17},
  doi        = {10.2140/agt.2017.17.3189},
  fjournal   = {Algebraic \& Geometric Topology},
  mrclass    = {55U35 (18D10 18G55)},
  mrnumber   = {3704256},
  mrreviewer = {Josu\'e\ Remedios},
  url        = {https://doi.org/10.2140/agt.2017.17.3189},
}

@PhdThesis{Lei22,
  author = {Leip, Malte},
  school = {University of Copenhagen},
  title  = {On the {H}ochschild {H}omology of {H}ypersurfaces as a {M}ixed {C}omplex},
  year   = {2022},
  month  = apr,
  type   = {Ph.D.~thesis},
}

@Article{Bea23,
  author     = {Beardsley, Jonathan},
  journal    = {Homology Homotopy Appl.},
  title      = {On bialgebras, comodules, descent data and {T}hom spectra in {$\infty$}-categories},
  year       = {2023},
  issn       = {1532-0073,1532-0081},
  number     = {2},
  pages      = {219--242},
  volume     = {25},
  doi        = {10.4310/hha.2023.v25.n2.a10},
  fjournal   = {Homology, Homotopy and Applications},
  mrclass    = {55N22 (16T10 18F20 18N70 55P43)},
  mrnumber   = {4689082},
  mrreviewer = {Yifei\ Zhu},
  url        = {https://doi.org/10.4310/hha.2023.v25.n2.a10},
}

@InCollection{Ste18,
  author     = {Stevenson, Greg},
  booktitle  = {Building bridges between algebra and topology},
  publisher  = {Birkh\"{a}user/Springer, Cham},
  title      = {A tour of support theory for triangulated categories through tensor triangular geometry},
  year       = {2018},
  pages      = {63--101},
  series     = {Adv. Courses Math. CRM Barcelona},
  mrclass    = {18E30 (55P60)},
  mrnumber   = {3793858},
  mrreviewer = {Beren Sanders},
}

@Book{Tab15,
  author     = {Tabuada, Gon{\c{c}}alo},
  publisher  = {American Mathematical Society, Providence, RI},
  title      = {Noncommutative motives},
  year       = {2015},
  isbn       = {978-1-4704-2397-1},
  note       = {With a preface by Yuri I. Manin},
  series     = {University Lecture Series},
  volume     = {63},
  doi        = {10.1090/ulect/063},
  mrclass    = {14A22 (14C15 18D10)},
  mrnumber   = {3379910},
  mrreviewer = {Pieter Belmans},
  pages      = {x+114},
}

@PhdThesis{Erg22,
  author = {Ergus, Aras},
  school = {École Polytechnique Fédérale de Lausanne},
  title  = {Hopf algebras and {H}opf--{G}alois extensions in $\infty$-categories},
  year   = {2022},
}

@Article{CT12,
  author     = {Cisinski, Denis-Charles and Tabuada, Gon{\c{c}}alo},
  journal    = {J. K-Theory},
  title      = {Symmetric monoidal structure on non-commutative motives},
  year       = {2012},
  issn       = {1865-2433},
  number     = {2},
  pages      = {201--268},
  volume     = {9},
  doi        = {10.1017/is011011005jkt169},
  fjournal   = {Journal of K-Theory. K-Theory and its Applications in Algebra, Geometry, Analysis \& Topology},
  mrclass    = {19D35 (14A22 14C35 16E40 18D10 19D55)},
  mrnumber   = {2922389},
  mrreviewer = {Satoshi Mochizuki},
}

@Article{BM24,
  author     = {Blumberg, Andrew J. and Mandell, Michael A.},
  journal    = {Mem. Amer. Math. Soc.},
  title      = {The strong {K}\"unneth theorem for topological periodic cyclic homology},
  year       = {2024},
  issn       = {0065-9266,1947-6221},
  number     = {1508},
  pages      = {v+102},
  volume     = {301},
  doi        = {10.1090/memo/1508},
  fjournal   = {Memoirs of the American Mathematical Society},
  isbn       = {978-1-4704-7138-5; 978-1-4704-7952-7},
  mrclass    = {19D55},
  mrnumber   = {4808710},
  mrreviewer = {Gabriel\ Angelini-Knoll},
  url        = {https://doi.org/10.1090/memo/1508},
}

@Article{AMN18,
  author   = {Antieau, Benjamin and Mathew, Akhil and Nikolaus, Thomas},
  journal  = {Selecta Math. (N.S.)},
  title    = {On the {B}lumberg--{M}andell {K}\"{u}nneth theorem for TP},
  year     = {2018},
  issn     = {1022-1824},
  number   = {5},
  pages    = {4555--4576},
  volume   = {24},
  doi      = {10.1007/s00029-018-0427-x},
  fjournal = {Selecta Mathematica. New Series},
  mrclass  = {14F30 (16E40 19D55)},
  mrnumber = {3874698},
  url      = {https://doi.org/10.1007/s00029-018-0427-x},
}

@Book{AR94,
  author     = {Ad\'amek, Ji{\v r}\'i and Rosick\'y, Ji{\v r}\'i},
  publisher  = {Cambridge University Press, Cambridge},
  title      = {Locally presentable and accessible categories},
  year       = {1994},
  isbn       = {0-521-42261-2},
  series     = {London Mathematical Society Lecture Note Series},
  volume     = {189},
  doi        = {10.1017/CBO9780511600579},
  mrclass    = {18Axx (18-02)},
  mrnumber   = {1294136},
  mrreviewer = {J.\ R.\ Isbell},
  pages      = {xiv+316},
  url        = {https://doi.org/10.1017/CBO9780511600579},
}

@Misc{Kerodon,
  author = {Lurie, Jacob},
  note   = {\url{https://kerodon.net}},
  title  = {Kerodon},
  url    = {https://kerodon.net},
}

@Book{Lan21,
  author    = {Land, Markus},
  publisher = {Birkh\"{a}user/Springer, Cham},
  title     = {Introduction to infinity-categories},
  year      = {2021},
  isbn      = {978-3-030-61523-9; 978-3-030-61524-6},
  series    = {Compact Textbooks in Mathematics},
  doi       = {10.1007/978-3-030-61524-6},
  mrclass   = {18-02 (18N50 18N55 18N60)},
  mrnumber  = {4259746},
  pages     = {ix+296},
  url       = {https://doi.org/10.1007/978-3-030-61524-6},
}

@InCollection{Gro20,
  author    = {Groth, Moritz},
  booktitle = {Handbook of homotopy theory},
  publisher = {CRC Press, Boca Raton, FL},
  title     = {A short course on {$\infty$}-categories},
  year      = {2020},
  isbn      = {978-0-815-36970-7},
  pages     = {549--617},
  series    = {CRC Press/Chapman Hall Handb. Math. Ser.},
  mrclass   = {18Nxx (55U35 55U40)},
  mrnumber  = {4197994},
}

@InCollection{ACam16,
  author     = {Antol\'in Camarena, Omar},
  booktitle  = {Mexican mathematicians abroad: recent contributions},
  publisher  = {Amer. Math. Soc., Providence, RI},
  title      = {A whirlwind tour of the world of {$(\infty,1)$}-categories},
  year       = {2016},
  isbn       = {978-1-4704-2192-2},
  pages      = {15--61},
  series     = {Contemp. Math.},
  volume     = {657},
  doi        = {10.1090/conm/657/13088},
  mrclass    = {18-01 (18-02 18D99 18G99)},
  mrnumber   = {3466443},
  mrreviewer = {Ram\'on\ Gonz\'alez Rodr\'iguez},
  url        = {https://doi.org/10.1090/conm/657/13088},
}

@InCollection{QS17,
  author     = {Qi, You and Sussan, Joshua},
  booktitle  = {Categorification and higher representation theory},
  publisher  = {Amer. Math. Soc., Providence, RI},
  title      = {Categorification at prime roots of unity and hopfological finiteness},
  year       = {2017},
  pages      = {261--286},
  series     = {Contemp. Math.},
  volume     = {683},
  doi        = {10.1090/conm/683},
  mrclass    = {16E45 (16E20 16T20 57M27)},
  mrnumber   = {3611717},
  mrreviewer = {Volodymyr Mazorchuk},
  url        = {https://doi.org/10.1090/conm/683},
}

@Article{SS00,
  author     = {Schwede, Stefan and Shipley, Brooke E.},
  journal    = {Proc. London Math. Soc. (3)},
  title      = {Algebras and modules in monoidal model categories},
  year       = {2000},
  issn       = {0024-6115,1460-244X},
  number     = {2},
  pages      = {491--511},
  volume     = {80},
  doi        = {10.1112/S002461150001220X},
  fjournal   = {Proceedings of the London Mathematical Society. Third Series},
  mrclass    = {18D10 (18D50 55P48 55U35)},
  mrnumber   = {1734325},
  mrreviewer = {Mark\ Hovey},
  url        = {https://doi.org/10.1112/S002461150001220X},
}

@Article{Kel94,
  author     = {Keller, Bernhard},
  journal    = {Ann. Sci. \'Ecole Norm. Sup. (4)},
  title      = {Deriving {DG} categories},
  year       = {1994},
  issn       = {0012-9593},
  number     = {1},
  pages      = {63--102},
  volume     = {27},
  coden      = {ASENAH},
  fjournal   = {Annales Scientifiques de l'\'Ecole Normale Sup\'erieure. Quatri\`eme S\'erie},
  mrclass    = {18E30 (16D90)},
  mrnumber   = {1258406},
  mrreviewer = {Jeremy Rickard},
  url        = {http://www.numdam.org/item?id=ASENS_1994_4_27_1_63_0},
}

@Article{Far21,
  author     = {Farinati, Marco A.},
  journal    = {Algebr. Represent. Theory},
  title      = {Hopfological algebra for infinite dimensional {H}opf algebras},
  year       = {2021},
  issn       = {1386-923X},
  number     = {5},
  pages      = {1325--1357},
  volume     = {24},
  doi        = {10.1007/s10468-020-09993-7},
  fjournal   = {Algebras and Representation Theory},
  mrclass    = {16T05 (16E35 18G80 19A49 81R50)},
  mrnumber   = {4313062},
  mrreviewer = {Ji-Wei He},
  url        = {https://doi.org/10.1007/s10468-020-09993-7},
}

@Article{Kho16,
  author     = {Khovanov, Mikhail},
  journal    = {J. Knot Theory Ramifications},
  title      = {Hopfological algebra and categorification at a root of unity: the first steps},
  year       = {2016},
  issn       = {0218-2165},
  number     = {3},
  pages      = {1640006, 26},
  volume     = {25},
  doi        = {10.1142/S021821651640006X},
  fjournal   = {Journal of Knot Theory and its Ramifications},
  mrclass    = {18E30 (16T05 18G60)},
  mrnumber   = {3475073},
  mrreviewer = {Andrei Marcus},
  url        = {http://dx.doi.org/10.1142/S021821651640006X},
}

@InCollection{Del90,
  author     = {Deligne, P.},
  booktitle  = {The {G}rothendieck {F}estschrift, {V}ol.\ {II}},
  publisher  = {Birkh\"auser Boston, Boston, MA},
  title      = {Cat\'egories tannakiennes},
  year       = {1990},
  isbn       = {0-8176-3428-2},
  pages      = {111--195},
  series     = {Progr. Math.},
  volume     = {87},
  mrclass    = {14A99 (12H05 18A99)},
  mrnumber   = {1106898},
  mrreviewer = {James\ Milne},
}

@Article{Spa88,
  author     = {Spaltenstein, N.},
  journal    = {Compositio Math.},
  title      = {Resolutions of unbounded complexes},
  year       = {1988},
  issn       = {0010-437X},
  number     = {2},
  pages      = {121--154},
  volume     = {65},
  coden      = {CMPMAF},
  fjournal   = {Compositio Mathematica},
  mrclass    = {18E25 (32C35)},
  mrnumber   = {932640},
  mrreviewer = {Michael M. Kapranov},
  url        = {http://www.numdam.org/item?id=CM_1988__65_2_121_0},
}

@InCollection{Day70,
  author     = {Day, Brian},
  booktitle  = {Reports of the {M}idwest {C}ategory {S}eminar, {IV}},
  publisher  = {Springer, Berlin},
  title      = {On closed categories of functors},
  year       = {1970},
  pages      = {1--38},
  series     = {Lecture Notes in Mathematics, Vol. 137},
  mrclass    = {18.10},
  mrnumber   = {0272852},
  mrreviewer = {J. F. Kennison},
}

@Book{Wei94,
  author     = {Weibel, Charles A.},
  publisher  = {Cambridge University Press, Cambridge},
  title      = {An introduction to homological algebra},
  year       = {1994},
  isbn       = {0-521-43500-5; 0-521-55987-1},
  series     = {Cambridge Studies in Advanced Mathematics},
  volume     = {38},
  doi        = {10.1017/CBO9781139644136},
  mrclass    = {18-01 (16-01 17-01 20-01 55Uxx)},
  mrnumber   = {1269324},
  mrreviewer = {Kenneth A. Brown},
  pages      = {xiv+450},
}

@Book{Nee01,
  author     = {Neeman, Amnon},
  publisher  = {Princeton University Press, Princeton, NJ},
  title      = {Triangulated categories},
  year       = {2001},
  isbn       = {0-691-08685-0; 0-691-08686-9},
  series     = {Annals of Mathematics Studies},
  volume     = {148},
  doi        = {10.1515/9781400837212},
  mrclass    = {18E30 (55-02 55N20 55U35)},
  mrnumber   = {1812507},
  mrreviewer = {Stanis{\l}aw Betley},
  pages      = {viii+449},
}

@Article{Pos11,
  author     = {Positselski, Leonid},
  journal    = {Mem. Amer. Math. Soc.},
  title      = {Two kinds of derived categories, {K}oszul duality, and comodule-contramodule correspondence},
  year       = {2011},
  issn       = {0065-9266},
  number     = {996},
  pages      = {vi+133},
  volume     = {212},
  doi        = {10.1090/S0065-9266-2010-00631-8},
  fjournal   = {Memoirs of the American Mathematical Society},
  isbn       = {978-0-8218-5296-5},
  mrclass    = {16E35 (16E45 16S37 16T15 18E30)},
  mrnumber   = {2830562},
  mrreviewer = {Ji-Wei He},
  url        = {http://dx.doi.org/10.1090/S0065-9266-2010-00631-8},
}

@Article{Jas25,
  author        = {Jasso, Gustavo},
  title         = {$Q$-shaped derived categories as derived categories of differential graded bimodules},
  year          = {2025},
  month         = jan,
  note          = {accepted for publication in Ann. Represent. Theory},
  archiveprefix = {arXiv},
  copyright     = {Creative Commons Attribution 4.0 International},
  doi           = {10.48550/ARXIV.2501.08255},
  eprint        = {2501.08255},
  primaryclass  = {math.RT},
  publisher     = {arXiv},
}

@Article{BGT14,
  author     = {Blumberg, Andrew J. and Gepner, David and Tabuada, Gon\c calo},
  journal    = {Adv. Math.},
  title      = {Uniqueness of the multiplicative cyclotomic trace},
  year       = {2014},
  issn       = {0001-8708,1090-2082},
  pages      = {191--232},
  volume     = {260},
  doi        = {10.1016/j.aim.2014.02.004},
  fjournal   = {Advances in Mathematics},
  mrclass    = {19D55 (19D23 55N15)},
  mrnumber   = {3209352},
  mrreviewer = {Jeffrey\ Giansiracusa},
  url        = {https://doi.org/10.1016/j.aim.2014.02.004},
}

@Article{QS22,
  author     = {Qi, You and Sussan, Joshua},
  journal    = {Forum Math. Pi},
  title      = {On some {$p$}-differential graded link homologies},
  year       = {2022},
  issn       = {2050-5086},
  pages      = {Paper No. e26, 58},
  volume     = {10},
  doi        = {10.1017/fmp.2022.19},
  fjournal   = {Forum of Mathematics. Pi},
  mrclass    = {57K18 (18N25)},
  mrnumber   = {4524363},
  mrreviewer = {Carlo\ Collari},
  url        = {https://doi.org/10.1017/fmp.2022.19},
}

@Article{QS23,
  author     = {Qi, You and Sussan, Joshua},
  journal    = {Algebr. Geom. Topol.},
  title      = {On some {$p$}-differential graded link homologies, {II}},
  year       = {2023},
  issn       = {1472-2747,1472-2739},
  number     = {7},
  pages      = {3357--3394},
  volume     = {23},
  doi        = {10.2140/agt.2023.23.3357},
  fjournal   = {Algebraic \& Geometric Topology},
  mrclass    = {57K18 (18G99)},
  mrnumber   = {4647678},
  mrreviewer = {Carlo\ Collari},
  url        = {https://doi.org/10.2140/agt.2023.23.3357},
}

@Article{BR23,
  author        = {Brav, Christopher and Rozenblyum, Nick},
  title         = {The cyclic Deligne conjecture and Calabi-Yau structures},
  year          = {2023},
  month         = may,
  archiveprefix = {arXiv},
  copyright     = {arXiv.org perpetual, non-exclusive license},
  doi           = {10.48550/ARXIV.2305.10323},
  eprint        = {2305.10323},
  primaryclass  = {math.AT},
  publisher     = {arXiv},
}

@Article{HSS17,
  author     = {Hoyois, Marc and Scherotzke, Sarah and Sibilla, Nicol\`o},
  journal    = {Adv. Math.},
  title      = {Higher traces, noncommutative motives, and the categorified {C}hern character},
  year       = {2017},
  issn       = {0001-8708,1090-2082},
  pages      = {97--154},
  volume     = {309},
  doi        = {10.1016/j.aim.2017.01.008},
  fjournal   = {Advances in Mathematics},
  mrclass    = {14F05 (14F42 18D05 19D55)},
  mrnumber   = {3607274},
  mrreviewer = {Satoshi\ Mochizuki},
  url        = {https://doi.org/10.1016/j.aim.2017.01.008},
}

@Article{BGT13,
  author     = {Blumberg, Andrew J. and Gepner, David and Tabuada, Gon{\c{c}}alo},
  journal    = {Geom. Topol.},
  title      = {A universal characterization of higher algebraic {$K$}-theory},
  year       = {2013},
  issn       = {1465-3060},
  number     = {2},
  pages      = {733--838},
  volume     = {17},
  doi        = {10.2140/gt.2013.17.733},
  fjournal   = {Geometry \& Topology},
  mrclass    = {19D10 (18D20 19D25 19D55 55N15 55U40)},
  mrnumber   = {3070515},
  mrreviewer = {Ross Staffeldt},
}

@Article{Jas26,
  author        = {Jasso, Gustavo},
  title         = {Stable $\infty$-categories for representation theorists},
  year          = {2026},
  month         = mar,
  archiveprefix = {arXiv},
  copyright     = {Creative Commons Attribution 4.0 International},
  doi           = {10.48550/ARXIV.2603.14970},
  eprint        = {2603.14970},
  primaryclass  = {math.RT},
  publisher     = {arXiv},
}

@InCollection{Kel06,
  author     = {Keller, Bernhard},
  booktitle  = {International {C}ongress of {M}athematicians. {V}ol. {II}},
  publisher  = {Eur. Math. Soc., Z\"urich},
  title      = {On differential graded categories},
  year       = {2006},
  pages      = {151--190},
  mrclass    = {18E30 (14A22 16D90)},
  mrnumber   = {2275593},
  mrreviewer = {Volodymyr V. Lyubashenko},
}

@Book{Joy08,
  author = {Joyal, Andr{\'e}},
  title  = {Notes on quasi-categories},
  year   = {2008},
  month  = jun,
  note   = {In preparation},
}

@Article{Joy02,
  author     = {Joyal, A.},
  journal    = {J. Pure Appl. Algebra},
  title      = {Quasi-categories and {K}an complexes},
  year       = {2002},
  issn       = {0022-4049},
  note       = {Special volume celebrating the 70th birthday of Professor Max Kelly},
  number     = {1-3},
  pages      = {207--222},
  volume     = {175},
  doi        = {10.1016/S0022-4049(02)00135-4},
  fjournal   = {Journal of Pure and Applied Algebra},
  mrclass    = {55U10 (18G55)},
  mrnumber   = {1935979},
  mrreviewer = {Donald M. Davis},
  url        = {http://dx.doi.org/10.1016/S0022-4049(02)00135-4},
}

@Book{Lur09,
  author     = {Lurie, Jacob},
  publisher  = {Princeton University Press, Princeton, NJ},
  title      = {Higher topos theory},
  year       = {2009},
  isbn       = {978-0-691-14049-0; 0-691-14049-9},
  series     = {Annals of Mathematics Studies},
  volume     = {170},
  doi        = {10.1515/9781400830558},
  mrclass    = {18-02 (18B25 18E35 18G30 18G55 55U40)},
  mrnumber   = {2522659},
  mrreviewer = {Mark Hovey},
  pages      = {xviii+925},
}

@Book{Lur17,
  author = {Lurie, Jacob},
  title  = {Higher Algebra},
  year   = {2017},
  month  = may,
  note   = {Available online at the author's webpage: \url{https://www.math.ias.edu/~lurie/papers/HA.pdf}},
}

@Book{Lur18SAG,
  author = {Lurie, Jacob},
  title  = {Spectral {A}lgebraic {G}eometry},
  year   = {2018},
  month  = june,
  note   = {Available online at the author's website: \url{https://www.math.ias.edu/~lurie/papers/SAG-rootfile.pdf}},
}

@Article{Qi14,
  author     = {Qi, You},
  journal    = {Compos. Math.},
  title      = {Hopfological algebra},
  year       = {2014},
  issn       = {0010-437X},
  number     = {1},
  pages      = {1--45},
  volume     = {150},
  doi        = {10.1112/S0010437X13007380},
  fjournal   = {Compositio Mathematica},
  mrclass    = {13D09 (13D15 16T05)},
  mrnumber   = {3164358},
  mrreviewer = {Kenneth A. Brown},
  url        = {http://dx.doi.org/10.1112/S0010437X13007380},
}

@Book{EGNO15,
  author     = {Etingof, Pavel and Gelaki, Shlomo and Nikshych, Dmitri and Ostrik, Victor},
  publisher  = {American Mathematical Society, Providence, RI},
  title      = {Tensor categories},
  year       = {2015},
  isbn       = {978-1-4704-2024-6},
  series     = {Mathematical Surveys and Monographs},
  volume     = {205},
  doi        = {10.1090/surv/205},
  mrclass    = {18D10 (16T05)},
  mrnumber   = {3242743},
  mrreviewer = {Julien Bichon},
  pages      = {xvi+343},
  url        = {http://dx.doi.org/10.1090/surv/205},
}

@Article{Hov02,
  author   = {Hovey, Mark},
  journal  = {Math. Z.},
  title    = {Cotorsion pairs, model category structures, and representation theory},
  year     = {2002},
  issn     = {0025-5874},
  number   = {3},
  pages    = {553--592},
  volume   = {241},
  doi      = {10.1007/s00209-002-0431-9},
  fjournal = {Mathematische Zeitschrift},
  mrclass  = {55U35 (18E30 18G55)},
  mrnumber = {1938704},
  url      = {https://doi.org/10.1007/s00209-002-0431-9},
}

@Book{Hov99,
  author     = {Hovey, Mark},
  publisher  = {American Mathematical Society, Providence, RI},
  title      = {Model categories},
  year       = {1999},
  isbn       = {0-8218-1359-5},
  series     = {Mathematical Surveys and Monographs},
  volume     = {63},
  mrclass    = {55U35 (18D15 18G30 18G55)},
  mrnumber   = {1650134},
  mrreviewer = {Teimuraz Pirashvili},
  pages      = {xii+209},
}

@InCollection{Sto13,
  author     = {{{\v{S}}\v{t}ov\'{i}\v{c}ek}, Jan},
  booktitle  = {Advances in representation theory of algebras},
  publisher  = {Eur. Math. Soc., Z\"urich},
  title      = {Exact model categories, approximation theory, and cohomology of quasi-coherent sheaves},
  year       = {2013},
  isbn       = {978-3-03719-125-5},
  pages      = {297--367},
  series     = {EMS Ser. Congr. Rep.},
  mrclass    = {18E10 (18E30 18F20)},
  mrnumber   = {3220541},
  mrreviewer = {R.\ H.\ Street},
  shorthand = {{\v{S}}to13},
}

@book {Mon93,
    AUTHOR = {Montgomery, Susan},
     TITLE = {Hopf algebras and their actions on rings},
    SERIES = {CBMS Regional Conference Series in Mathematics},
    VOLUME = {82},
 PUBLISHER = {The American Mathematical Society, Providence, RI},
      YEAR = {1993},
     PAGES = {xiv+238},
      ISBN = {0-8218-0738-2},
   MRCLASS = {16W30},
  MRNUMBER = {1243637},
MRREVIEWER = {E.\ J.\ Taft},
       DOI = {10.1090/cbms/082},
       URL = {https://doi.org/10.1090/cbms/082},
}

@Article{Bec14,
  author     = {Becker, Hanno},
  journal    = {Adv. Math.},
  title      = {Models for singularity categories},
  year       = {2014},
  issn       = {0001-8708,1090-2082},
  pages      = {187--232},
  volume     = {254},
  doi        = {10.1016/j.aim.2013.11.016},
  fjournal   = {Advances in Mathematics},
  mrclass    = {55U35 (16E35 18E30)},
  mrnumber   = {3161097},
  mrreviewer = {Timothy\ Porter},
  url        = {https://doi.org/10.1016/j.aim.2013.11.016},
}

@Article{Gil11,
  author     = {Gillespie, James},
  journal    = {J. Pure Appl. Algebra},
  title      = {Model structures on exact categories},
  year       = {2011},
  issn       = {0022-4049},
  number     = {12},
  pages      = {2892--2902},
  volume     = {215},
  doi        = {10.1016/j.jpaa.2011.04.010},
  fjournal   = {Journal of Pure and Applied Algebra},
  mrclass    = {18E10 (18G35 55U15 55U35)},
  mrnumber   = {2811572},
  mrreviewer = {Timothy Porter},
  url        = {https://doi.org/10.1016/j.jpaa.2011.04.010},
}

@Book{Hap88,
  author     = {Happel, Dieter},
  publisher  = {Cambridge University Press, Cambridge},
  title      = {Triangulated categories in the representation theory of finite-dimensional algebras},
  year       = {1988},
  isbn       = {0-521-33922-7},
  series     = {London Mathematical Society Lecture Note Series},
  volume     = {119},
  doi        = {10.1017/CBO9780511629228},
  mrclass    = {16A46 (16A48 16A62 16A64 18E30)},
  mrnumber   = {935124},
  mrreviewer = {Alfred G. Wiedemann},
  pages      = {x+208},
}

@Book{Cis19,
  author    = {Cisinski, Denis-Charles},
  publisher = {Cambridge University Press, Cambridge},
  title     = {Higher categories and homotopical algebra},
  year      = {2019},
  isbn      = {978-1-108-47320-0},
  series    = {Cambridge Studies in Advanced Mathematics},
  volume    = {180},
  doi       = {10.1017/9781108588737},
  mrclass   = {18D05 (18F20 18G55 55U35)},
  mrnumber  = {3931682},
  pages     = {xviii+430},
  url       = {https://doi.org/10.1017/9781108588737},
}

@Book{Kas95,
  author     = {Kassel, Christian},
  publisher  = {Springer-Verlag, New York},
  title      = {Quantum groups},
  year       = {1995},
  isbn       = {0-387-94370-6},
  series     = {Graduate Texts in Mathematics},
  volume     = {155},
  doi        = {10.1007/978-1-4612-0783-2},
  mrclass    = {17B37 (16W30 18D10 20F36 57M25 81R50)},
  mrnumber   = {1321145},
  mrreviewer = {Yu.\ N.\ Bespalov},
  pages      = {xii+531},
  url        = {https://doi.org/10.1007/978-1-4612-0783-2},
}

@Online{OT20a,
  archiveprefix = {arXiv},
  author        = {Mariko Ohara and Dai Tamaki},
  eprint        = {2012.07159},
  title         = {Cotorsion pairs in Hopfological algebra},
  year          = {2020},
}

@Article{Oha25a,
  author        = {Ohara, Mariko},
  title         = {Hopfological invariants for tame subextensions},
  year          = {2025},
  month         = feb,
  archiveprefix = {arXiv},
  copyright     = {Creative Commons Attribution 4.0 International},
  doi           = {10.48550/ARXIV.2502.15191},
  eprint        = {2502.15191},
  primaryclass  = {math.KT},
  publisher     = {arXiv},
}

@Article{Oha25,
  author   = {Ohara, Mariko},
  journal  = {J. Algebra},
  title    = {A model structure and {H}opf-cyclic theory on the category of coequivariant modules over a comodule algebra},
  year     = {2025},
  issn     = {0021-8693,1090-266X},
  pages    = {365--389},
  volume   = {668},
  doi      = {10.1016/j.jalgebra.2025.01.011},
  fjournal = {Journal of Algebra},
  mrclass  = {16T15 (16E30 18G20 55N25 55U10 57T05)},
  mrnumber = {4857709},
  url      = {https://doi.org/10.1016/j.jalgebra.2025.01.011},
}

@Article{Kra05,
  author     = {Krause, Henning},
  journal    = {Compos. Math.},
  title      = {The stable derived category of a {N}oetherian scheme},
  year       = {2005},
  issn       = {0010-437X},
  number     = {5},
  pages      = {1128--1162},
  volume     = {141},
  doi        = {10.1112/S0010437X05001375},
  fjournal   = {Compositio Mathematica},
  mrclass    = {18E30 (14F05 16E30 16E65 55U35)},
  mrnumber   = {2157133},
  mrreviewer = {J. P. C. Greenlees},
  url        = {https://doi.org/10.1112/S0010437X05001375},
}

@Article{Oha24,
  author        = {Ohara, Mariko},
  title         = {A model structure on the category of equivariant A-modules over a Hopf algebra},
  year          = {2024},
  month         = jan,
  archiveprefix = {arXiv},
  copyright     = {Creative Commons Attribution 4.0 International},
  doi           = {10.48550/ARXIV.2401.05687},
  eprint        = {2401.05687},
  primaryclass  = {math.KT},
  publisher     = {arXiv},
}

@Article{HJ22,
  author   = {Holm, Henrik and J{\o}rgensen, Peter},
  journal  = {J. Lond. Math. Soc. (2)},
  title    = {The {$Q$}-shaped derived category of a ring},
  year     = {2022},
  issn     = {0024-6107},
  number   = {4},
  pages    = {3263--3316},
  volume   = {106},
  doi      = {10.1112/jlms.12662},
  fjournal = {Journal of the London Mathematical Society. Second Series},
  mrclass  = {18G80 (16E35 18E35 18N40)},
  mrnumber = {4524199},
}

@InCollection{Sal79,
  author     = {Salce, Luigi},
  booktitle  = {Symposia {M}athematica, {V}ol. {XXIII} ({C}onf. {A}belian {G}roups and their {R}elationship to the {T}heory of {M}odules, {INDAM}, {R}ome, 1977)},
  publisher  = {Academic Press, London-New York},
  title      = {Cotorsion theories for abelian groups},
  year       = {1979},
  pages      = {11--32},
  mrclass    = {20K40 (18E40)},
  mrnumber   = {565595},
  mrreviewer = {P. L. Sperry},
}

@Book{Joy,
  author = {Joyal, Andr{\'e}},
  title  = {The {T}heory of {Q}uasi-{C}ategories and its {A}pplications},
  note   = {In preparation},
}

@InProceedings{HJ24a,
  author    = {Holm, Henrik and J{\o}rgensen, Peter},
  booktitle = {Triangulated Categories in Representation Theory and Beyond},
  title     = {A Brief Introduction to the Q-Shaped Derived Category},
  year      = {2024},
  address   = {Cham},
  editor    = {Bergh, Petter Andreas and Oppermann, Steffen and Solberg, {\O}yvind},
  pages     = {141--167},
  publisher = {Springer Nature Switzerland},
  isbn      = {978-3-031-57789-5},
}

@Article{HJ24,
  author   = {Holm, Henrik and J{\o}rgensen, Peter},
  journal  = {Trans. Amer. Math. Soc.},
  title    = {The {$Q$}-shaped derived category of a ring---compact and perfect objects},
  year     = {2024},
  issn     = {0002-9947,1088-6850},
  number   = {5},
  pages    = {3095--3128},
  volume   = {377},
  doi      = {10.1090/tran/8979},
  fjournal = {Transactions of the American Mathematical Society},
  mrclass  = {18G80 (16E35 18N40)},
  mrnumber = {4744776},
  url      = {https://doi.org/10.1090/tran/8979},
}

@Article{BW05,
  author   = {Barr, Michael and Wells, Charles},
  journal  = {Repr. Theory Appl. Categ.},
  title    = {Toposes, triples and theories},
  year     = {2005},
  note     = {Corrected reprint of the 1985 original [MR0771116]},
  number   = {12},
  pages    = {x+288},
  fjournal = {Reprints in Theory and Applications of Categories},
  mrclass  = {18-02 (03G30 18B25 18C10 18C15)},
  mrnumber = {2178101},
}

@Article{Ram23,
  author        = {Ramzi, Maxime},
  title         = {Separability in homotopical algebra},
  year          = {2023},
  month         = oct,
  archiveprefix = {arXiv},
  doi           = {10.48550/arXiv.2305.17236},
  eprint        = {2305.17236},
  publisher     = {arXiv},
}

@Article{Mat16,
  title={The Galois group of a stable homotopy theory},
  author={Mathew, Akhil},
  journal={Advances in Mathematics},
  volume={291},
  pages={403--541},
  year={2016},
  publisher={Elsevier}
}

@InCollection{Kra10,
  author     = {Krause, Henning},
  booktitle  = {Triangulated categories},
  publisher  = {Cambridge Univ. Press, Cambridge},
  title      = {Localization theory for triangulated categories},
  year       = {2010},
  pages      = {161--235},
  series     = {London Math. Soc. Lecture Note Ser.},
  mrclass    = {18E35 (18E30)},
  mrnumber   = {2681709},
}

@article{lurie2017elliptic,
  title={Elliptic cohomology I: Spectral abelian varieties},
  author={Lurie, Jacob},
  year={2017}
}

\end{document}
